\documentclass[12pt, a4paper, reqno]{amsart}
\usepackage{amsmath,amssymb,amsfonts, amscd,  mathrsfs}
\usepackage{pxfonts, txfonts}
 \def\itbf{\itshape\bfseries }

\usepackage{latexsym}
\usepackage[applemac]{inputenc}
\usepackage{cite}
\usepackage{calc}                   

\numberwithin{equation}{section} 
\theoremstyle{plain}
\newtheorem{thm}{Theorem}[section]
\newtheorem{coro}[thm]{Corollary}
\newtheorem{lema}[thm]{Lemma}

\newtheorem{prop}[thm]{Proposition}

\newtheorem{rems}[thm]{Remarks}

\numberwithin{equation}{section}
\theoremstyle{plain}
\theoremstyle{definition}

\theoremstyle{remark}

\newcommand{\pFqcomma}{\mskip\pFqmuskip}

\input xy
\xyoption{all}
\newcommand{\cal}{\mathscr}
\usepackage{enumerate}
\newcommand{\ggeq}{>\!\!>}
\newcommand{\lleq}{<\!\!<}
\renewcommand{\Re}{\operatorname{Re}}
\renewcommand{\Im}{\operatorname{Im}}
\def\sg{\operatorname{sg}}
\usepackage{cases}
\def\supp{\operatorname{supp}}
\def\id{\operatorname{id}}

\usepackage[dvips,hyperindex]{hyperref}
\hypersetup{linkbordercolor=0.2 1 0.5}
\usepackage{color}

\def\1{\hbox{\bf 1}}

\newcommand\C{{\mathbb{C}}}

\newcommand\N{{\mathbb{N}}}

\newcommand\R{{\mathbb{R}}}

\newcommand\KK{{\mathbb{K}}}

\def\HHH{\mbox{${\mathcal H}$\kern-5.6pt${\mathcal H}$}}
\usepackage{amsmath}
\newmuskip\pFqmuskip

\newcommand*\pFq[6][8]{%
  \begingroup 
  \pFqmuskip=#1mu\relax
  \mathcode`\,=\string"8000
  \begingroup\lccode`\~=`\,
  \lowercase{\endgroup\let~}\pFqcomma
  {}_{#2}F_{#3}{\left[\genfrac..{0pt}{}{#4}{#5};#6\right]}%
  \endgroup
}

\input xy
\xyoption{all}
\usepackage{enumerate}
\usepackage{float}
\usepackage{graphicx}
\renewcommand{\mathcal}{\mathscr}
\newcommand{\bigsearrow}{\rotatebox[origin=b]{-45}{$\xrightarrow{\kern10mm}$}}
\newcommand{\bigswarrow}{\rotatebox[origin=b]{45}{$\xleftarrow{\kern10mm}$}}
\newcommand{\bignearrow}{\rotatebox[origin=b]{-315}{$\xrightarrow{\kern10mm}$}}
\begin{document}
\title{On a family of differential-reflection operators: 
intertwining operators, and  Fourier transform of rapidly decreasing functions}
\author[S. Ben Said, A. Boussen  and M. Sifi]
{Salem Ben Said, Asma Boussen \& Mohamed Sifi}
\address{S. Ben Said: Institut  \'Elie Cartan,
Universit\'e de Lorraine-Nancy, B.P. 239, F-54506
Vandoeuvres-L\`es-Nancy, France}
\email{salem.bensaid@univ-lorraine.fr}
\address{A. Boussen et M. Sifi: Universit\'e de Tunis El Manar,
Facult\'e des Sciences de Tunis, LR11ES11 Laboratoire d'Analyse
Math\'ematiques et Applications, 2092, Tunis, Tunisie}
\email{asma.boussen@live.fr, mohamed.sifi@fst.rnu.tn}

\begin{abstract}
We introduce a family of differential-reflection operators
$\Lambda_{A, \varepsilon}$ acting on smooth functions defined on
 $\R.$ Here $A$ is a    Strum-Liouville function with additional hypotheses and $\varepsilon\in \R.$ For special pairs $(A,\varepsilon),$ we recover Dunkl's, Heckman's and Cherednik's operators (in  one dimension). The spectral problem for the   operators $\Lambda_{A, \varepsilon}$ is studied. In particular, we obtain suitable growth estimates for the eigenfunctions of $\Lambda_{A, \varepsilon}$.

As the  operators $\Lambda_{A, \varepsilon}$ are mixture of $d/dx$
and reflection operators, we   prove the existence of an intertwining 
operator $V_{A,\varepsilon}$ between  $\Lambda_{A, \varepsilon}$ and the 
usual derivative. The   positivity of  $V_{A,\varepsilon}$ is also established. 

Via the eigenfunctions of $\Lambda_{A,\varepsilon},$ we introduce a generalized Fourier transform $\mathcal F_{A,\varepsilon}.$   An $L^p$-harmonic analysis for $\mathcal F_{A,\varepsilon}$ is 
developed when $0<p\leq {2\over{1+\sqrt{1-\varepsilon^2}}}$ and $-1\leq \varepsilon\leq 1.$ In particular, an $L^p$-Schwartz space isomorphism theorem for $\mathcal F_{A,\varepsilon}$ is proved. 

\end{abstract}
\maketitle

\section{Introduction}
Dunkl's ascertainment in the late eighties of the operators that now bear his name is one of the most significant  developments in the theory of special
 functions associated with root systems \cite{Dun}. Some early work in this area was done by Koornwinder \cite{Koo1}. 
A lot of the motivation for the subject comes from analysis on symmetric spaces. In the
one-variable cases, spherical functions on Riemannian symmetric spaces can be written as special functions depending on parameters which assume only special
 discrete values. The case of more general parameter values yields   special functions associated with root systems.

In \cite{Dun} Dunkl  generalized the operator $d/dx$ to a mixture of a differential and a reflection operators (in one dimension):
\index{$D_\alpha$lmain}
\begin{equation}\label{Dop}
D_\alpha f(x)=f'(x)+{{2\alpha+1}\over x} \left({{f(x)-f(-x)}\over 2}\right), \qquad \alpha >-1/2.
\end{equation}
By the specialization $\alpha={1\over 2}d-1$ with $d\in \N_{\geq 2},$ the   operator
$D_\alpha^2$ coincides on even functions with the radial part of the Laplace operator on the flat symmetric space $M(d)/SO(d),$ where $M(d)$ is the motion group of $\R^d.$ Important work in the analysis of Dunkl operators  has been done by several authors 
(see  \cite{DJO, DX, dJ, He1,  O1, R3, R31, roslervoit, xuu}; this list is far from being complete). 
 
Some years after, Heckman \cite{Heck1} wrote down a trigonometric variant of the Dunkl operators \eqref{Dop}
 (in one dimension): \index{$H_{\alpha,\beta}$}
$$
H_{\alpha, \beta} f(x)=f'(x) +\bigg\{(2\alpha+1)\coth x+(2\beta+1)\tanh x\bigg\} \left({{f(x)-f(-x)}\over 2}\right),
$$
 where $\alpha\geq \beta \geq -1/2$ and $\alpha\not =-1/2. $
Heckman's operators play a key role in proving the existence of the nowadays called Opdam's shift operators.
For $\alpha={1\over 2}(p-1) $ and $\beta ={1\over 2}(q-1)$ with  $ p\geq q>0,$ the  restriction of $H_{\alpha, \beta} ^2$ to  even functions coincides with the radial part of the Laplace-Beltrami operator on Riemannian symmetric spaces of the non-compact type and of real rank one. Significant  results  in the analysis of Heckman operators have been obtained  by several authors  (see for instance \cite{NSS,  Heck1, HeckOpd, Opd12, Heck11, BOS}).  
 
Next, in \cite{Cher1} Cherednik  made a slight but significant variation of Heckman's operator. He put (in one dimension)
\index{$\widetilde H_{\alpha,\beta}$}
\begin{equation}\label{Cherop}
\widetilde H_{\alpha, \beta} f(x)=f'(x) +\bigg\{(2\alpha+1)\coth x+(2\beta+1)\tanh x\bigg\} \left({{f(x)-f(-x)}\over 2}\right)-\varrho f(-x),
\end{equation}
 where $\alpha\geq \beta \geq -1/2,$   $\alpha\not =-1/2, $ and $\varrho=\alpha+\beta+1.$ 
It is known by now  that  harmonic analysis associated with  $\widetilde H_{\alpha, \beta}$  has a considerable technical difficulties to be overcome compare to harmonic analysis for Heckman's operator $ H_{\alpha, \beta}$ (see for instance \cite{Op, Cher2, AASS, Bch}).
 
The growing interest on these differential-reflection operators comes from their relevance for generalizing harmonic analysis on Riemannian symmetric spaces, and from their importance  for developing new topics in mathematical physics and probability (see for instance  \cite{GPS, roslervoit, GY, Bch, DS, BF}). 
 
In the present paper we consider  some aspects of harmonic analysis associated with the following    family of $(A,\varepsilon)$-operators  
$$
\Lambda_{A,\varepsilon} f(x)= f'(x)+{{A'(x)}\over{A(x)}} \left({{f(x)-f(-x)}\over 2}\right) -\varepsilon \varrho f(-x),
$$
where $A$ is   so-called a Ch\'ebli function on $\R$ (i.e.  $A$ is a continuous
 $\R^+$-valued function on $\R$ satisfying certain regularity and convexity hypotheses),
 $\varrho$ is the index of $A,$ and $\varepsilon \in \R.$ We note that $\varrho\geq 0.$ The function $A$ and the real number $\varepsilon$ are the deformations parameters
 giving back the above three cases (as special examples) when:
\begin{enumerate} 
\item $A(x)=A_\alpha(x)=|x|^{2\alpha+1}$ and $\varepsilon $ arbitrary (Dunkl's operators), \label{dunklcase}
\item $A(x)=A_{\alpha,\beta}(x)=|\sinh x|^{2\alpha+1} (\cosh x)^{2\beta+1}$ and $\varepsilon =0$ (Heckman's operators), \index{$A_{\alpha,\beta}(x)$lmain}
\item $A(x)=A_{\alpha,\beta}(x)=|\sinh x|^{2\alpha+1} (\cosh x)^{2\beta+1}$ and $\varepsilon =1$ (Cherednik's operators).
\end{enumerate}

This paper consists of three  parts. In the first part  we consider  the spectral problem for this  family of $(A,\varepsilon)$-operators.
More precisely, let $\lambda \in \C$ and consider the equation
\begin{equation}\label{syin}
\Lambda_{A,\varepsilon } f(x)=i\lambda f(x),
\end{equation}
where $f:\R\rightarrow \C.$
 We prove that there exists a unique solution  $\Psi_{A,\varepsilon}(\lambda,\cdot)$ of \eqref{syin}
satisfying $\Psi_{A,\varepsilon}(\lambda, 0)=1.$  
Moreover, under the assumption $-1\leq \varepsilon\leq 1,$ we establish in Theorems \ref{esti} and \ref{estder}  suitable estimates for the growth of the eigenfunction
$\Psi_{A,\varepsilon}(\lambda, x)$ and of its  partial derivatives. Our first step is Theorem \ref{thm-posi}, where we prove that $\Psi_{A,\varepsilon}(\lambda,\cdot)>0$ whenever $\lambda\in i\R.$ These estimates are the key tools for developing $L^p$-harmonic analysis associated with   the $(A, \varepsilon)$-operators (see Sections \ref{sec8} and \ref{sec9}). 

We note that $\Psi_{A,\varepsilon}$ reduces to the Dunkl kernel   in the $(A_\alpha, \varepsilon)$-case \cite{Dun2, Rosen};  to the Heckman kernel  in the $(A_{\alpha,\beta}, 0)$-case \cite{Heck11,  BOS}; and to the Cherednik kernel (or Opdma's kernel)  in the $(A_{\alpha,\beta}, 1)$-case \cite{Op, AASS}.

In the second part of this paper we study the existence and the positivity of an  intertwining operator between  $\Lambda_{A, \varepsilon}$ and  the ordinary derivative. We prove that there exists a unique isomorphism $V_{A,\varepsilon} : C^\infty(\R) \rightarrow C^\infty(\R)$  satisfying 
 $\Lambda_{A, \varepsilon}\circ V_{A,\varepsilon}=V_{A,\varepsilon}\circ {d\over {dx}},$  with $V_{A,\varepsilon} f(0)=f(0)$ (see  Theorem \ref{th63}). The construction of $V_{A,\varepsilon}$ involves Delsarte type operators \cite{Th, Li}.   
  
The intertwining operator  $V_{A,\varepsilon}$ plays a crucial  role for developing Fourier analysis associated with the $(A,\varepsilon)$-operators. In particular, it allows to write the eigenfunction $\Psi_{A,\varepsilon}$ as
\begin{equation}
\Psi_{A,\varepsilon}(\lambda, x)= V_{A,\varepsilon}(e^{i\lambda \,\cdot})(x),
\end{equation}
which gives a link between the Fourier transform with kernel  $\Psi_{A,\varepsilon}$  (say $\mathcal F_{A,\varepsilon}$) and the Euclidean Fourier transform. This alliance between $\mathcal F_{A,\varepsilon}$ and the Euclidean Fourier transform   will be a crucial trick to overcome difficulties in several places. 

Another important  result concerning the intertwining operator  $V_{A,\varepsilon}$  is that the latter is of positive type  in the sense that, if $f\geq 0$ then $V_{A,\varepsilon} f\geq 0$ 
(see Theorem \ref{cent}). The major technical step in  the proof of Theorem \ref{cent} is the positivity of $V_{A,\varepsilon}(h_t(u,\cdot))(x)$,  where $h_t(u,v)$ denotes the Euclidean heat kernel at time $t>0.$ For $\varepsilon=0$ and $1,$ this result can be found in \cite{tripo1} and \cite{tripo2}.
We pin down that the positivity of $V_{A,\varepsilon}$ played a fundamental  role in \cite{BBS2} in  establishing an analogue of Beurling's theorem, and its relatives such as theorems of type Gelfand-Shilov, Morgan's, Hardy's, and Cowling-Price in the setting of this paper.  

In the particular case  where $A=A_\alpha$ (see \eqref{dunklcase}), the intertwining operator $V_{A, \varepsilon}$ reduces to the Dunkl intertwining operator  in one dimension (see for instance \cite{R3, R31}).

The third part of this paper is concerned with a development of the $L^p$-harmonic analysis for a Fourier transform $\cal F_{A,\varepsilon}$ 
when $0<p\leq {2\over {1+  \sqrt{1-\varepsilon^2}}}$ and   $-1\leq \varepsilon\leq 1.$ Here $$\cal F_{A,\varepsilon}f(\lambda)=\int_\R f(x) \Psi_{A,\varepsilon}(\lambda, -x) A(x)dx$$ for $f\in L^1(\R, A(x)dx).$

 Using the estimates for the growth of  $\Psi_{A,\varepsilon}(\lambda, x)$   we get  the  holomorphic properties of $\cal F_{A, \varepsilon}$ on $L^p(\R, A(x)dx).$ A Riemann-Lebesgue lemma is also obtained for $1\leq p<{2\over {1+  \sqrt{1-\varepsilon^2}}}.$ 

 We then turn our attention to an $L^p$-Schwartz space  isomorphism theorem for $\cal F_{A, \varepsilon}.$   In \cite{Ha} Harish-Chandra proved an $L^2$-Schwartz space isomorphism for the spherical Fourier transform on non-compact Riemannian symmetric spaces. This result was extended to $L^p$-Schwartz spaces with $0<p<2$ by Trombli and Varadrajan \cite{TV} (see also \cite{EM, FJ, clerc}). In the early nighties, Anker gave a new and simple proof of their result, based on the Paley-Wiener theorem for the spherical Fourier transform on Riemannian symmetric spaces \cite{Anker}.  Recently, Anker's method was used in \cite{PPN} to prove an $L^p$-Schwartz space  isomorphism theorem for the Heckman-Opdam hypergeometric functions. Our Approach is inspired from Anker's paper [loc. cit.]. More precisely,  for  $-1\leq \varepsilon\leq 1$ and  $0<p\leq {2\over {1+  \sqrt{1-\varepsilon^2}}},$ put
\begin{equation}\label{tube}
\C_{p,\varepsilon}:=\left\{ \lambda\in \C\;|\; |\Im \lambda| \leq \varrho \left((2/ p) -1 -\sqrt{1-\varepsilon^2}\right)  \right\}.
\end{equation}
Denote by  $\mathcal S_p(\R)$  the $L^p$-Schwartz space on $\R,$ and by $\mathcal S(\C_{p,  \varepsilon}) $   the Schwartz space on the tube domain  $\C_{p,\varepsilon}.$ 
We prove  that  $\mathcal F_{A,\varepsilon}$ is a topological isomorphism between $\mathcal S_p(\R)$ and
 $ \mathcal S(\C_{p,\varepsilon})$ (see Theorem \ref{+-1}).

We close the third part of this paper by establishing a result in connection with  pointwise multipliers of $\mathcal S(\C_{p,\varepsilon}).$ More precisely, for arbitrary $\alpha\geq 0,$ a function $\psi$ defined on  the tube domain $\C_\alpha:=\big\{ \lambda\in \C\;|\; |\Im\lambda| \leq \alpha\big\}$
  is called a pointwise multiplier of $\cal S(\C_\alpha)$ if the mapping $\phi \mapsto \psi \phi$ is continuous from $\cal S(\C_\alpha)$  into itself. In \cite{BBM} Betancor {\it et al.} characterize  the set of pointwise multipliers of the Schwartz spaces $\cal S(\C_\alpha).$  

Under the assumptions $0<p \leq {2\over {1+  \sqrt{1-\varepsilon^2}}}$  whenever $\varrho=0,$ and $ {2\over {2+  \sqrt{1-\varepsilon^2}}}\leq p \leq {2\over {1+  \sqrt{1-\varepsilon^2}}}$  whenever $\varrho>0,$ we prove that if $T$ is in   the dual space $\mathcal S_p'(\R)$ of  $\mathcal S_p(\R)$ such that $\psi:=\mathcal F_{A,\varepsilon}(T)$  is 
a pointwise multiplier of  $\mathcal S(\C_{p,\varepsilon}),$ then for any $s\in \N$  
there exist $\ell\in \N$ and continuous functions $f_m$ defined on $\R,$ $m=0,1,\ldots  ,\ell ,$  such that 
$$T=\sum_{m=0}^\ell \Lambda_{A,\varepsilon}^m \,f_m  $$ and, for every such $m,$ 
\begin{equation}\label{modi}
\sup_{x\in \R} \,(|x|+1)^s \,e^{({2\over p}-\sqrt{1-\varepsilon^2})\, \varrho|x|} \,| f_m (x)|<\infty.
\end{equation}

The organization of this paper is as follows: In Section 2 we recapitulate  some definitions and basic notations, as well as some  results from literature. In Sections 3 and 4  we study the main  properties of the eigenfunction $\Psi_{A,\varepsilon}.$ In particular, we obtain estimates for the growth of  $\Psi_{A,\varepsilon}$  and  of its partial derivatives. A Laplace type representation of the eigenfunction $\Psi_{A,\varepsilon}$  is derived  in Section 5.   
Sections 6 and 7 are devoted to the existence and to the positivity of the intertwining operator 
$V_{A,\varepsilon}$ between  $\Lambda_{A, \varepsilon}$ and  the ordinary derivative. In Section 8 we develop the $L^p$-harmonic analysis for the Fourier  transform $\mathcal F_{A,\varepsilon},$ where  we mainly prove an $L^p$-Schwartz space  isomorphism theorem for $\cal F_{A, \varepsilon}.$ Finally, in Section 9 we characterize the distributions $T\in \mathcal S_p'(\R)$ so that $\mathcal F_{A,\varepsilon}(T)$ is a pointwise multiplier   of the Schwartz space  $ \mathcal S(\C_{p,\varepsilon}).$

\section{Background}\label{sec2}
In this introductory section we present results from \cite{C2, C3, C1, Tri, Tribook}. See also \cite{BH, BX2, BX1}. 

Throughout this paper we will denote by   $A$ \index{$A$lmain}  a function on $\R$  satisfying  the following hypotheses:
\begin{enumerate}
\item[{(H1)}]  $A(x)=|x|^{2\alpha+1} B(x),$ where $\alpha>-{1\over 2}$ and $B$ is any even,  positive and smooth function 
on $\R$ with   $B(0)=1.$
\item[{(H2)}] $A$ is increasing and unbounded on $\R_+.$
\item[{(H3)}]   ${{A'}/ {A}}$ is a decreasing and smooth function on $\R^*_+,$ and  hence  the limit $2\varrho:= \lim\limits_{x\rightarrow +\infty}{{A'(x)}/{A(x)}} \geq 0$ exists. 
\end{enumerate}
Such a function $A$ is called a Ch\'ebli function. From (H1) it follows that 
\begin{equation}\label{hnew}
{{A'(x)}\over{A(x)}}={{2\alpha+1}\over x} +C(x),\qquad x\not=0,
\end{equation} where $C:={{B'}/{B}}$ is an odd and smooth function on $\R.$   

Let $ \Delta_A,$ or simply $\Delta  ,$ be the following second order differential operator  \index{$\Delta$}
\begin{equation}\label{L}
\Delta={{d^2}\over {dx^2}}+ {{A'(x)}\over{A(x)}}  {d\over {dx}} .
\end{equation}
 For  $\mu\in \C,$ we consider the Cauchy problem 
\begin{equation}\label{CP}
 \begin{cases}
&\!\!\!\!\!\Delta f(x)=-(\mu^2+\varrho^2) f(x)\\
&\!\!\!\!\!f(0)=1,\quad f'(0)=0.
\end{cases}
\end{equation}
In \cite{C1} the author  proved that the system  \eqref{CP} admits  a unique solution 
$\varphi_\mu.$  For every $\mu\in \C$, the solution $\varphi_\mu$ is an even smooth function  on $\R$ 
and the map $\mu\mapsto \varphi_\mu(x)$ is  analytic.  The following Laplace type representation of $\varphi_\mu$  
can be found in \cite{C1} (see also \cite{Tri}).
\begin{lema} \label{Ch1}  For every $x\in \R^* $ there exists a  probability measure $\nu_x$ on $\R$ supported in 
$[-|x|,|x|]$ such that for all $\mu \in \C$
$$\varphi_\mu(x)=\int_{-|x|}^{|x|} e^{(i\mu-\varrho)t} \nu_x(dt).$$
Also, for $x\in \R^*,$ there is a non-negative even continuous function $K (|x|,\cdot)$
supported in $[-|x|, |x|]$  such that for all  $\mu\in \C$  \index{$K (|x|,s)$}
\begin{equation}\label{K1}
\varphi_{\mu } (x) =\int_{0}^{|x|} K (|x|,t) \cos(\mu t) dt.
\end{equation} 
\end{lema}

The following  estimates of  the eigenfunctions $\varphi_\mu$  can be found in  \cite{C2, C1,  Tribook,  BX1}.
\begin{lema}\label{BX} Let $\mu\in \C$ such that $| \Im \mu |\leq \varrho.$   Then  
\begin{enumerate} [\upshape 1)]
\item $\varphi_{\pm i\varrho}(x) =1.$ 
\item $\varphi_{-\mu}(x)=\varphi_\mu(x).$
\item   $|\varphi_\mu(x)|\leq 1.$  
\item    $|\varphi_\mu(x)|\leq \varphi_{i \Im \mu}(x) \leq e^{| \Im \mu | |x|} \varphi_0(x).$ 
\item    $|\varphi_\mu'(x)|\leq c\, (\varrho^2+\vert \mu\vert^2) e^{| \Im \mu | |x|} \varphi_0(x).$ 
\item    $ e^{-\varrho |x|} \leq \varphi_0(x) \leq c (|x|+1)e^{-\varrho |x|} .$
\end{enumerate}
\end{lema}

The Ch\'ebli   transform of $f\in L^1(\R_+, A(x)dx)$ is given by \index{$\mathcal F_{\Delta}(f)$}
\begin{equation}\label{Fdelta}
 \mathcal F_{\Delta}(f)(\mu):=\int_{\R_+} f(x)\varphi_{\mu } (x) A(x) dx. 
 \end{equation}

The following Plancherel and inversion formulas for $\mathcal F_{\Delta}$ are proved in \cite{C1}.
 \begin{thm}  There exists a unique positive   measure $\pi $ with support $\R_+$ such that 
 ${\mathcal F}_{\Delta}$ induces 
 an isometric isomorphism from $L^2(\R_+, A(x)dx)$ onto $L^2(\R_+, \pi(d\mu)),$ 
 and for any $f\in L^1(\R_+, A(x)dx) \cap L^2(\R_+, A(x)dx)$ we have 
$$ \int_{\R_+} |f(x)|^2 A(x) dx=\int_{\R_+} | \mathcal F_{\Delta}(f)(\mu)|^2 \:\pi(d\mu).$$ 
 The inverse transform is given by 
\begin{equation}\label{inv1}
\mathcal F_{\Delta}^{-1}g(x)=\int_{\R_+} g(\mu) \varphi_{\mu } (x) \:\pi(d\mu). 
\end{equation}
 \end{thm}
 
To have a nice behavior for the Plancherel measure $\pi$ we must add 
a further (growth) restriction on the function $A.$  Following \cite{Tri},  we will assume that 
$A'/A$ satisfies the following additional hypothesis: \index{$({\rm H}4)$}
\begin{itemize}
\item[{(H4)}] There exists a constant $\delta>0$ such that for all $x\in [x_0,\infty)$ (for some $x_0> 0$),

\begin{equation}
\displaystyle  {{A'(x)}\over {A(x)}}=
\begin{cases} \;
 \displaystyle  2\varrho+e^{-\delta x} D(x)  &  \text{  if  }\; \varrho >0, \vspace{.2cm}\\
 
\; \displaystyle   {{2\alpha+1}\over x}+e^{-\delta x} D(x)   & \text{  if  } \; \varrho =0, 
  \end{cases}
\end{equation}
with $D$ being  a smooth function bounded together with its derivatives.
  \end{itemize}

In these circumstances  the Plancherel measure $\pi$ is absolutely continuous 
with respect to the Lebesgue measure and has density 
$|c(\mu)|^{-2}$ \index{$|c(\mu)|^{-2}$lmain} where  $c $ is continuous function on $\R_+$ and zero free 
on $\R^*_+$ (see \cite{BX2}). Moreover, by \cite[Proposition 6.1.12 and Corollary 6.1.5]{Tribook} (see also \cite{Bra1}), for  $\mu\in \C$ we have
\begin{enumerate} 
\item [{(i)}] If $\varrho \geq 0$ and $\alpha>-1/2,$ then $|c(\mu)|^{-2}\sim |\mu|^{2\alpha+1}$ whenever $|\mu|\ggeq 1.$
\item[{(ii)}] If $\varrho > 0$ and $\alpha>-1/2,$ then $|c(\mu)|^{-2}\sim |\mu|^{2 }$ whenever $|\mu|\lleq 1.$
\item[{(iii)}] If $\varrho= 0$ and $\alpha>0,$ then $|c(\mu)|^{-2}\sim |\mu|^{2\alpha+1}$ whenever $|\mu|\lleq 1.$
\end{enumerate}
In the literature, the function $c$ is called Harish-Chandra's function 
of the operator $\Delta .$ We refer to \cite{Bra2} for more details on the $c$-function. 

Henceforth we will assume that  Ch\'ebli's  function $A$ satisfies the additional  hypothesis  (H4). It follows that  for $|x|$ is large enough:
\begin{itemize}
\item[{(i)}]  $A(x)=O(e^{2\varrho |x|})$  for  $\varrho>0$. 
\item[{(ii)}]  $A(x)=O( |x|^{2\alpha+1})$ for $\varrho=0.$
\end{itemize}

We close this section by giving some basic results of  (the analogue of)  the Abel transform associated 
with the second order differential operator $\Delta .$

Denote by $\mathcal D_e(\R)$ the space of even and compactly supported  functions in $C^\infty(\R).$ 
In \cite{Tri} the author has proved that the Abel transform defined on $\mathcal D_e(\R)$ by 
\index{$\mathcal A$lmain}
\begin{equation}\label{abel}
\mathcal A f(y)={1\over 2}\int_{|x|>|y|} K(|x|,y) f(x) A(x)dx
\end{equation}
 is an automorphism  of $\mathcal D_e(\R)$  and satisfying 
\begin{equation}\label{1.10}\mathcal A \circ (\Delta +\varrho^2)=\displaystyle {{d^2}\over{dx^2}}\circ  \mathcal A .\end{equation}
  Furthermore, on $\mathcal D_e(\R),$ we have
\begin{equation}\mathcal F_{\Delta} =\mathcal F_{\rm euc}  \circ \mathcal A ,\end{equation} where $\mathcal F_{\rm euc}$ is the Euclidean Fourier transform. \index{$\mathcal F_{\rm{euc}}(f)$}
 
\section{A family of differential-reflection operators}
For $\varepsilon\in \R$ we consider the following differential-reflection operators \index{$\Lambda_{A, \varepsilon}$}
\begin{equation}\label{Op} 
\Lambda_{A,\varepsilon} f(x)= f'(x)+{{A'(x)}\over{A(x)}} \left({{f(x)-f(-x)}\over 2}\right) -\varepsilon \varrho f(-x).
\end{equation}
In view of \eqref{hnew} and the hypothesis (H4) on $A'/A,$ the  space $\mathcal D(\R)$ \index{$\mathcal D(\R)$lmain}
(of   smooth  functions with compact support on $\R$) and the  space $\mathcal S(\R)$ (of  Schwartz functions on $\R$)
\index{$\mathcal S (\R)$} are invariant under the action of $ \Lambda_{A,\varepsilon}.$
 
 Let $S$ denote the symmetry \index{$S$}   $(S f)(x):= f(-x).$  The following lemma is needed 
 later. The easy proof is left to the reader.
\begin{lema}\label{lema1} 
Let $f\in C^\infty(\R)$ such that $ \sup_{x\in \R} (1+|x|)^r e^{s|x|} |f^{(t)}(x)|<\infty$  for every $r, t\in \N$  and for some $2\varrho \leq s<\infty,$
  and let $g\in C^\infty(\R)$  such that   $g$ and all its derivatives are 
at most of polynomial growth. Then 
$$\int_{\R} \Lambda_{A,\varepsilon} f(x) g(x) A(x) dx
=-\int_{\R}  f(x) (\Lambda_{A,\varepsilon} +2\varepsilon \varrho S) g(x) A(x)dx.$$
\end{lema}

Let $\lambda\in \C$   and consider the initial data problem  
\begin{equation}\label{system}
\Lambda_{A,\varepsilon} f(x)=i\lambda f(x), \qquad f(0)=1,
\end{equation}
where $f:\R\rightarrow \C.$ We  have the following  statement. 
\begin{thm} \label{thm1} Let  $\lambda\in \C.$ There exists a unique solution  
$\Psi_{A, \varepsilon}(\lambda,\cdot)$ to the problem   \eqref{system}. 
 Further, for every $x\in \R,$ the function 
 $\lambda\mapsto \Psi_{A, \varepsilon}(\lambda,x)$ is analytic on $\C.$
More explicitly: \index{$\Psi _{ \varepsilon } (\lambda,x)$}
 \begin{itemize}
\item[{(i)}] For $i\lambda\not =\varepsilon \varrho,$  
\begin{equation}\label{2.2}
\Psi_{A, \varepsilon}(\lambda,x)=
\varphi_{\mu_\varepsilon}(x)+{1\over{ i\lambda-\varepsilon \varrho}}\varphi_{\mu_\varepsilon}'(x),
\end{equation}
where \index{$\mu_\varepsilon$}
\begin{equation}\mu_\varepsilon^2:=\lambda^2+(\varepsilon^2-1)\varrho^2.\end{equation}
We may rewrite the solution \eqref{2.2}  as
\begin{equation}\label{2.4}
\Psi_{A, \varepsilon}(\lambda,x)=\varphi_{\mu_\varepsilon}(x)+ (i\lambda+\varepsilon \varrho) {{\sg(x)}\over {A(x)}} \int_0^{|x|} \varphi_{\mu_\varepsilon}(t) A(t)dt.
\end{equation}
\item[{(ii)}] For $i\lambda=\varepsilon\varrho,$  
\begin{equation}\label{=0}
\Psi_{A, \varepsilon}(\lambda,x)=1+2\varepsilon \varrho {{\sg(x) }\over{A(x)}}\int_0^{|x|} A(t)dt.
\end{equation}
\end{itemize}
\end{thm}
\begin{proof} Assume first  that $i\lambda\not =\varepsilon\varrho.$  After the formula \eqref{2.4} is established the restriction on $\lambda$ can be dropped by analytic continuation.  Write $f$ as the superposition  $f=f_e+f_o$
of an  even function $f_e$ and an odd  function $f_o.$  
Then, the problem    \eqref{system} is equivalent to the following   system: 
\begin{subnumcases}{}
&$\displaystyle  \!\!\! \!\!\! \!\!\! \!\!\! f_o'(x)+{{A'(x)}\over {A(x)}} f_o(x)   =(i\lambda+\varepsilon \varrho ) f_e(x), $  \\
  & $\displaystyle\!\!\! \!\!\! \!\!\! \!\!\!  f'_e(x)   =(i\lambda -\varepsilon \varrho) f_o(x),$ \label{3b} \\
  & $\displaystyle\!\!\! \!\!\! \!\!\! \!\!\! f_e(0)=1,\; f_o(0)=0.$
  \end{subnumcases}
Combining  the two equations above yields $$f_e''(x)+{{A'(x)}\over {A(x)}}  f_e'(x)=-(\lambda^2+\varepsilon^2\varrho^2) f_e(x).$$
That is  $$\Delta f_e(x)=-\big(\underbrace{\lambda^2+(\varepsilon^2-1)\varrho^2}_{:=\mu_\varepsilon^2}+\varrho^2\big)f_e(x).$$
Since $f_e(0)=1,$   the uniqueness of the solution to the Cauchy problem \eqref{CP} gives 
 $$f_e(x)=\varphi_{\mu_\varepsilon}(x),$$
 which, in part, explains the uniqueness of the desired solution.   Now, from \eqref{3b} we obtain 
$$
f_o(x)={1\over {i\lambda-\varepsilon \varrho}} \varphi_{\mu_\varepsilon}'(x). 
$$
Consequently 
\begin{equation}\label{so}
\Psi_{A, \varepsilon}(\lambda, x)=\varphi_{\mu_\varepsilon}(x)+{1\over {i\lambda-\varepsilon \varrho}}\varphi_{\mu_\varepsilon}'(x).
\end{equation}
Further we have 
\begin{equation}\label{Lw}
 \varphi_{\mu_\varepsilon}'(x) =-({\mu_\varepsilon}^2+\varrho^2) {{\sg(x)}\over {A(x)}} \int_0^{|x|} \varphi_{\mu_\varepsilon}(t) \,A(t)dt,
 \end{equation}
 which is a consequence of  the following   known formula  for even functions 
\begin{equation}\label{Lw}
g'(x)= {{\sg(x)}\over {A(x)}} \int_0^{|x|} \Delta\, g(t) \,A(t)dt
 \end{equation}
 and the fact that $\varphi_\mu$ satisfies \eqref{CP}. Hence,  we can rewrite the solution   \eqref{so} as 
\begin{equation}\label{2.7}
\Psi_{A, \varepsilon}(\lambda, x)=\varphi_{\mu_\varepsilon}(x)+(i\lambda+\varepsilon \varrho)  {{\sg(x)}\over {A(x)}} \int_0^{|x|} \varphi_{\mu_\varepsilon}(t) A(t)dt.
\end{equation} 

Since the function $\mu\mapsto \varphi_\mu(x)$ is holomorphic for all $\mu\in \C,$ it follows from \eqref{2.7}
 that for every $x\in \R,$
 the map $\lambda \mapsto \Psi_{A, \varepsilon}(\lambda, x)$ is analytic on $\C.$ 
\end{proof}

\section{Growth of the eigenfunctions}

The eigenfunction $\Psi_{A, \varepsilon}$ is of particular interest as it gives rise to an associated integral transform on $\R$ which generalizes the Euclidean Fourier transform in a natural way (see Section \ref{sec8}). Its definition and essential properties rely    on suitable growth estimates of $\Psi_{A, \varepsilon}$. The following positivity result is the basic ingredient  in obtaining these estimates. 
\begin{thm}\label{thm-posi} Assume that $-1\leq \varepsilon\leq 1.$ For  all $\lambda\in i\R,$ the  function $\Psi_{A, \varepsilon}(\lambda, \cdot)$ is real and strictly positive.
\end{thm}
\begin{proof} If we take complex conjugates in \eqref{system}, we see that 
$\overline{\Psi_{A, \varepsilon}(\lambda, \cdot)}$ and $\Psi_{A, \varepsilon}(\lambda, \cdot)$
 satisfy the same system \eqref{system}. Since 
$\overline{\Psi_{A, \varepsilon}(\lambda, 0)}=1,$ the uniqueness part in Theorem \ref{thm1}  shows that 
$\overline{\Psi_{A, \varepsilon}(\lambda, x)}=\Psi_{A, \varepsilon}(\lambda, x)$ for all $x\in \R.$

Assume that $\Psi_{A, \varepsilon}(\lambda,\cdot)$ is not strictly positive. Since $\Psi_{A, \varepsilon}(\lambda, 0)=1>0,$ 
it follows that 
$\Psi_{A, \varepsilon}(\lambda,\cdot)$ vanishes. 
Let $x_0$ be a zero of $\Psi_{A, \varepsilon}(\lambda,\cdot)$ so that 
$$|x_0|={\rm inf} \left\{ |x|\;:\;  \Psi_{A, \varepsilon}(\lambda,x)=0 \right\}. $$
Since $\Psi_{A, \varepsilon}(\lambda,0)=1$ we have $\Psi_{A, \varepsilon}(\lambda,x) \geq 0$
 on $[-|x_0|, |x_0|].$ In particular $\Psi_{A, \varepsilon}(\lambda,-x_0)\geq 0.$
We claim that
\begin{equation}
\begin{cases}\label{claim1}
  \Psi_{A, \varepsilon}'(\lambda,x_0)=0,\\
 \Psi_{A, \varepsilon}(\lambda,-x_0)=0.
\end{cases}
\end{equation}
To prove \eqref{claim1}, let us first assume that $x_0> 0.$ Then   $\Psi_{A, \varepsilon}'(\lambda,x_0)\leq 0.$
Moreover, 
\begin{subequations}\label{diff}
\begin{align}
\Psi_{A, \varepsilon}'(\lambda,x)&= -{{A'(x)}\over{2A(x)}} \Big( \Psi_{A, \varepsilon}(\lambda,x)-\Psi_{A, \varepsilon}(\lambda,-x)\Big) +\varepsilon \varrho \Psi_{A, \varepsilon}(\lambda,-x)+i\lambda \Psi_{A, \varepsilon}(\lambda,x) \\
&=  \Big({{A'(x)}\over{2A(x)}} +\varepsilon \varrho\Big)  \Big( \Psi_{A, \varepsilon}(\lambda,-x)-\Psi_{A, \varepsilon}(\lambda,x)\Big)+(i\lambda+\varepsilon \varrho)\Psi_{A, \varepsilon}(\lambda,x) . \label{diff}
\end{align}
\end{subequations}
Thus 
 \begin{equation}\label{posi}
 \Psi_{A, \varepsilon}'(\lambda,x_0)=\Big({{A'(x_0)}\over{2A(x_0)}} +\varepsilon \varrho\Big)   \Psi_{A, \varepsilon}(\lambda,-x_0) .
 \end{equation}
From \eqref{posi} it follows that  $ \Psi_{A, \varepsilon}'(\lambda,x_0)$  is positive. This is   due to the fact that $\varepsilon  \geq -1$ and the fact that $A'/(2A)$ is a decreasing function on $\R^*_+$ and 
 $\lim_{x\rightarrow +\infty} A'(x)/2A(x)=\varrho.$
We deduce that  $\Psi_{A, \varepsilon}'(\lambda,x_0)=0,$ and  therefore, from \eqref{posi},   $\Psi_{A, \varepsilon}(\lambda,-x_0)=0.$ 

Now, let us assume that  $x_0< 0.$ Then  $\Psi_{A, \varepsilon}'(\lambda,x_0)\geq 0.$ Moreover, for $x_0<0,$ equation \eqref{posi} implies 
$\Psi_{A, \varepsilon}'(\lambda,x_0)\leq 0.$ This   is due to $\varepsilon \leq 1$ and to  
assumptions on $A'/(2A). $ Then, as above, we conclude that  $\Psi_{A, \varepsilon}'(\lambda,x_0)=0,$ and once again appealing to 
\eqref{posi} we have  $\Psi_{A, \varepsilon}(\lambda,-x_0)=0.$ 
This finishes the proof of the claim \eqref{claim1}.

Starting this time from  $\Psi_{A, \varepsilon}(\lambda,-x_0)=0$  and proceeding analogously as in  the case   
$\Psi_{A, \varepsilon}(\lambda,x_0)=0,$   
we conclude that 
\begin{equation*}
\begin{cases}
 \Psi_{A, \varepsilon}'(\lambda,-x_0)=0,\\
  \Psi_{A, \varepsilon}(\lambda,x_0)=0.
  \end{cases}
\end{equation*}
 In summary,  $\Psi_{A, \varepsilon}(\lambda,\pm x_0)=0$ 
and $\Psi_{A, \varepsilon}'(\lambda, \pm x_0)=0.$   Differentiating \eqref{diff}, we see that the second derivative of 
$\Psi_ {A, \varepsilon}(\lambda,\cdot)$ vanishes at $\pm x_0.$ 
Repeating the same  argument over and over again to get $ \Psi_{A, \varepsilon}^{(k)}(\lambda, \pm x_0)=0$ for all $k\in \N.$ 
Since $\Psi_{A, \varepsilon}(\lambda,\cdot)$ is a real analytic  function, we deduce that   $\Psi_{A, \varepsilon}(\lambda,x)= 0 $ for all $x\in \R.$ 
This contradicts  $\Psi_ \varepsilon(\lambda,0)=1.$ Thus, either $\Psi_{A, \varepsilon}(\lambda,x)$ is  strictly positive for all $x,$ 
or it is strictly negative for all $x.$ But since $\Psi_{A, \varepsilon}(\lambda,0)=1,$ it must be  $\Psi_{A, \varepsilon}(\lambda,x)>0$ for all $x\in \R.$
\end{proof}

The following theorem contains  important estimates for the growth  of  the eigenfunction $\Psi_{A, \varepsilon}.$ 
\begin{thm}\label{esti} Suppose  that $-1\leq \varepsilon\leq 1$ and  $x\in \R.$ Then:   
\begin{enumerate}[\upshape 1)]
\item For real $\lambda$  we have $ |\Psi_{A, \varepsilon}(\lambda,x)|  \leq \sqrt 2.$  
\item  For  $\lambda=a+ib\in \C $ we have   $ |\Psi _{A, \varepsilon}(\lambda,x)| \leq \Psi_{A, \varepsilon}(i b,x).$ 
\item  For  $\lambda=ib \in i\R$ we have $\Psi_{A, \varepsilon}(ib ,x)\leq \Psi_{A, \varepsilon}(0,x) \;e^{|b |\, |x|}.$ 
\item  For $\lambda=0$ we distinguish the following two cases:
\begin{enumerate}[\upshape a)]
\item For   $\varepsilon =0,$ we have $\Psi_{A,0}(0,x)=1 .$
\item   For   $\varepsilon\not=0,$ there is a  constant $c_\varepsilon>0$  such that  $\Psi_{A, \varepsilon}(0,x)\leq c_\varepsilon (|x|+1) e^{-\varrho (1- \sqrt{1-\varepsilon^2 } )|x|}   .$ 
\end{enumerate}
\end{enumerate}
\end{thm}
\begin{proof}  1) Assume that $\lambda\in \R.$ Since $\Psi_{A, \varepsilon}(\lambda, x)$ is a solution of 
 the problem   \eqref{system}, we deduce that 
\begin{equation}\label{di-di}
 \Psi_{A, \varepsilon}'(\lambda,x)=-{{A'(x)}\over{2A(x)}} \Big( \Psi_{A, \varepsilon}(\lambda,x)-\Psi_{A, \varepsilon}(\lambda,-x)\Big) +\varepsilon \varrho \Psi_{A, \varepsilon}(\lambda,-x)+i\lambda \Psi_{A, \varepsilon}(\lambda,x). 
 \end{equation}
Thus   $\Psi_{A, \varepsilon}(\lambda,-x)$ satisfies  the following  equation
\begin{equation}\label{exp}
  \big\{\Psi_{A, \varepsilon}(\lambda,-x) \big\}'={{A'(-x)}\over{2A(x)}} \Big( \Psi_{A, \varepsilon}(\lambda ,-x)-\Psi_{A, \varepsilon}(\lambda,x)\Big) -\varepsilon \varrho \Psi_{A, \varepsilon}(\lambda,x)-i\lambda \Psi_{A, \varepsilon}(\lambda,-x).
\end{equation}
If we take complex conjugates in \eqref{exp}, we obtain
$$
\overline{  \big\{\Psi_{A, \varepsilon}(\lambda,-x) \big\}' }={{A'(-x)}\over{2A(x)}} \Big( \overline{\Psi_{A, \varepsilon}(\lambda ,-x)}-\overline{\Psi_{A, \varepsilon}(\lambda,x)}\Big) -\varepsilon \varrho \overline{\Psi_{A, \varepsilon}(\lambda,x)}+i\lambda \overline{\Psi_{A, \varepsilon}(\lambda,-x).}
$$
Hence
 \begin{eqnarray*}
 \big\{ |\Psi_{A, \varepsilon}(\lambda,-x)|^2\big\}' &=&   \: \big\{\Psi_{A, \varepsilon}(\lambda,-x) \big\}' \;\overline{\Psi_{A, \varepsilon}(\lambda,-x)}+
 \overline{\big\{\Psi_{A, \varepsilon}(\lambda,-x) \big\}'} \; {\Psi_{A, \varepsilon}(\lambda,-x)}\\
 &=& \: {{A'(-x)}\over{2A(x)}} \Big( \Psi_{A, \varepsilon}(\lambda ,-x)-\Psi_{A, \varepsilon}(\lambda,x)\Big)\overline{\Psi_{A, \varepsilon}(\lambda,-x)} -\varepsilon \varrho \Psi_{A, \varepsilon}(\lambda,x) \overline{\Psi_{A, \varepsilon}(\lambda,-x)}\\
 && +{{A'(-x)}\over{2A(x)}} \Big( \overline{\Psi_{A, \varepsilon}(\lambda ,-x)}-\overline{\Psi_{A, \varepsilon}(\lambda,x)}\Big) {\Psi_{A, \varepsilon}(\lambda,-x)} -\varepsilon \varrho \overline{\Psi_{A, \varepsilon}(\lambda,x)} {\Psi_{A, \varepsilon}(\lambda,-x)}.
\end{eqnarray*}
Similarly we have 
 \begin{eqnarray*}
 \big\{ |\Psi_{A, \varepsilon}(\lambda,x)|^2\big\}' &=&  \big\{\Psi_{A, \varepsilon}(\lambda,x) \big\}' \;\overline{\Psi_{A, \varepsilon}(\lambda,x)}+
 \overline{\big\{\Psi_{A, \varepsilon}(\lambda,x) \big\}'} \; {\Psi_{A, \varepsilon}(\lambda,x)}\\
 &= & -{{A'(x)}\over{2A(x)}} \Big( \Psi_{A, \varepsilon}(\lambda ,x)-\Psi_{A, \varepsilon}(\lambda,-x)\Big)\overline{\Psi_{A, \varepsilon}(\lambda,x)} +\varepsilon \varrho \Psi_{A, \varepsilon}(\lambda,-x) \overline{\Psi_{A, \varepsilon}(\lambda,x)}\\
&&   -{{A'(x)}\over{2A(x)}} \Big( \overline{\Psi_{A, \varepsilon}(\lambda ,x)}-\overline{\Psi_{A, \varepsilon}(\lambda,-x)}\Big) {\Psi_{A, \varepsilon}(\lambda,x)} +\varepsilon \varrho \overline{\Psi_{A, \varepsilon}(\lambda,-x)} {\Psi_{A, \varepsilon}(\lambda,x)}.
\end{eqnarray*}
Using the fact that  $A'$ is  an odd  function  we obtain 
\begin{align*}
 &\big\{ |\Psi_{A, \varepsilon}(\lambda,-x)|^2\big\}'+ \big\{ |\Psi_{A, \varepsilon}(\lambda,x)|^2\big\}'\\
 =  &\, \Big(\Psi_{A, \varepsilon}(\lambda,-x)-\Psi_{A, \varepsilon}(\lambda,x)\Big)  
\Big(-{{A'(x)}\over{2A(x)}}  \overline{\Psi_{A, \varepsilon}(\lambda,-x)} +{{A'(x)}\over{2A(x)}} \overline{\Psi_{A, \varepsilon}(\lambda,x)}\Big)\hfill\\ 
  & +\Big( \overline{ \Psi_{A, \varepsilon}(\lambda,-x)}- \overline{\Psi_{A, \varepsilon}(\lambda,x)}\Big)  
\Big(-{{A'(x)}\over{2A(x)}}   {\Psi_{A, \varepsilon}(\lambda,-x)} +{{A'(x)}\over{2A(x)}}  {\Psi_{A, \varepsilon}(\lambda,x)}\Big)\hfill\\ 
  = &   -{{A'(x)}\over{A(x)}} \;\left | \Psi_{A, \varepsilon}(\lambda,-x)-\Psi_{A, \varepsilon}(\lambda,x)\right|^2.\hfill
\end{align*}
Since ${{A'(x)}/ {A(x)}} \geq 2\varrho\geq  0$ for all  $x\in \R_+,$ it follows that 
  $$ \big\{ |\Psi_{A, \varepsilon}(\lambda,-x)|^2\big\}'+ \big\{ |\Psi_{A, \varepsilon}(\lambda,x)|^2\big\}'\leq 0,\qquad \forall x\in \R_+.$$ This implies 
$$   |\Psi_{A, \varepsilon}(\lambda,-x)|^2 +   |\Psi_{A, \varepsilon}(\lambda,x)|^2 \leq   |\Psi_{A, \varepsilon}(\lambda,0)|^2 +   |\Psi_{A, \varepsilon}(\lambda,0)|^2=2 ,\qquad \forall x\in \R_+.$$
As a consequence
 $$   |\Psi_{A, \varepsilon}(\lambda,-x)| \leq \sqrt 2 \quad \text{and} \quad    |\Psi_{A, \varepsilon}(\lambda,x)|\leq \sqrt 2, \qquad \forall x\in \R_+ .$$
  
This finishes the proof of the first statement.   
  \\
 2) For $\lambda =a+ib\in \C$  we define  the function $Q_{\varepsilon,\lambda}$ by 
 $$Q_{\varepsilon,\lambda}(x)={{\Psi_{A, \varepsilon}(\lambda,x)}\over {\Psi_{A, \varepsilon}(ib,x)}}.$$ By Theorem \ref{thm-posi}  the function $Q_{\varepsilon,\lambda}$ is well defined. 
Since $ \Psi_{A, \varepsilon}(\lambda,x)$ satisfies the differential-reflection  equation \eqref{di-di}, it follows that 
$${{ \Psi_{A, \varepsilon}'(\lambda,x) }\over{\Psi_{A, \varepsilon}(ib,x)}}=-{{A'(x)}\over{2A(x)}} \Big( Q_{\varepsilon,\lambda}(x)-Q_{\varepsilon,\lambda}(-x){{\Psi_{A, \varepsilon}(ib,-x)}\over{\Psi_{A, \varepsilon}(ib,x)}}\Big)
+\varepsilon \varrho Q_{\varepsilon,\lambda}(-x) {{\Psi_{A, \varepsilon}(ib,-x)}\over{\Psi_{A, \varepsilon}(ib,x)}}+i\lambda Q_{\varepsilon,\lambda}(x).$$
We now take the derivative of $Q_{\varepsilon,\lambda}:$ 
\begin{eqnarray*}
 Q'_{\varepsilon,\lambda}(x) 
&=&   {{ \Psi_{A, \varepsilon}'(\lambda,x) }\over{\Psi_{A, \varepsilon}(ib,x)}} -Q_{\varepsilon,\lambda}(x) {{ \Psi_{A, \varepsilon}'(ib,x) }\over{\Psi_{A, \varepsilon}(ib,x)}}\\
&= & -{{A'(x)}\over{2A(x)}}  Q_{\varepsilon,\lambda}(x) +{{A'(x)}\over{2A(x)}} Q_{\varepsilon,\lambda}(-x){{\Psi_{A, \varepsilon}(ib,-x)}\over{\Psi_{A, \varepsilon}(ib,x)}}+\varepsilon \varrho Q_{\varepsilon,\lambda}(-x) {{\Psi_{A, \varepsilon}(ib,-x)}\over{\Psi_{A, \varepsilon}(ib,x)}}\\
& & -Q_{\varepsilon,\lambda}(x) \Big(-{{A'(x)}\over{2A(x)}}  \Big(1-{{\Psi_{A, \varepsilon}(ib,-x)}\over{\Psi_{A, \varepsilon}(ib,x)}}\Big) +\varepsilon \varrho {{\Psi_{A, \varepsilon}(ib,-x)}\over{\Psi_{A, \varepsilon}(ib,x)}}-b\Big)
+i\lambda Q_{\varepsilon,\lambda}(x)\\
&=&  \Big( {{A'(x)}\over{2A(x)}} +\varepsilon \varrho\Big)  \Big(Q_{\varepsilon,\lambda}(-x)-Q_{\varepsilon,\lambda}(x)\Big) {{\Psi_{A, \varepsilon}(ib,-x)}\over{\Psi_{A, \varepsilon}(ib,x)}}+(i\lambda+b) Q_{\varepsilon,\lambda}(x).
\end{eqnarray*}
Hence,
\begin{align} \label{210}
&\big\{|Q_{\varepsilon,\lambda}(x)|^2\big\}' \nonumber\\
=&\, Q_{\varepsilon,\lambda}'(x)\overline{Q_{\varepsilon,\lambda}(x)}+Q_{\varepsilon,\lambda}(x) \big\{\overline{Q_{\varepsilon,\lambda}(x)}\big\}' \nonumber\\
=&\,   2\Re\big\{ Q_{\varepsilon,\lambda}'(x)\overline{ Q_{\varepsilon,\lambda}(x)}\big\}  \nonumber\\
= & \, 2\Re \Big\{\Big( {{A'(x)}\over{2A(x)}} +\varepsilon \varrho\Big)  \Big(Q_{\varepsilon,\lambda}(-x)\overline{Q_{\varepsilon,\lambda}(x)}-|Q_{\varepsilon,\lambda}(x)|^2\Big) {{\Psi_{A, \varepsilon}(ib,-x)}\over{\Psi_{A, \varepsilon}(ib,x)}}+(i\lambda+b) |Q_{\varepsilon,\lambda}(x)|^2\Big\} \nonumber \\
=&\,  2\Re \Big\{\Big ({{A'(x)}\over{2A(x)}} +\varepsilon \varrho\Big) \Big( Q_{\varepsilon,\lambda}(-x)\overline{Q_{\varepsilon,\lambda}(x)} -|Q_{\varepsilon,\lambda}(x)|^2\Big) {{\Psi_{A, \varepsilon}(ib,-x)}\over{\Psi_{A, \varepsilon}(ib,x)}} \Big\} \nonumber \\
=&\,  -2\Big({{A'(x)}\over{2A(x)}} +\varepsilon \varrho\Big)\left( |Q_{\varepsilon,\lambda}(x)|^2-\Re \big\{Q_{\varepsilon,\lambda}(-x)\overline{Q_{\varepsilon,\lambda}(x)} \big\}  \right){{\Psi_{A, \varepsilon}(ib,-x)}\over{\Psi_{A, \varepsilon}(ib,x)}}.
\end{align}
Similarly  we have 
\begin{align}
\big\{|Q_{\varepsilon,\lambda}(-x)|^2\big\}'&= -Q_{\varepsilon,\lambda}'(-x)\overline{Q_{\varepsilon,\lambda}(-x)}-Q_{\varepsilon,\lambda}(-x)  \overline{Q_{\varepsilon,\lambda}'(-x)} \nonumber \\
&=  -2\Re\big\{ Q_{\varepsilon,\lambda}'(-x)\overline{ Q_{\varepsilon,\lambda}(-x)}\big\}\nonumber \\
&= -2\Big({{A'(x)}\over{2A(x)}} -\varepsilon \varrho\Big)\Big( |Q_{\varepsilon,\lambda}(-x)|^2-\Re \big\{Q_{\varepsilon,\lambda}(x)\overline{Q_{\varepsilon,\lambda}(-x)} \big\}  \Big) {{\Psi_{A, \varepsilon}(ib,x)}\over{\Psi_{A, \varepsilon}(ib,-x)}}.\label{211}
\end{align}
Since $-1\leq \varepsilon\leq 1,$ then, by assumptions on the function $A'/(2A),$  we  have 
  $$\Big({{A'(x)}\over{2A(x)}} \pm \varepsilon \varrho\Big)\geq 0 \qquad \forall \,x\in \R_+ . $$
It follows from \eqref{210} and \eqref{211} that for every $x\in \R_+$
 $$
 \big\{|Q_{\varepsilon,\lambda}(x)|^2\big\}'\leq -2\Big ({{A'(x)}\over{2A(x)}} +\varepsilon \varrho\Big)
  |Q_{\varepsilon,\lambda}(x)|
 \Big( |Q_{\varepsilon,\lambda}(x)|- |Q_{\varepsilon,\lambda}(-x)|  \Big) {{\Psi_{A, \varepsilon}(ib,-x)}\over{\Psi_{A, \varepsilon}(ib,x)}},$$
and
$$ \big\{|Q_{\varepsilon,\lambda}(-x)|^2\big\}'\leq -2\Big({{A'(x)}\over{2A(x)}} -\varepsilon \varrho\Big)
 |Q_{\varepsilon,\lambda}(-x)| \Big( |Q_{\varepsilon,\lambda}(-x)| - |Q_{\varepsilon,\lambda}(x)|    \Big){{\Psi_{A, \varepsilon}(ib,x)}\over{\Psi_{A, \varepsilon}(ib,-x)}}.$$
Thus we can conclude that 
$$ \big\{|Q_{\varepsilon,\lambda}(x)|^2\big\}'\leq 0 \quad \text{if} \quad  |Q_{\varepsilon,\lambda}(x)| \geq  |Q_{\varepsilon,\lambda}(-x)|,$$
 and $$\big\{|Q_{\varepsilon,\lambda}(-x)|^2\big\}'\leq 0  \quad \text{if} \quad  |Q_{\varepsilon,\lambda}(-x)| \geq  |Q_{\varepsilon,\lambda}(x)|.$$
 As a real analytic function of $x,$ $|Q_{\varepsilon,\lambda}(x)|^2$ and $|Q_{\varepsilon,\lambda}(-x)|^2$ coincide either everywhere or on a discrete 
 subset of $\R$ with no accumulation point. In the first case, $|Q_{\varepsilon,\lambda}(x)|^2 = |Q_{\varepsilon,\lambda}(-x)|^2$ is a 
 decreasing function of $x\in \R_+. $
 In the second case, for $x\in \R_+,$ let $$M(x):=\max \big\{ |Q_{\varepsilon,\lambda}(x)|^2 ,|Q_{\varepsilon,\lambda}(-x)|^2\big\}.$$
  If $|Q_{\varepsilon,\lambda}(x)|> |Q_{\varepsilon,\lambda}(-x)|,$ then $M(x)=|Q_{\varepsilon,\lambda}(x)|^2$  and 
  $M'(x)= \big\{|Q_{\varepsilon,\lambda}(x)|^2\big\}'<0.$  If $|Q_{\varepsilon,\lambda}(x)|< |Q_{\varepsilon,\lambda}(-x)|,$ then 
  $M(x)=|Q_{\varepsilon,\lambda}(-x)|^2$  and $M'(x)= \big\{|Q_{\varepsilon,\lambda}(-x)|^2\big\}'<0. $
 If $|Q_{\varepsilon,\lambda}(x)|= |Q_{\varepsilon,\lambda}(-x)|$ for some $x\in \R_+,$ then $M$ has left and right derivatives
  at $x,$ which are non-positive. 
 Thus $M$ is decreasing on $\R_+.$ In conclusion,  for every $x\in \R_+,$     $|Q_{\varepsilon,\lambda}(x)|^2\leq  M(0) =1$
  and  $|Q_{\varepsilon,\lambda}(-x)|^2\leq  M(0)=1.$
That is for every  $x\in \R,$ we have $|Q_{\varepsilon,\lambda}(x)|\leq  |Q_{\varepsilon,\lambda}(0)|=1 .$ 
This  finishes the proof of the second statement. 
\\
3) We proceed analogously to the function $Q_{\varepsilon,\lambda}$ above 
by considering the function  $$R_{\varepsilon,b}(x):={{\Psi_{A, \varepsilon}(ib,x) e^{-| b |\, |x|}} \over{\Psi_{A, \varepsilon}(0,x)}} .$$
\\
4) The fact that $\Psi_{A,0}(0,x) =1$ follows immediately from \eqref{=0}. Assume that  $\varepsilon\not =0.$ In this case 
  $$\Psi_{A, \varepsilon}(0,x)=\varphi_{\mu_\varepsilon^0}(x) -{1\over {\varepsilon \varrho}} \varphi_{\mu_\varepsilon^0}'(x),$$
  where $\mu_\varepsilon^0$ satisfies   $({\mu_\varepsilon^0})^2=(\varepsilon^2-1)\varrho^2.$ Since  $|\varepsilon |\leq 1,$ it follows from Lemma \ref{BX} 4), 5),   6)   that there exists a positive constant $c_\varepsilon$ such that 
    $$ \Psi_{A, \varepsilon}(0,x) \leq c_\varepsilon (|x|+1) e^{-\varrho (1- \sqrt{1-\varepsilon^2 } )|x|}.$$
 \end{proof} 

Henceforth, we will assume that  $-1\leq \varepsilon \leq 1.$ Beside the growth estimates above, we will include  estimates for  the partial derivatives of 
$ \Psi_{A, \varepsilon} $. We remind the reader that 
$$
 \Psi_{A, \varepsilon}(0,x)+\Psi_{A, \varepsilon}(0,-x)  =2\varphi_{i\sqrt{1-\varepsilon^2}\varrho}(x), 
 $$
with $\varphi_{i\varrho}(x)=1$ and $$\varphi_{i\sqrt{1-\varepsilon^2}\varrho}(x)\leq c(|x|+1)e^{- \varrho (1-\sqrt{1-\varepsilon^2}) |x| } ,$$ (see  Lemma \ref{BX}).

\begin{thm}\label{estder}  
\begin{enumerate}[\upshape 1)] 
\item Assume that $\lambda\in \C$ and $|x|\geq x_0$ with $x_0>0.$
Given $N\in \N,$ there is a  positive constant $c$ such that 
\begin{equation}\label{es1}
\Big| {{\partial ^N}\over{\partial x^N}} \Psi_{A,\varepsilon}(\lambda, x)\Big| \leq c (|\lambda|+1)^N 
e^{|\Im \lambda | \;|x|} \varphi_{i\sqrt{1-\varepsilon^2}\varrho}(x).
\end{equation} 
\item Assume that $\lambda\in \C$  and $x\in \R. $
Given $M\in \N,$ there is a  positive constant $c$ such that 
\begin{equation}\label{es2}
\Big| {{\partial ^M}\over{\partial \lambda^M}}   \Psi_{A, \varepsilon}(\lambda, x)\Big| \leq c |x| ^M
e^{|\Im \lambda | \;|x|} \varphi_{i\sqrt{1-\varepsilon^2}\varrho}(x).
 \end{equation}
 \end{enumerate}
 \end{thm}
\begin{proof} 1) If $N=0$ this is nothing but Theorem \ref{esti} 2), 3) and 4). So assume $N\geq 1.$ On the one hand,  
$\Psi_{A, \varepsilon}(\lambda, x)$ satisfies the following equation 
$$
 \Psi_{A, \varepsilon}'(\lambda,x)=-{{A'(x)}\over{2A(x)}} \Big( \Psi_{A, \varepsilon}(\lambda,x)-\Psi_{A, \varepsilon}(\lambda,-x)\Big) +\varepsilon \varrho \Psi_{A, \varepsilon}(\lambda,-x)+i\lambda \Psi_{A, \varepsilon}(\lambda,x).
$$
This  allows us to express   the derivatives of $\Psi_{A, \varepsilon}(\lambda, \cdot)$ in terms of lower order derivatives. On the other hand, since  $A'/(2A)$ satisfies the hypothesis  (H4), it follows that there exists a positive constant $C$ such that $$\left| \left({{A'(x)}\over{2A(x)}}\right)^{(N)}\right| \leq C,\qquad \forall\; |x|\geq x_0 \;\text{with}\; x_0>0.$$  
 Now the estimate \eqref{es1} can be proved by induction on $N.$
 \\
2) Recall that the mapping $\lambda\mapsto \Psi_{A, \varepsilon}(\lambda, x)$ is entire, for every $x\in \R,$ 
and that \begin{equation}\label{estine}
 |   \Psi_{A,\varepsilon}(\lambda, x) | \leq c 
e^{|\Im \lambda | \;|x|} \varphi_{i\sqrt{1-\varepsilon^2}\varrho}(x)
\end{equation} for all $\lambda\in \C$ and $x\in \R.$
If $M=0$ this is just \eqref{estine}. So assume $M>0.$ If $x=0,$ the statement follows from Liouville's theorem. If $x\not =0,$ 
apply  Cauchy's integral formula for $ \Psi_{A, \varepsilon}(\lambda, x)$  over a circle with radius proportional to ${1\over{|x|}},$ centered at $\lambda$ in the complex plane.
\end{proof}

\section{A Laplace type representation of  the eigenfunctions}

In this section we will show that $\Psi_{A, \varepsilon}(\lambda,\cdot)$ can be expressed as the 
Laplace transform of a compactly supported function. In the literature this is the so-called Mehler's type formula. 

Denote by $C_e^\infty(\R)$ the space of even functions in $C^\infty(\R).$ For $f\in C_e^\infty(\R)$ we set 
\index{$\mathcal E_{\varepsilon}$lmain}
\begin{equation}\label{e}
\mathcal E_\varepsilon f(x):= f(x)-{{\varrho_\varepsilon |x|}\over 2} \int_{|y|<|x|} f(y)
{{J_1(\varrho_\varepsilon \sqrt{x^2-y^2})}\over { \sqrt{x^2-y^2}}} dy,
\end{equation}
where  
$J_1$ is the Bessel function of the first kind, and \index{$\varrho_\varepsilon$}
\begin{equation}\label{rho}
\varrho_\varepsilon:=\sqrt{1-\varepsilon^2}\,\varrho .
\end{equation}
If $\varepsilon=\pm1,$ then $\varrho_{\pm1}=0,$ and therefore $\mathcal E_{\pm1}=\id.$ The following statement  is nothing but a reformulation of  Proposition 2.1  in \cite{Th}. See also Theorem 5.1 in \cite{Li}.
\begin{prop} The transform integral $\mathcal E_\varepsilon$ is an automorphism of $C_e^\infty(\R)$   satisfying 
\begin{equation}\label{c3}
\begin{cases}
&\!\!\!\displaystyle {{d^2}\over {dx^2}} \circ \mathcal E_\varepsilon =\mathcal E_\varepsilon \circ \Big( {{d^2}\over {dx^2}} -\varrho_\varepsilon^2\Big),\vspace{.2cm}\\
&\!\!\!\mathcal E_\varepsilon f(0)=f(0).
\end{cases}
\end{equation}
The transform inverse  $\mathcal E_\varepsilon^{-1}$ is given by 
\begin{equation}\label{e-1}
\mathcal E_\varepsilon^{-1} f(x)=f(x)+{{\varrho_\varepsilon |x|}\over 2} \int_{|y|<|x|} f(y) {{I_1(\varrho_\varepsilon \sqrt{x^2-y^2})}\over { \sqrt{x^2-y^2}}} dy, 
\end{equation}
where $I_1$ is the modified Bessel function of the first kind. 
\end{prop}

Let  $\mathcal D_e(\R)$ be the space of even functions in $\mathcal D(\R).$ For $g\in  \mathcal D_e(\R)$  put
\index{${}^{\rm t}\mathcal E_{\varepsilon}$}
\begin{equation}
\label{e-t11}{}^{\rm t} \mathcal E_\varepsilon g(y)=g(y)-{{\varrho_\varepsilon}\over 2}   \int_{|x|> |y|}  \, |x|\, g(x) \,{{J_1(\varrho_\varepsilon \sqrt{x^2-y^2})}\over{ \sqrt{x^2-y^2}}} dx, 
\end{equation}
 where $\varrho_\varepsilon$ is as in \eqref{rho}. We may rewrite ${}^{\rm t} \mathcal E_\varepsilon g$ as
\begin{equation}
\label{e-t12}{}^{\rm t} \mathcal E_\varepsilon g(y)= -   \int_{ |y|}^\infty  g'(x) J_0(\varrho_\varepsilon \sqrt{x^2-y^2})dx. 
\end{equation}
 Below we will show that $\mathcal D_e(\R)$ is stable by ${}^{\rm t} \mathcal E_\varepsilon .$ Thus, one may check   that for all $f\in C_e^\infty(\R)$ and all  $g\in \mathcal D_e(\R),$  
 $$ \int_\R\mathcal E_\varepsilon f(x) \,g(x)dx=\int_\R f(y) \;   {}^{\rm t} \mathcal E_\varepsilon  g(y)dy.$$

 \begin{thm} The integral transform  ${}^{\rm t} \mathcal E_\varepsilon $ is an automorphism of $\mathcal D_e(\R)$  satisfying 
\begin{equation}\label{et}
{}^{\rm t} \mathcal E_\varepsilon  \circ {{d^2}\over{dx^2}} =\Big({{d^2}\over{dx^2}}-\varrho_\varepsilon^2\Big) \circ{}^{\rm t} \mathcal E_\varepsilon .
 \end{equation}
  The transform  inverse ${{}^{\rm t}\mathcal E_\varepsilon}^{-1}$ is given by 
\begin{equation}
\label{e-t21}{}^{\rm t}\mathcal  E_\varepsilon^{-1} g(y)=g(y) +{{\varrho_\varepsilon}\over 2}  \int_{|x|> |y|} \,|x|\, g(x)\, {{I_1(\varrho_\varepsilon \sqrt{x^2-y^2})}\over{ \sqrt{x^2-y^2}}} dx ,
\end{equation}
which we may rewrite it as 
\begin{equation}
\label{e-t22}
{}^{\rm t}\mathcal  E_\varepsilon^{-1} g(y)=- \int_{ |y|}^\infty g'(x)  I_0(\varrho_\varepsilon \sqrt{x^2-y^2})  dx .
\end{equation}
  \end{thm}

\begin{proof} It is clear that   ${}^{\rm t} \mathcal E_\varepsilon  g $ is an even function whenever  $ g$ 
is even. A direct calculation gives the intertwining property  \eqref{et},  which we  may rewrite it  as  
$D^2\circ  {}^{\rm t} \mathcal E_\varepsilon = {}^{\rm t} \mathcal E_\varepsilon \circ (D^2+\varrho_\varepsilon^2),$ where $D:= {{d }\over{dx }}.$ Thus, for all $N\in \N$ 
and for all $y\in \R_+,$ we have
\begin{eqnarray*}
D^{2N}\circ   {}^{\rm t}\mathcal E_\varepsilon g(y)&=&  {}^{\rm t}\mathcal E_\varepsilon\circ  (D^2+\varrho_\varepsilon^2)^Ng(y)\\
&=&-\int_y^\infty J_0\big(\varrho_\varepsilon \sqrt{x^2-y^2}\big) \;D(D^2+\varrho_\varepsilon^2)^N g(x) dx.
\end{eqnarray*} 
Using the well know fact that $|J_0(r)| \leq 1$ for all $r\in \R_+,$ it follows that if 
$ \supp (g)\subset [-a,a],$ then there exists a constant $c$ such that 
$$\sup_{y\in [-a,a]}| D^{2N}\circ   {}^{\rm t}\mathcal E_\varepsilon   g(y)| \leq c \sup_{x\in [-a,a]}| D^{M} g(x)| <\infty,$$ 
for some positive  integer $M.$ 
Thus, the space $\mathcal D_e(\R)$  is stable by $ {}^{\rm t} \mathcal E_\varepsilon .$ 

We now prove that the transform inverse of ${}^{\rm t} \mathcal E_\varepsilon $ is given by  \eqref{e-t21}. Recall that we may rewrite ${}^{\rm t} \mathcal E_\varepsilon $  as
$${}^{\rm t} \mathcal E_\varepsilon g(y)= -   \int_{ |y|}^\infty   g'(x) J_0(\varrho_\varepsilon \sqrt{x^2-y^2})dx.
$$
We may also rewrite  the ``potential" transform inverse as 
$$
{}^{\rm t}\mathcal  E_\varepsilon^{-1} g(y)=- \int_{ |y|}^\infty g'(x)  I_0(\varrho_\varepsilon \sqrt{x^2-y^2})  dx .
$$
We will assume that $y>0.$ Then 
\begin{eqnarray*}
 {}^{\rm t} \mathcal E_\varepsilon  \big( {}^{\rm t}\mathcal  E_\varepsilon^{-1} g\big)(y) 
&=&-\int_{x>y} \left\{  {}^{\rm t}\mathcal  E_\varepsilon^{-1} g (x)\right\}' J_0(\varrho_\varepsilon \sqrt{x^2-y^2})dx\\
&=& \int_{x>y} \Big\{ \int_{s>x} g'(s) I_0(\varrho_\varepsilon \sqrt{s^2-x^2}) ds \Big\}' J_0(\varrho_\varepsilon \sqrt{x^2-y^2})dx\\
&=&-\int_{x>y} g'(x) J_0(\varrho_\varepsilon \sqrt{x^2-y^2})dx \\
&& +\int_{x>y}   \Big\{ \int_{s>x} g'(s) \partial_x I_0(\varrho_\varepsilon \sqrt{s^2-x^2}) ds \Big\}J_0(\varrho_\varepsilon \sqrt{x^2-y^2})dx.
\end{eqnarray*}
Integration by parts implies 
$$\int_{s>x} g'(s) \partial_x I_0(\varrho_\varepsilon \sqrt{s^2-x^2}) ds=  {{\varrho_\varepsilon^2}\over 2}\, x \,g(x) -\int_{s>x} g(s) \partial_s\partial_x I_0(\varrho_\varepsilon \sqrt{s^2-x^2}) ds. $$
Above we have used the fact that $I_0'(z)=I_1(z)$ and that  the function  $\bigl({z\over 2}\bigr)^{-\nu} I_\nu(z)$ is normalized at $0$ by $1.$
Thus,
\begin{eqnarray*}
{}^{\rm t} \mathcal E_\varepsilon  \big( {}^{\rm t}\mathcal  E_\varepsilon^{-1} g\big)(y)&=&-\int_{x>y} g'(x) J_0(\varrho_\varepsilon \sqrt{x^2-y^2})dx + {{\varrho_\varepsilon^2}\over 2} \int_{x>y} x \,g(x) \,J_0(\varrho_\varepsilon \sqrt{x^2-y^2})dx\\
 &&-\int_{s>y}   g(s) \:\Big\{  \int_{y}^s   J_0(\varrho_\varepsilon \sqrt{x^2-y^2})   \partial_s\partial_x I_0(\varrho_\varepsilon \sqrt{s^2-x^2}) dx\Big\} ds.
\end{eqnarray*}
Next, we will compute the integral within  brackets on the right hand side of the identity above. 
On the one hand, since $I_0'(z)=I_1(z),$ we have 
\begin{eqnarray*}
 \partial_s\partial_x I_0(\varrho_\varepsilon \sqrt{s^2-x^2}) &=&-\varrho_\varepsilon x \partial_s \Big( (s^2-x^2)^{-1/2} I_1( \varrho_\varepsilon \sqrt{s^2-x^2}) \Big)\\
&=&\varrho_\varepsilon {{xs}\over{ (s^2-x^2)^{3/2}}} I_1(\varrho_\varepsilon \sqrt{s^2-x^2}) -\varrho_\varepsilon^2  {{xs}\over{ (s^2-x^2)}} I_1'(\varrho_\varepsilon \sqrt{s^2-x^2}) \\
&=&-\varrho_\varepsilon^2  {{xs}\over{ (s^2-x^2)}}I_2(\varrho_\varepsilon \sqrt{s^2-x^2}) .
\end{eqnarray*}
Above we have used the well known differentiation identity
$I_\nu'(z)= I_{\nu+1}(z)+{\nu\over z} I_\nu(z). $
On the other hand, using the following integral formula (see \cite[formula (1), page 725]{GR})
$$\int_0^a x^{\mu+1} (a^2-x^2)^{-\mu/2 -1}  J_\mu(x) I_\nu(\sqrt{a^2-x^2})dx= {{\left({a\over 2}\right)^\mu 
\Gamma\left({{\nu-\mu}\over2}\right)}\over{2\Gamma\left( {{\nu+\mu}\over2}+1\right)}} J_\nu(a) ,\quad   \Re \nu>\Re \mu >-1$$
 we have
 \begin{align*}
 &\int_{y}^s   J_0(\varrho_\varepsilon \sqrt{x^2-y^2})   \partial_s\partial_x I_0(\varrho_\varepsilon \sqrt{s^2-x^2}) dx\\
= &-\varrho_\varepsilon^2  s  \int_{y}^s   {{x}\over{ (s^2-x^2)}} J_0(\varrho_\varepsilon \sqrt{x^2-y^2}) I_2(\varrho_\varepsilon \sqrt{s^2-x^2}) dx \\
=&-{{\varrho_\varepsilon^2}\over 2} s J_2(\varrho_\varepsilon \sqrt{s^2-y^2}) \\
=&-\varrho_\varepsilon {s \over {\sqrt{s^2-y^2}}} J_1(\varrho_\varepsilon \sqrt{s^2-y^2}) +{{\varrho_\varepsilon^2}\over 2} s J_0(\varrho_\varepsilon \sqrt{s^2-y^2}) \\
= &\; \partial_s \Big(J_0(\varrho_\varepsilon \sqrt{s^2-y^2})\Big)+{{\varrho_\varepsilon^2}\over 2} s J_0(\varrho_\varepsilon \sqrt{s^2-y^2}). 
\end{align*}
Above we have used the recurrence relation   $J_{\nu+1}(z)+J_{\nu-1}(z)={{2\nu}\over z} J_\nu(z)$ and the fact that $J_0'(z)=-J_1(z).$
Consequently,
\begin{eqnarray*}
 {}^{\rm t} \mathcal E_\varepsilon  \big( {}^{\rm t}\mathcal  E_\varepsilon^{-1} g\big)(y)
&=&-\int_{x>y} g'(x) J_0(\varrho_\varepsilon \sqrt{x^2-y^2})dx + {{\varrho_\varepsilon^2}\over 2} \int_{x>y} x g(x) J_0(\varrho_\varepsilon \sqrt{x^2-y^2})dx\\
 &&-\int_{x>y}   g(x) \partial_x \Big(J_0(\varrho_\varepsilon \sqrt{x^2-y^2})\Big) dx - {{\varrho_\varepsilon^2}\over 2}  \int_{x>y} x  g(x)   J_0(\varrho_\varepsilon \sqrt{x^2-y^2}) dx\\
&=& g(y).
\end{eqnarray*}
Similarly  one proves that 
${}^{\rm t}\mathcal  E_\varepsilon^{-1}  \big(  {}^{\rm t} \mathcal E_\varepsilon  g\big)=g.$
\end{proof}

Recall from Lemma \ref{Ch1} that for every $\mu\in \C,$ the eigenfunction $\varphi_{\mu }$  has the following Laplace type representation 
\begin{equation}\label{laplace}
\varphi_{\mu } (x) =\int_{0}^{|x|} K(|x|,y) \cos({\mu y})dy,\qquad x\in \R^*,
\end{equation} 
where $K(|x|,\cdot)$  is a non-negative even continuous function supported in $[-|x|, |x|].$ 
The following alternative Laplace type representation of $\varphi_{\mu }$ is needed for later use. 

For $x\in \R$ and $y\in \R_+$ put \index{$K_{\varepsilon}(x,y)$}
\begin{equation}\label{Kepsilon}
K_\varepsilon({ |x|},y):= {}^{\rm t}\mathcal E_\varepsilon^{-1} K (|x|,\cdot)(y).
\end{equation}
Observe that $ K_\varepsilon (x,\cdot)$ is even, continuous and supported in $[-|x|,|x|].$ We note  that if  $\varepsilon=\pm1,$ then  the transformation  $\mathcal E_{\pm1}=\id,$ and therefore $K_{\pm1}(x,y)=K (|x|,y).$
\begin{lema}\label{alt} Let $\lambda\in \C.$ The integral representation \eqref{laplace} can be rewritten as 
$$\varphi_{\mu_\varepsilon }(x)=\int_0^{|x|}  K_\varepsilon(x,y) \cos(\lambda y) dy,$$
where the relationship between $\mu_\varepsilon$  and $\lambda$ is $\mu_\varepsilon^2 = \lambda^2+ ( \varepsilon^2-1)\varrho^2.$
\end{lema}
\begin{proof} The cases $\varepsilon=\pm 1$ are trivial since $\mathcal E_{\pm 1}=\id.$
So assume $\varepsilon\not = \pm 1.$   Observe that we may rewrite $\mathcal E_\varepsilon$ as 
$$  \mathcal E_\varepsilon (f)(y) :=\Big(y\int_0^1 f(yt) J_0(\varrho_\varepsilon y \sqrt{1-t^2})dt\Big)'_y$$
 for $y>0.$ Using the following  integral formula  (see \cite[formula (7), page 722]{GR})
 $$ \int_0^a \cos (ct)\:J_0 (b\sqrt{a^2-t^2})  dt={{\sin (a\sqrt{b^2+c^2})}\over {\sqrt{b^2+c^2}}},\qquad b>0$$
  we obtain
\begin{eqnarray*}
\mathcal E_\varepsilon \big(\cos (\mu_\varepsilon \cdot)\big)(y)&= &
\Big(y\int_0^1 \cos(\mu_\varepsilon yt) J_0(\varrho_\varepsilon y \sqrt{1-t^2})dt\Big)'_y\\
&=&\Big({{\sin(y\sqrt{\varrho_\varepsilon^2+\mu_\varepsilon^2})}\over{\sqrt{\varrho_\varepsilon^2+\mu_\varepsilon^2}}}\Big)'_y\\
&=&\cos\Big(\big\{\underbrace{(1-\varepsilon^2) \varrho^2 +\mu_\varepsilon^2}_{=\lambda^2}\big\}^{1/2} \,y\Big) .
\end{eqnarray*}
 Thus 
$$\mathcal E_\varepsilon^{-1} \big(\cos (\lambda\cdot )\big)(y) =\cos (\mu_\varepsilon y).$$  
Hence the integral representation \eqref{laplace} becomes 
\begin{eqnarray*}
\varphi_{\mu_\varepsilon }(x)&=&\int_0^{|x|} K  (|x|,y) \mathcal E_\varepsilon^{-1} (\cos (\lambda \cdot))(y) dy\\
&=&\int_0^{|x|}{}^{\rm t} \mathcal E_\varepsilon^{-1} K (|x|,\cdot)(y) \cos(\lambda y) dy.
\end{eqnarray*}
\end{proof}

We now establish a Laplace type representation of the eigenfunction $\Psi_{A, \varepsilon}(\lambda, \cdot).$
Henceforth we will use the  following notation 
\begin{equation}
G_\varepsilon(x,y):=\int_{|y|}^{|x|}  K_\varepsilon (t,y) A(t)dt. \index{$G_{\varepsilon}(x,y)$}
\end{equation}
The function $G_\varepsilon(x,\cdot )$ is even, continuous on  its support $[-|x|,|x|]$ and of class $C^1$ on $]-|x|,|x|[$  (see e.g. \cite[Lemma 2.8]{MT}).

\begin{thm}\label{Mh} For all $\lambda\in \C$   the  function $ \Psi_{A, \varepsilon}(\lambda, \cdot):\R^* \rightarrow \C $ is the Laplace transform of a compactly supported function. More precisely,
$$\Psi_{A, \varepsilon}(\lambda, x)=\int_{|y|<|x|}  \KK_\varepsilon(x,y) e^{i\lambda y} dy,\qquad \forall x\in \R^*,$$ where 
\index{$\KK_{\varepsilon}(x,y)$}
\begin{equation}\label{K}
\KK_\varepsilon(x,y):={{K_\varepsilon (x,y)}\over 2} +\varepsilon \varrho {{\sg(x) }\over{2A(x)}} G_\varepsilon(x,y) -{{\sg(x)}\over{2A(x)}}\partial_y G_\varepsilon(x,y).
\end{equation} 
\end{thm}
\begin{proof} By invoking the identity \eqref{2.4}  and Lemma \ref{alt}  in the second equality below
we have 
\begin{eqnarray*}
\Psi_{A, \varepsilon}(\lambda,x) &=&\varphi_{\mu_\varepsilon}(x) +(i\lambda+\varepsilon \varrho)   {{ \sg(x) }\over{A(x)}} \int_0^{|x|} \varphi_{\mu_\varepsilon}(t) A(t)dt\\
&=&\int_0^{|x|}  K_\varepsilon(x,y) \cos (\lambda y) dy+\varepsilon \varrho {{\sg(x)}\over{A(x)}} \int_0^{|x|} \Big\{\int_0^{| t|}  K_\varepsilon (t,y) \cos(\lambda y) dy\Big\} A(t)  dt\\
&&+i{{\sg(x)}\over{A(x)}} \int_0^{|x|}  \Big\{ \int_0^{| t|}  K_\varepsilon (t,y) \,\lambda\, \cos (\lambda y) dy\Big\} A(t) dt\\
&=&\int_0^{|x|} K_\varepsilon(x,y) \cos (\lambda y) dy +\varepsilon \varrho {{\sg(x)}\over{A(x)}} \int_0^{|x|}  \cos(\lambda y) \Big\{\int_{|y|}^{|x|}  K_\varepsilon (t,y) A(t) dt  \Big\} dy\\
&&+i{{\sg(x)}\over{A(x)}} \int_0^{|x|}   \big( \sin (\lambda y) \big)'\:\Big\{ \int_{|y|}^{|x|}  K_\varepsilon (t,y) A(t)dt\Big\} dy\\
&=&\int_{-|x|}^{|x|} 
\Big\{ {{ K_\varepsilon(x,y)}\over 2}+\varepsilon \varrho {{\sg(x)}\over{2A(x)}} {{ G_\varepsilon(x,y)} } - {{\sg(x)}\over{2A(x)}} {d\over{dy}}  G_\varepsilon(x,y)\Big\} e^{i\lambda y} dy.
\end{eqnarray*}
\end{proof}

\section{The existence of an intertwining operator}
This section is concerned with the existence of an intertwining operator 
between  $\Lambda_{A,\varepsilon}$
 and the ordinary  derivative   $d/dx.$   

Recall from \eqref{abel}  the definition of the Abel transform of a function  $f\in \mathcal D_e(\R)$ (the space of even and smooth functions with compact support on $\R$),
$$\mathcal A f(y)={1\over 2}\int_{|x|>|y|}  K(|x|,y) f(x) A(x)dx,\qquad   y\in \R$$ where $K(|x|,y)$ is as in \eqref{K1}. 
It is natural to define for smooth even functions the dual transform $ {}^{\rm t} \mathcal A$  of $\mathcal A$ in the following sense
 $$\int_{\R} f(y) \mathcal Ag(y) dy=\int_{\R}  {}^{\rm t}  \mathcal A f(x) g(x) A(x)dx. $$
In  \cite{Tri} the author showed that \index{${}^{\rm t}\mathcal A$lmain}
 $$ {}^{\rm t} \mathcal A  f(x) ={1\over 2}\int_{|u|<|x|} K(|x|,u)f(u) du.$$
Further, by \cite[Theorem 5.1]{Tri},  the transform
  ${}^{\rm t} \mathcal A $  is an automorphism  of  $C_e^\infty(\R)$ (the space of even and smooth functions on $\R$)  satisfying 
\begin{equation}\label{center3}
(\Delta +\varrho^2)  \circ{}^{\rm t} \mathcal A  = {}^{\rm t}\mathcal A   \circ  {{d^2}\over{dx^2}} ,
\end{equation}
where  $\Delta $ is the  operator \eqref{L}.

For $-1\leq \varepsilon\leq 1$ we define the integral transform $\mathcal A_\varepsilon$ on 
$\mathcal D_e(\R)$ by   \index{$\mathcal A_{\varepsilon}$lmain}
$$\mathcal A_\varepsilon g(y)={1\over 2}\int_{|x|>|y|}  K_\varepsilon(x,y) g(x) A(x) dx,$$
where the kernel $ K_\varepsilon$ is as in \eqref{Kepsilon}.  We note  that for $\varepsilon=\pm 1$ the transform $\mathcal A_\varepsilon$ reduces to the Abel transform $\mathcal A$. Thus we may think of  $\mathcal A_\varepsilon$ as a  
deformation of  the transform $\mathcal A.$ Let  $ {}^{\rm t}\mathcal A_\varepsilon  $ be the linear mapping of 
$C_e^\infty (\R)$ so that 
$$\int_{\R} f(y)\,   \mathcal A_\varepsilon g(y) dy =\int_{\R}  {}^{\rm t}\mathcal A_\varepsilon f(x) \,g(x) A(x) dx,$$ 
for  $f\in C_e^\infty (\R)$ and  $ g\in \mathcal D_e(\R).$ Then \index{${}^{\rm t}\mathcal A_{\varepsilon}$lmain}
$$   {}^{\rm t}\mathcal A_\varepsilon  f(x)={1\over 2}\int_{|y|<|x|}  K_\varepsilon (x,y) f(y)dy .$$
Notice that for $f\in C_e^\infty(\R)$ and $g\in \mathcal D_e(\R) ,$   the functions $  {}^{\rm t} \mathcal A_\varepsilon  f$ and  $ \mathcal A_\varepsilon g$ belong respectively to  $C_e^\infty(\R)$ and $\mathcal D_e(\R).$ 
Moreover,
\begin{equation}\label{center2}
 \mathcal A_\varepsilon =  {}^{\rm t}\mathcal E_\varepsilon^{-1} \circ \mathcal A   ,
 \end{equation}
 and 
 \begin{equation}\label{center2}
 {}^{\rm t}\mathcal A_\varepsilon =   {}^{\rm t} \mathcal A \circ    \mathcal E_\varepsilon^{-1} .
 \end{equation}
 The next corollary contains some additional properties of $ \mathcal A_\varepsilon$  and  ${}^{\rm t}\mathcal A_\varepsilon .$
\begin{coro}\label{next} Let $D$ be the ordinary  derivative and let $\Delta $ be the   operator \eqref{L}. Then for all $\varepsilon\in \R$ we have: 
\begin{enumerate}[\upshape 1)]
\item 
 $ \mathcal A_\varepsilon \circ (\Delta +\varrho^2)=(D^2+\varrho_\varepsilon^2)\circ  \mathcal A_\varepsilon,$ 
where $\varrho_\varepsilon^2=(1-\varepsilon^2)\varrho^2.$
\item  $(\Delta +\varrho^2) \circ  {}^{\rm t}  \mathcal A_\varepsilon   = {}^{\rm t}  \mathcal A_\varepsilon   \circ  (D^2+\varrho_\varepsilon^2).$
  \end{enumerate}
  \end{coro}
  \begin{proof}   The first statement  is an immediate consequence of \eqref{1.10} and \eqref{et}. The second transmutation 
property follows from \eqref{c3} and \eqref{center3}.
\end{proof}

Recall that the space of smooth functions $C^\infty(\R)$  equipped with  the topology of 
compact convergence for all derivatives is a Fr\'echet space. For $f\in C^\infty(\R)$ we define $V_{A, \varepsilon} f$ by 
\index{$V_{\varepsilon}$}
\begin{equation}\label{vvv}V_{A, \varepsilon} f(x)=
\begin{cases}\displaystyle 
\int_{|y|<|x|}  \KK_\varepsilon(x,y) f(y) dy,&\qquad x\not = 0\\
f(0),&\qquad x=0
\end{cases}\end{equation}
where the kernel $\KK_\varepsilon(x,y)$ is as in \eqref{K}. Observe that 
\begin{equation}\label{VE}
\Psi_{A, \varepsilon}(\lambda, x)=V_{A, \varepsilon}(e^{i\lambda\, \cdot}\,)(x).
\end{equation}
\begin{lema}\label{LVV}
The operator $V_{A, \varepsilon}$ can be expressed as 
\begin{equation}\label{VV}
V_{A, \varepsilon} f(x)=\Big(\id +\varepsilon \varrho \mathcal M\Big)\;  {}^{\rm t} \mathcal A_\varepsilon f_e(x)+
\Big(\varepsilon^2\varrho^2 \mathcal M+{d\over {dx}}\Big)\;
 {}^{\rm t} \mathcal A_\varepsilon (If_o)(x),
 \end{equation}
  where \index{$\mathcal M$}
  $$\mathcal M h(x):={{\sg(x)}\over {A(x)}} \int_0^{|x|} h(t) A(t)dt,$$
  and \index{$I$}\begin{equation}\label{I}
  I h(x):=\int_0^x h(t)dt.
  \end{equation}
\end{lema}
\begin{proof} 
As usual, we write $f$ as the superposition  $f=f_e+f_o$  of an even function $f_e$ 
and an odd function $f_o.$ On  the one hand, we have
\begin{eqnarray*}
V_{A, \varepsilon} f_e(x)&=&\int_{-|x|}^{|x|} {{ K_\varepsilon(x,y)}\over 2} f_e(y)dy+\varepsilon \varrho {{\sg(x)}\over{2A(x)}}
\int_{-|x|}^{|x|}   G_\varepsilon(x,y) f_e(y)dy\\
&=&\int_0^{|x|}  K_\varepsilon (x,y) f_e(y)dy +\varepsilon \varrho {{\sg(x)}\over{A(x)}} \int_0^{|x|}  G_\varepsilon(x,y) f_e(y)dy\\
&=&  {}^{\rm t} \mathcal A_\varepsilon f_e(x)+\varepsilon \varrho \mathcal M \circ  {}^{\rm t} \mathcal A_\varepsilon   f_e(x).
\end{eqnarray*}
On the other hand, 
$$V_{A, \varepsilon} f_o(x)=-{{\sg(x)}\over{A(x)}} \int_0^{|x|} f_o(y)\partial_y   G_\varepsilon(x,y) dy.$$
We claim that 
\begin{equation}\label{claimm}
-{{\sg(x)}\over{A(x)}} \int_0^{|x|} f_o(y) \partial_y  G_\varepsilon(x,y) dy =
\Big(\varepsilon^2\varrho^2  \mathcal M +{d\over {dx}}\Big)\;  {}^{\rm t}\mathcal A_\varepsilon  (I f_o)(x),
\end{equation}
where $I f_o$ is as in \eqref{I}. Indeed, by invoking  formula \eqref{Lw} in the first equality  below and the transmutation property    in Corollary \ref{next}.3  in the second equality below we have 
\begin{align*}
{d\over {dx}}\;{}^{\rm t} \mathcal A_\varepsilon  ( I f_o)(x) 
&=  {{\sg(x)}\over{A(x)}}\int_0^{|x|}  \Delta  \, {}^{\rm t}\mathcal A_\varepsilon   (I f_o)(s) A(s) ds\\
&= {{\sg(x)}\over{A(x)}} \int_0^{|x|}    {}^{\rm t}\mathcal A_\varepsilon  \Big({{d^2}\over {dx^2}}-\varepsilon^2\varrho^2\Big)  (I f_o)(s) A(s)ds\\
&= {{\sg(x)}\over{A(x)}} \int_0^{|x|}   {}^{\rm t}\mathcal A_\varepsilon (  f_o') (s) A(s) ds -\varepsilon^2\varrho^2  {{\sg(x)}\over{A(x)}} \int_0^{|x|}   {}^{\rm t} \mathcal A_\varepsilon  (I f_o) (s) A(s) ds \\
&= {{\sg(x)}\over{A(x)}} \int_0^{|x|} \Big\{\int_0^{s} K_\varepsilon (s,u) f_o'(u)du\Big\} A(s) ds 
-\varepsilon^2\varrho^2  \mathcal M\circ    {}^{\rm t}\mathcal A_\varepsilon  (I f_o) (x)\\
  &= {{\sg(x)}\over{A(x)}} \int_0^{|x|}   f_o'(u)\Big\{\int_{u}^{|x|}  K_\varepsilon (s,u)  A(s)ds\Big\} du
-\varepsilon^2\varrho^2  \mathcal M\circ   {}^{\rm t}\mathcal A_\varepsilon (I f_o) (x)\\
  &=- {{\sg(x)}\over{A(x)}} \int_0^{|x|} f_o(u) \partial_u G_\varepsilon(x,u) du    -\varepsilon^2\varrho^2  \mathcal M\circ   {}^{\rm t} \mathcal A_\varepsilon  (I f_o) (x).
\end{align*}
This concludes the proof of  claim \eqref{claimm}, and therefore the proof of the Lemma \ref{LVV}.
\end{proof}

\begin{thm}\label{th63} The operator    $V_{A, \varepsilon}$ is the unique automorphism of $C^\infty(\R)$ such that 
\begin{equation}\label{inter}
\Lambda_{A,\varepsilon} \circ V_{A, \varepsilon}=V_{A, \varepsilon}\circ {d\over {dx}},
\end{equation}
where $\Lambda_{A,\varepsilon}$ is the differential-reflection operator \eqref{Op}.
\end{thm}
\begin{proof}  For the proof of this theorem it is more convenient to rewrite   $V_{A, \varepsilon}  f_o$ in \eqref{VV} as 
\begin{equation}\label{711}
V_{A, \varepsilon} f_o(x)=\mathcal M\circ  {}^{\rm t} \mathcal A_\varepsilon \big( f_o'\big)(x).
\end{equation}
 Indeed,
 \begin{eqnarray*}
 V_{A, \varepsilon} f_o(x)&=&
{d\over {dx}}  \; {}^{\rm t}\mathcal A_\varepsilon  (If_o)(x)+\varepsilon^2 \varrho^2 \mathcal M \; {}^{\rm t}\mathcal A_\varepsilon (If_o)(x)\\
&=& \mathcal M\; \Delta \;  {}^{\rm t} \mathcal A_\varepsilon (If_o)(x) +\varepsilon^2 \varrho^2 \mathcal M\; {}^{\rm t} \mathcal A_\varepsilon  (If_o)(x)\\
&= &\mathcal M (\Delta +\varepsilon^2 \varrho^2)\; {}^{\rm t} \mathcal A_\varepsilon  (If_o)(x)\\
&=&\mathcal M \;{}^{\rm t} \mathcal A_\varepsilon  \big(  (If_o)''\big)(x)\\
&=&\mathcal M \; {}^{\rm t} \mathcal A_\varepsilon  \big(  f_o' \big)(x).
\end{eqnarray*}

Let  $C^\infty_e(\R)$ and $C^\infty_o(\R)$ be the 
subspaces of even and odd functions in $C^\infty(\R)$, respectively. 
Firstly, the operator  $d/dx$  is one to one from $C^\infty_o(\R)$ onto $C^\infty_e(\R)$,  and $d/dx\circ I=I\circ d/dx=\id.$ Secondly, The transform $\mathcal M$ is an isomorphism from $C^\infty_e(\R)$ to $C^\infty_o(\R)$ and its inverse is given by 
\begin{equation}\label{M-1}
\mathcal M^{-1}=  {d\over {dx}}+ {{A'(x)}\over {A(x)}}\id.
\end{equation} 
Thus,  from \eqref{VV} and \eqref{711} it follows that $V_{A, \varepsilon}$ is an automorphism of $C^\infty(\R).$  We now prove the transmutation property \eqref{inter}. 

By \eqref{711} we have  
\begin{eqnarray*}
\Lambda_{A,\varepsilon} (V_{A, \varepsilon} f_o)&=&\Lambda_{A,\varepsilon} \big(\mathcal M \;  {}^{\rm t}\mathcal A_\varepsilon (  f_o')\big)\\
&=&\big(\id +\varepsilon \varrho \mathcal M\big)\, {}^{\rm t} \mathcal A_\varepsilon (  f_o').
\end{eqnarray*}
Above we have used the fact that 
$$\Lambda_{A,\varepsilon} \circ \mathcal M= \id +\varepsilon \varrho \mathcal M.$$
Moreover, one can check that 
\begin{eqnarray*}
\Lambda_{A,\varepsilon} (V_{A, \varepsilon} f_e)&=&\Lambda_{A,\varepsilon} (  \,  {}^{\rm t}\mathcal A_\varepsilon   f_e+\varepsilon \varrho \mathcal M \circ   {}^{\rm t}\mathcal A_\varepsilon   f_e)\\
&=&{d\over{dx}} \;  {}^{\rm t}\mathcal A_\varepsilon   f_e-\varepsilon \varrho  \, {}^{\rm t}\mathcal A_\varepsilon   f_e +
\varepsilon \varrho\Big( {d\over{dx}}   +  {{A'(x)}\over{A(x)}} \Big)\mathcal M  \; {}^{\rm t}\mathcal A_\varepsilon   f_e+
 \varepsilon^2\varrho^2 \mathcal M \;  {}^{\rm t} \mathcal A_\varepsilon   f_e\\
&=&\Big({d\over{dx}} +\varepsilon^2\varrho^2 \mathcal M\Big)  \;  {}^{\rm t}\mathcal A_\varepsilon   f_e.
\end{eqnarray*}
Above we have used \eqref{M-1}. In summary,
\begin{equation}\label{id1}
\Lambda_{A,\varepsilon} (V_{A, \varepsilon} f)=\Big({d\over {dx}} +\varepsilon^2\varrho^2 \mathcal M\Big)  \;  {}^{\rm t}\mathcal A_\varepsilon   f_e+
\big(\id +\varepsilon \varrho \mathcal M\big)\, {}^{\rm t} \mathcal A_\varepsilon ( f_o').
\end{equation}
Now, by invoking the expression \eqref{VV} of the operator $V_{A, \varepsilon} $ we get
$$V_{A, \varepsilon} (f_e')=\Big(\varepsilon^2\varrho^2 \mathcal M+ {d\over {dx}} \Big) \,   {}^{\rm t}\mathcal A_\varepsilon  f_e,$$ 
and 
$$V_{A, \varepsilon}(   f_o')= \Big(\id + \varepsilon \varrho \mathcal M \Big)  \,    {}^{\rm t}\mathcal A_\varepsilon   (  f_o').$$
That is 
$$V_{A, \varepsilon}( f' )=\Big(\varepsilon^2\varrho^2 \mathcal M+ {d\over {dx}} \Big) \,   {}^{\rm t}\mathcal A_\varepsilon  f_e+
 \Big(\id + \varepsilon \varrho \mathcal M \Big)  \,    {}^{\rm t}\mathcal A_\varepsilon   (  f_o').$$ 
This compares well  with \eqref{id1}.

The uniqueness of $V_{A, \varepsilon}$ is due to the fact that $\Psi_{A, \varepsilon}(\lambda,x)=V_{A, \varepsilon}(e^{i\lambda\, \cdot}\,)(x)$ and   that $V_{A, \varepsilon} f(0)=f(0).$  
\end{proof}

On the space  $\mathcal D(\R)$ of smooth functions with compact support, we consider  
   the dual   operator ${}^{\rm t}V_{A, \varepsilon}$ of $V_{A, \varepsilon}$ in the sense that 
   \index{${}^{\rm t}V_{\varepsilon} $}
\begin{equation}\label{vtt} \int_\R V_{A, \varepsilon} f(x) g(x) A(x) dx=\int_\R f(y)\, {}^{\rm t}V_{A, \varepsilon}  g(y)dy.\end{equation}
That is  
\begin{equation}\label{Vt}
{}^{\rm t}V_{A, \varepsilon} g(y)=\int_{|x|>|y|} \KK_\varepsilon(x,y) g(x) A(x)dx.
\end{equation}

\begin{lema}\label{LVVt}
The dual operator ${}^{\rm t}V_{A, \varepsilon} $ can be expressed as 
\begin{equation}\label{expVt}
{}^{\rm t}V_{A, \varepsilon}   g(y)= \mathcal A_\varepsilon g_e(y)-\Big( \varepsilon \varrho -{d\over{dx}}\Big) \mathcal A_\varepsilon (J g_o)(y),
\end{equation}
where \index{$J$}
$$Jh(x):= \int_{-\infty}^x h(t)dt.$$
\end{lema}
\begin{proof} The reader will have no trouble verifying that for every even function 
$f\in C^\infty(\R)$ and every odd function $g\in \mathcal D(\R) $
$$\int_\R \mathcal M f(x) g(x) A(x) dx=-\int_\R f(x) J g(x) A(x) dx. $$
Starting from the expression \eqref{VV} of $V_{A, \varepsilon}$ in Lemma \ref{LVV}, and by invoking \eqref{711} in the first equality below, we obtain   
\begin{align*}
&\int_{\R} V_{A, \varepsilon} f(x)g(x) A(x)dx \\
=&\int_{\R} \Big\{  {}^{\rm t}\mathcal  A_\varepsilon  f_e(x) g_e(x)  
+\varepsilon \varrho   \mathcal M\:  {}^{\rm t}\mathcal A_\varepsilon    f_e(x) g_o(x)   +  \mathcal M \: {}^{\rm t}\mathcal A_\varepsilon    f'_o(x) g_o(x) \Big\} A(x)dx\hfill \\
=&\int_{\R} \Big\{ f_e(x) \mathcal A_\varepsilon g_e(x)  -\varepsilon \varrho  {}^{\rm t} \mathcal A_\varepsilon  f_e(x) Jg_o(x) A(x)   -  {}^{\rm t} \mathcal A_\varepsilon  (f_o')(x) Jg_o(x)  A(x)\Big\}dx\hfill \\
=&\int_{\R}\Big\{ f_e(x) \mathcal A_\varepsilon g_e(x)  -\varepsilon \varrho    f_e(x)  \mathcal A_\varepsilon Jg_o(x)   -  f_o'(x) \mathcal A_\varepsilon Jg_o(x) \Big\} dx\hfill \\
=&\int_{\R} \Big\{ f(x) \mathcal A_\varepsilon g_e(x)  -\varepsilon \varrho   f(x)  \mathcal A_\varepsilon Jg_o(x)   +  f_o(x) {d\over {dx}} \mathcal A_\varepsilon Jg_o(x)\Big\} dx\hfill \\
=&\int_{\R} f(x) \Big\{  \mathcal A_\varepsilon g_e(x)  -\varepsilon \varrho  \, \mathcal A_\varepsilon 
Jg_o(x)+ {d\over{dx}} \mathcal A_\varepsilon Jg_o(x)\Big\}dx. 
\end{align*}

This finishes the proof of Lemma \ref{LVVt}. 
\end{proof}

The operator ${}^{\rm t}V_{A, \varepsilon} $ satisfies the following additional property.
\begin{thm} 
The operator ${}^{\rm t}V_{A, \varepsilon} $ is the unique automorphism of $\mathcal D(\R)$ satisfying the intertwining property 
$${d\over{dx}} \circ  {}^{\rm t}V_{A, \varepsilon} ={}^{\rm t}V_{A, \varepsilon}  \circ \big( \Lambda_{A,\varepsilon} +2\varepsilon \varrho S\big),$$
where   $S $ denotes the symmetry $(S f)(x):= f(-x).$
 \end{thm}
\begin{proof}
The statement follows immediately from   below:
\begin{eqnarray*}
\int_{\R} {d\over{dx}} \,  {}^{\rm t} V_{A, \varepsilon}   f(x) g(x) dx  
&=&-\int_{\R} f(x) V_{A, \varepsilon}   g'(x) A(x)dx\\
&=&-\int_{\R} f(x) \Lambda_{A,\varepsilon } V_{A, \varepsilon} g(x) A(x) dx\\
&=&\int_{\R} \big(\Lambda_{A,\varepsilon} +2 \varepsilon\varrho S\big)f(x) V_{A, \varepsilon} g(x) A(x)dx\\
&=&\int_{\R}  {}^{\rm t}V_{A, \varepsilon}  \big(\Lambda_{A,\varepsilon} +2 \varepsilon \varrho S\big)f(x)  g(x)  dx. 
\end{eqnarray*}
Above we have used Lemma \ref{lema1}.
\end{proof}

\section{The positivity of the intertwining operator}
We shall say that a linear operator $L$ on $\mathcal D (\R)$ is positive, if $L$ leaves the positive cone 
$$\mathcal D(\R)_+:=\{ f\in \mathcal D(\R) \;:\, f(x)\geq 0 \;\text{for all } x\in \R\}$$ invariant. The following statement is the central result of this section.  
\begin{thm}\label{cent} For $-1\leq \varepsilon\leq 1,$ the intertwining operator $V_{A, \varepsilon}$ is positive. 
\end{thm}

For $\varepsilon=0$ and $1$, Theorem \ref{cent} is  known (cf. \cite{tripo1} and \cite{tripo2}). However, the case $-1\leq \varepsilon\leq 1$ has some technical difficulties to be over come compare to $\varepsilon=0$ and $1,$ as  $\varepsilon$ could be positive as well as negative.

The proof of the above theorem affords several steps, the crucial one being the positivity of 
$V_{A, \varepsilon} (p_s(u,\cdot))(x)$ for every  $s>0$ and  $u,x\in \R,$ where $$p_s(u,v):=\frac{e^{-\frac{(u-v)^2}{4s}}}{2\sqrt{\pi s}}$$ denotes the Euclidean heat kernel. 

For simplicity we will  write $W_\varepsilon (s; u,x)$ instead of $V_{A, \varepsilon} (p_s(u,\cdot))(x).$ Below we list some properties of $W_\varepsilon (s; u,x).$

\begin{lema}\label{w_ep} For every $s>0$ and  $u, x\in\R$,  we have
\begin{enumerate}[\upshape 1)]
\item  
 $\displaystyle W_\varepsilon(s;u,x)=\frac{1}{2\pi}\int_{\R}\Psi_{A, \varepsilon}(-\lambda,x)e^{-s\lambda^2}\, e^{i\lambda u}d\lambda.$
 \item The function $(u,x)\mapsto W_\varepsilon(s;u,x)$ is of class $C^1$ on $\R^2.$
\item 
 $\displaystyle (\Lambda_{A,\varepsilon}+ {\partial_u} ) W_\varepsilon(s;u,x)=0.$
\item  $\displaystyle{\lim_{\Vert (u,x)\Vert\rightarrow+\infty} W_\varepsilon(s;u,x)=0.
}$
\end{enumerate}
\end{lema}
\begin{proof} 
1) For $x=0,$ we have $W_\varepsilon(s;u,0) = p_s(u,0)=\frac{e^{- {u^2}/{4s}}}{2\sqrt{\pi s
}}.$ Thus, for $x=0$,  the statement    follows from the well known fact 
\begin{equation}\label{kw} \int_\R e^{-s\lambda^2} e^{i\lambda \xi} d\lambda = \sqrt{\pi\over s}\, e^{-{{\xi^2}\over {4s}}}.
\end{equation}
For $x\neq0$, using again \eqref{kw} together with the Laplace type representation \eqref{vvv} of $V_{A, \varepsilon}$, we have 
\begin{eqnarray*}
W_\varepsilon(s;u,x)&=&\frac{1}{2\pi}\int_{-|x|}^{|x|}\mathbb{K}_{\varepsilon}(x,y)\Big(\int_{\R}e^{-s\lambda^2}e^{i\lambda(u-y)}d\lambda \Big)dy\\
&=&\frac{1}{2\pi}\int_{\R}\Big(\int_{-|x|}^{|x|}\mathbb{K}_{\varepsilon}(x,y)e^{-i\lambda y}dy \Big)e^{-s\lambda^2}e^{i\lambda u}\, d\lambda\\
&=&\frac{1}{2\pi}\int_{\R} \Psi_{A, \varepsilon}(-\lambda, x)e^{-s\lambda^2}e^{i\lambda u}\,d\lambda.
\end{eqnarray*}
2) For $|x|\geq x_0$ with $x_0>0,$ the statement follows from 1) and the growth estimate of $|\partial_x  \Psi_{A, \varepsilon}(\lambda,x)|$ (see Theorem \ref{estder}). Assume that $|x|\leq x_0.$ Using the fact  that 
$$|\varphi_\mu' (x)|\leq c\, |\mu^2+\varrho^2|\,(|x|+1) \, |x|\,e^{(|\Im \mu| -\varrho)|x|},$$ and that 
$$|\varphi_\mu'' (x)|\leq c \,|\mu^2+\varrho^2| \, (|x|+1)^2 \,e^{(|\Im \mu| -\varrho)|x|},$$
for all $\mu\in \C$ and $x\in \R$ (cf. \cite[Proposition 6.I.5]{Tribook}), we deduce from \eqref{2.2} that 
$$|\partial_x  \Psi_{A, \varepsilon}(\lambda,x)|\leq c ( |\lambda| +1)^2 ( |x|+1)^2 e^{(|\Im \mu_\varepsilon | -\varrho)|x|},$$
where $\mu_\varepsilon^2=\lambda^2-(1-\varepsilon^2)\varrho^2.$  It follows that in both cases $\lambda^2-(1-\varepsilon^2)\varrho^2 \gtreqless 0,$ we have $|\partial_x  \Psi_{A, \varepsilon}(\lambda,x)|\leq c ( |\lambda | +1)^2 $ for all $|x|\leq x_0. $\\
3) In view of 1), the present statement is easy to check.\\
4) For $x=0,$  $W_\varepsilon(s;u,0) = p_s(u,0)=\frac{e^{-\frac{u^2}{4s}}}{2\sqrt{\pi s
}}\rightarrow 0$ as $\Vert (u,x) \Vert \rightarrow \infty.$

For $x\neq0$, using 1)  and the  growth property  of the eigenfunction $\Psi_{A, \varepsilon}$ in Theorem \ref{esti}.4, we get
$$|W_\varepsilon(s;u,x)|\leq c_{\varepsilon}(1+|x|)e^{-\varrho(1-\sqrt{1-\varepsilon^2})|x|}.   $$
Now, the statement follows  by means of polar coordinates. 
\end{proof}

 The following lemma is also needed. 
 
 \begin{lema}\label{w_eps} Writing $W_\varepsilon$ as  $ W_\varepsilon(s;u,x)= W_\varepsilon^{\rm{even}}(s;u,x)+ W_\varepsilon^{\rm{odd}}(s;u,x),$ 
 where $W_\varepsilon^{\rm{even}}$ (resp. $W_\varepsilon^{\rm{odd}}$) denotes the even (resp. the odd) part of $W_\varepsilon$ with respect to  $x$, we have  $W_\varepsilon^{\rm{even}} (s;u,x) >0.$
 \end{lema}
\begin{proof}  Using Lemma \ref{w_ep}.1 together with the expression \eqref{2.4} of the eigenfunction 
$\Psi_{A, \varepsilon}$, we have
\begin{align*}
 &W_\varepsilon (s;u,x)  \\
=& \,\frac{1}{2\pi}\int_{\R} \Big(\varphi_{\mu_{\varepsilon}}(x)
 +(-i\lambda+\varepsilon\varrho)\frac{\sg(x)}{A(x)}\int_{0}^{|x|}\varphi_{\mu_{\varepsilon}}(z)A(z)dz\Big)\,e^{-s\lambda^2}e^{i\lambda u}\, d\lambda \\
=& \, \frac{1}{2\pi}\int_{\R}\varphi_{\mu_{\varepsilon}}(x)e^{-s\lambda^2}e^{i\lambda u} d\lambda+
 \frac{\sg(x)}{2\pi A(x)} 
\int_{\R}\Big(
\int_{0}^{|x|}\varphi_{\mu_{\varepsilon}}(z)A(z)dz\Big)\,(-i\lambda+\varepsilon\varrho) e^{-s\lambda^2}e^{i\lambda u}\,d\lambda\hfill\\
=&\!\!:W_\varepsilon^{\text{even}}(s;u,x)+ W_\varepsilon^{\text{odd}}(s;u,x).
\end{align*}

Next we shall prove that $W_\varepsilon^{\text{even}} (s;u,x) >0.$
By Lemma \ref{alt}, we have
$$\displaystyle{\varphi_{\mu_{\varepsilon}}(x)=\int_0^{|x|}K_{\varepsilon}(x,r)\cos(\lambda
r)dr},$$ where the kernel $K_{\varepsilon}(x,s)$  is nonnegative. Thus,
\begin{eqnarray*}
W_\varepsilon^{\text{even}}(s;u,x)&=&\frac{1}{\pi}\int_0^{+\infty}e^{-s\lambda^2}\cos(\lambda u)\,\Big(\int_0^{|x|}K_\varepsilon(x,r)\cos(\lambda r)dr \Big)\,d\lambda\\
&=&\frac{1}{\pi}\int_0^{|x|}K_\varepsilon(x,r)\,\Big(\int_0^{+\infty}e^{-s\lambda^2}\cos(\lambda u) \cos(\lambda r)d\lambda \Big)\,dr\\
&=&\frac{1}{4\sqrt{\pi s}}\int_0^{|x|}K_\varepsilon(x,r)\,\Big(e^{-\frac{(u-r)^2}{4s}}+e^{-\frac{(u+r)^2}{4s}}\Big)\,dr.
\end{eqnarray*}
Using the fact that $r\mapsto K_\varepsilon(x,r)$ is even, we deduce that 
\begin{eqnarray*}
 W_\varepsilon^{\text{even}}(s;u,x)&=& \frac{1}{4\sqrt{\pi s}}\int_{-|x|}^{|x|}e^{-\frac{(u+r)^2}{4s}}K_\varepsilon(x,r)dr\\
&\geq& \frac{e^{-{(|u|+|x|)^2}/{4s}}}{4\sqrt{\pi s}}\int_{-|x|}^{|x|}K_\varepsilon(x,r)dr\\
&=&\frac{e^{-{(|u|+|x|)^2}/ {4s}}}{2\sqrt{\pi s}}\varphi_{i\sqrt{1-\varepsilon^2}\varrho}(x)>0.
\end{eqnarray*}
\end{proof}
  
Now we come to the crucial step in the proof of  Theorem \ref{cent}.

\begin{thm}\label{step} For every $s>0$ and $u,x\in \R,$ we have
$$W_\varepsilon(s;u,x)\geq 0.$$
\end{thm}
\begin{proof} For  $(u, x)\in \R\times\{0\}$, we have 
$$W_\varepsilon(s;u,0)= p_s(u,0)=\frac{e^{-\frac{u^2}{4s}}}{2\sqrt{\pi
s}}>0.$$
For $s>0$ and $(u,x)\in (\R\times \{0\})^{\rm c}$, we will assume that  $W_\varepsilon(s;u,x)$ is not always non-negative. Since $W_\varepsilon(s;u,0)>0$ and
 $ \lim_{\|(u,x)\|\rightarrow+\infty} W_\varepsilon(s;u,x) =0$ (see Lemma \ref{w_ep}.3), then the above assumption implies that 
  the function  $ (u,x)\mapsto  W_\varepsilon(s;u,x)$   admits an absolute minimum $ (u_0,x_0)\in (\R\times\{0\} )^{\rm c} $ such that  $W_\varepsilon(s;u_0,x_0)<0$.
In particular, $W_\varepsilon^{\text{odd}} (s; u_0, x_0) =(W_\varepsilon(s; u_0, x_0)-W_\varepsilon(s; u_0,- x_0))/2 \leq0$. 
We claim that \begin{equation} \label{<0}W_\varepsilon^{\text{odd}} (s; u_0, x_0) <0.
\end{equation}
Indeed, if $W_\varepsilon^{\text{odd}} (s; u_0, x_0)=0,$ then $W_\varepsilon  (s; u_0, x_0)=
W_\varepsilon^{\text{even}} (s; u_0, x_0),$ which is impossible since $W_\varepsilon  (s; u_0, x_0)<0$ while $W_\varepsilon^{\text{even}} (s; u_0, x_0)>0$ (see Lemma \ref{w_eps}).

On the other hand, using the fact that $(u_0,x_0)$ is an absolute minimum, we have 
\begin{align}
\Big(\Lambda_{A,\varepsilon}+ {\partial}_u\Big)W_\varepsilon(s; u_0, x_0)& = \Big(\frac{A'(x_0)}{A(x_0)}+2\varepsilon \varrho\Big) W_\varepsilon^{\text{odd}}(s; u_0, x_0)-\varepsilon\varrho W_\varepsilon(s; u_0, x_0)\label{eq731}\\
& =  -\varepsilon\varrho W_\varepsilon^{\text{even}}(s; u_0, x_0)+\Big(\frac{A'(x_0)}{A(x_0)}+\varepsilon \varrho\Big)
W_\varepsilon^{\text{odd}} (s; u_0, x_0).\label{eq741}
\end{align}

Recall that our assumption  is that $W_\varepsilon(s;u,x)$ is not always non-negative for all $s>0$ and  $(u,x)\in (\R\times \{0\})^{\rm c}.$  We shall use Lemma \ref{w_ep}.2 to prove that this assumption fails.

 \hspace{.5cm}
 \underline {\sf case 1:} For $\varrho=0,$ the identity \eqref{eq731} becomes 
$$ \Big(\Lambda_{A,\varepsilon}+ {\partial}_u\Big)W_\varepsilon(s; u_0, x_0)=\frac{A'(x_0)}{A(x_0)}W_\varepsilon^{\text{odd}}(s; u_0, x_0). $$
Now, Lemma \ref{w_ep}.2 and the inequality \eqref{<0} imply  that  $(A'/A)(x_0)=0,$ which is not true in the light of the hypotheses (H2) and (H4) on $A'/A$ with $\varrho=0.$

 \hspace{.5cm}
 \underline {\sf case 2:} Let $\varrho>0$ and $\varepsilon=0.$ As in the previous case, 
   Lemma \ref{w_ep}.2 and the inequality \eqref{<0} imply   that $(A'/A)(x_0)=0.$ 
However, by the hypothesis (H2) on $A'/A ,$ we have $A'/A(x)\gtreqless  \pm 2\varrho \gtrless 0$ for all $x\gtrless  0.$ Hence our assumption does not hold true.

 \hspace{.5cm}
 \underline {\sf case 3:} Let   $\varrho>0$ and $\varepsilon>0.$ 
 
 \hspace{.7cm}
 \underline {\it subcase 3.1:} Assume that $x_0>0.$ As $W_\varepsilon^{\text{even}} (s; u_0, x_0) >0$ and $W_\varepsilon^{\text{odd}} (s; u_0, x_0) <0,$ it follows from \eqref{eq741} that  $\Big(\Lambda_{A,\varepsilon}+ {\partial}_u\Big)W_\varepsilon(s; u_0, x_0)<0, $ which is absurd by  Lemma \ref{w_ep}.2.

 \hspace{.7cm}
 \underline {\it subcase 3.2:} Assume that $x_0<0.$ We pin down that 
 \begin{equation}\label{leq00} 
 \frac{A'(x_0)}{A(x_0)}+2\varepsilon\varrho\leq-2(1-\varepsilon)\varrho\leq0 .
 \end{equation}
By Lemma \ref{w_ep}.2, we have 
\begin{equation}\label{cont} \Big(\Lambda_{A,\varepsilon}+ {\partial}_u+\varepsilon \varrho \id \Big)W_\varepsilon(s; u_0, x_0)=\varepsilon \varrho W_\varepsilon(s; u_0, x_0) <0,
\end{equation} while, by \eqref{eq731}, \eqref{<0} and \eqref{leq00}, 
$$ \Big(\Lambda_{A,\varepsilon}+ {\partial}_u+\varepsilon \varrho \id \Big)W_\varepsilon(s; u_0, x_0)= \Big(\frac{A'(x_0)}{A(x_0)}+2\varepsilon \varrho\Big) W_\varepsilon^{\text{odd}}(s; u_0, x_0)\geq 0,$$ which contradicts the inequality \eqref{cont}.

 \hspace{.5cm}
 \underline {\sf case 4:} Let  $\varrho>0$ and $\varepsilon<0.$ 

 \hspace{.7cm}
 \underline {\it  subcase 4.1:} Assume that $x_0>0.$ Note that 
 \begin{equation}\label{geq00} 
 \frac{A'(x_0)}{A(x_0)}+2\varepsilon\varrho\geq2(1+\varepsilon)\varrho\geq0.
 \end{equation}
Hence, the identities  \eqref{eq731}, \eqref{<0} and \eqref{geq00} imply  that $ \Big(\Lambda_{A,\varepsilon}+ {\partial}_u\Big)W_\varepsilon(s; u_0, x_0)<0,$ which is absurd by Lemma \ref{w_ep}.2.

 \hspace{.7cm}
 \underline {\it  subcase 4.2:} Assume that $x_0<0.$ On  the one hand, by Lemma \ref{w_ep}.2,  we have
 \begin{equation}\label{contt} \Big(\Lambda_{A,\varepsilon}+ {\partial}_u-\varepsilon \varrho \id \Big)W_\varepsilon(s; u_0, x_0)=-\varepsilon \varrho W_\varepsilon(s; u_0, x_0) <0.
\end{equation} 
On the other hand, since $x_0<0$ we have $A'/A(x_0)<0.$ Thus, by \eqref{eq731} and \eqref{<0}, 
$$ \Big(\Lambda_{A,\varepsilon}+ {\partial}_u-\varepsilon \varrho \id \Big)W_\varepsilon(s; u_0, x_0)=-2\varepsilon \varrho W_\varepsilon^{\text{even}}(s; u_0, x_0)
+\frac{A'(x_0)}{A(x_0)}  W_\varepsilon^{\text{odd}}(s; u_0, x_0)> 0,$$ which contradicts the inequality \eqref{contt}.

This finishes the proof of Theorem \ref{step}.
\end{proof}

Now we are ready to prove the central result of this section. 

\begin{proof}[Proof of Theorem \ref{cent}] 
Let $f$ be a positive function in $\mathcal{D}(\R)$. Proving  that $V_{A, \varepsilon}(f) \geq 0$  is equivalent to  showing that ${}^{\rm t}V_{A, \varepsilon} (f)\geq 0$ (see \eqref{vvv} and \eqref{Vt}).

By \eqref{vtt} we have
\begin{eqnarray*}
\int_{\R}f(x)V_{A, \varepsilon}(p_s(u,.))(x)A(x)dx
&=&
\int_{\R}{}^{\rm t}V_{A, \varepsilon}  f  (x) p_s(x,u)dx\\
&=&
({}^{\rm t}V_{A, \varepsilon}  f \ast_{\text {euc} }q_s) (u),
\end{eqnarray*}
where $q_s(r):=\displaystyle{\frac{e^{-r^2/4s}}{2\sqrt{\pi s}}}$.
Since $f\geq 0$ and $V_{A, \varepsilon}(p_s(u,.))(x)=W_\varepsilon(s;u,x)\geq 0,$ it follows that
$({}^{\rm t}V_{A, \varepsilon}  f \ast_{\text {euc} }q_s) (u)\geq 0$ 
for all $s>0$ and $u,x\in \R.$   Thus 
$$0\leq \lim_{s\rightarrow 0}\, ({}^{\rm t}V_{A, \varepsilon}  f \ast_{\text {euc} }q_s) (u)={}^{\rm t}V_{A, \varepsilon}  f(u).$$

\end{proof}
\section{Fourier transform of  $L^p$-Schwartz spaces}\label{sec8}

 Assume that $-1\leq \varepsilon\leq 1.$ For $f\in L^1(\R, A(x)dx)$ put \index{$\mathcal F_{A,\varepsilon}(f)$}
\begin{equation}
\mathcal F_{A,\varepsilon} f(\lambda)=\int_{\R} f(x) \Psi_{A, \varepsilon}(\lambda,-x) A(x)dx.
\end{equation}
To state its alleged inverse transform, let us introduce the following  Plancherel measure 
\index{$\pi_\varepsilon(d\lambda)$}
\begin{equation}\label{plancherel}
 \pi_\varepsilon(d\lambda)= 
\frac{|\lambda|}{\sqrt{\lambda^2-(1-\varepsilon^2 )\varrho^2}\;
\big |c\big(\sqrt{\lambda^2-(1-\varepsilon^2)\varrho^2}\big)\big |^2}
{\bf 1}_{\mathbb{R}\,\backslash\, ]-\sqrt {1-\varepsilon^2}\varrho,\sqrt {1-\varepsilon^2}\varrho[}(\lambda)
d\lambda,
\end{equation}
where $c$ is the Harish-Chandra's function associated with  the second order differential   operator $\Delta $ (see Section \ref{sec2} for more details on the $c$-function).  Below $\check f (x):=f(-x).$ 

 
\begin{thm}\label{invth} Let $f$ be a smooth function with compact support  on $\R$. Then
\begin{enumerate}[\upshape 1)]
\item (Inversion formula) 
\begin{equation}\label{invfor}f(x)={1\over 4}   \int_{\R}  \mathcal F_{A,\varepsilon}( f)(\lambda)   \Psi_{A, \varepsilon}({\lambda},x) \Big( 1-{{\varepsilon \varrho}\over {i\lambda}}\Big)   \pi_\varepsilon(d\lambda).\end{equation}
\item (Plancherel formula)
\begin{equation}\label{pl2}
\int_\R |f(x)|^2 A(x)dx={1\over 4} 
\int_{\R} \mathcal F_{A,\varepsilon}(f)(\lambda) \overline{\mathcal F_{A,\varepsilon}(\check f)(-\lambda)} \Big( 1-{{\varepsilon \varrho}\over {i\lambda}}\Big) \pi_\varepsilon(d\lambda).
\end{equation}
We may rewrite \eqref{pl2} for two smooth and compactly supported functions $f$ and $g$ as 
\begin{equation}
\int_\R f(x) g(-x) A(x)dx={1\over 4}  \int _\R \mathcal F_{A,\varepsilon} (f)(\lambda) \mathcal F_{A,\varepsilon}(g)(\lambda) 
\Big( 1-{{\varepsilon \varrho}\over {i\lambda}}\Big) \pi_\varepsilon(d\lambda).
\end{equation}

\end{enumerate}
\end{thm}

\begin{proof}  For the sake of completeness, we provide a detailed proof.\\
1)  Below $Jh(x):=\int_{-\infty}^x h(t)dt.$ 
Using the   superposition \eqref{2.2} of the eigenfunction $\Psi_\varepsilon(\lambda,x),$ we obtain
$$ 
\mathcal F_{A,\varepsilon} f(\lambda) =2\mathcal F_{\Delta} (f_e)(\mu_\varepsilon)+ 2(i\lambda+\varepsilon \varrho)  \mathcal F_{\Delta} ( Jf_o)(\mu_\varepsilon),\label{7.3}
$$
where $F_{\Delta}$ is as in \eqref{Fdelta}.
By the inversion formula \eqref{inv1} for the Ch\'ebli transform $\mathcal F_{\Delta}$ we deduce that 
\begin{equation}\label{f4}
f(x)=\int_{\R^+} \Big\{ \mathcal F_{\Delta}(f_e)(\mu_\varepsilon) \varphi_{\mu_\varepsilon}(x) +\mathcal F_{\Delta}(Jf_o)(\mu_\varepsilon) \varphi_{\mu_\varepsilon}'(x)\Big\} \: \pi(d\mu_\varepsilon).
\end{equation}
Now, let us express  $\varphi_{\mu_\varepsilon}$ and  $\varphi_{\mu_\varepsilon}' $ in terms of $\Psi_{A,\varepsilon}$ as follows
$$\varphi_{\mu_\varepsilon}(x) ={1\over 2} \Big(\Psi_{A,\varepsilon}({-\lambda},-x)+\Psi_{A,\varepsilon}({-\lambda},x)\Big),\qquad 
\varphi_{\mu_\varepsilon}'(x) ={{i\lambda+\varepsilon \varrho}\over 2} \Big(\Psi_{A,\varepsilon}({-\lambda},-x)-\Psi_{A,\varepsilon}({-\lambda},x)\Big).$$ 
Consequently, formula \eqref{f4} becomes   
\begin{eqnarray}
f(x)
&=&{1\over 2}\int_{\R^+}  \Psi_{A,\varepsilon}({-\lambda},-x)\Big\{ \mathcal F_{\Delta}(f_e)(\mu_\varepsilon)  + (i\lambda+\varepsilon \varrho) \mathcal F_{\Delta}(Jf_o)(\mu_\varepsilon) \Big\} \: \pi(d\mu_\varepsilon)\nonumber \\
&&+{1\over 2}\int_{\R^+} \Psi_{A,\varepsilon}({-\lambda},x)\Big\{ \mathcal F_{\Delta}(f_e)(\mu_\varepsilon)  - (i\lambda+\varepsilon \varrho) \mathcal F_{\Delta}(Jf_o)(\mu_\varepsilon) \Big\} \:\pi(d\mu_\varepsilon)\nonumber \\
&=&{1\over 4} \int_{\R^+} \Big\{\Psi_{A,\varepsilon}({-\lambda},-x) \mathcal F_{A,\varepsilon}(f )(\lambda)+ \Psi_{A,\varepsilon}({-\lambda},x) \mathcal F_{A,\varepsilon}(\check f)(\lambda) \Big\} \: \pi(d\mu_\varepsilon) .\label{eqna}
\end{eqnarray}
Further, it is easy to check that 
\begin{equation}\label{claim4}
\Psi_{A,\varepsilon}(\lambda,x)=  \Big( 1+{{\varepsilon \varrho}\over {i\lambda}}\Big)\Psi_{A,\varepsilon}({-\lambda},-x)-{{\varepsilon \varrho}\over {i\lambda}} \Psi_{A,\varepsilon}({\lambda},-x),
\end{equation}
and therefore 
\begin{eqnarray}\label{8.8}
 \mathcal F_{A,\varepsilon}(\check f)(\lambda) =\Big( 1+{{\varepsilon \varrho}\over {i\lambda}}\Big) \mathcal F_{A,\varepsilon}(f)(-\lambda)-{{\varepsilon \varrho}\over {i\lambda}} \mathcal F_{A,\varepsilon}(f)(\lambda).
\end{eqnarray} 
In view of \eqref{claim4} and \eqref{8.8}  we obtain
\begin{align} 
\int_{\R^+}   \mathcal F_{A,\varepsilon}(f )(\lambda) \Psi_{A,\varepsilon}({-\lambda},-x)  \: \pi(d\mu_\varepsilon)
 =& \int_{\R^+}  \Psi_{A,\varepsilon}({\lambda},x)  \mathcal F_{A,\varepsilon}( f)(\lambda) \:\Big( 1-{{\varepsilon \varrho}\over {i\lambda}}\Big)  \:\pi(d\mu_\varepsilon)\nonumber\\
& +\int_{\R^+}  \Psi_{A,\varepsilon}({-\lambda},x)  \mathcal F_{A,\varepsilon}( f)(\lambda) \: \Big({{\varepsilon \varrho}\over {i\lambda}}\Big)\:  \pi(d\mu_\varepsilon),\label{78a}
\end{align}
and 
\begin{align}
\int_{\R^+} \mathcal F_{A,\varepsilon}(\check f)(\lambda) \Psi_{A,\varepsilon}({-\lambda},x)  \:  \pi(d\mu_\varepsilon)
 = &\int_{\R^+}  \Psi_{A,\varepsilon}({-\lambda},x)  \mathcal F_{A,\varepsilon}( f)(-\lambda) \:\Big( 1+{{\varepsilon \varrho}\over {i\lambda}}\Big)\:  \pi(d\mu_\varepsilon)\nonumber\\
& +\int_{\R^+}  \Psi_{A,\varepsilon}({-\lambda},x)   \mathcal F_{A,\varepsilon}(  f)(\lambda) \:\Big(-{{\varepsilon \varrho}\over {i\lambda}}\Big)\: \pi(d\mu_\varepsilon).  \label{78b}
\end{align}
By substituting \eqref{78a} and \eqref{78b} into \eqref{eqna}, we get the inversion formula  \eqref{invfor}.\\
2) On the one hand, using  the fact that 
$\overline{\Psi_{A,\varepsilon}(\lambda, x)}= \Psi_{A,\varepsilon}(- \lambda, x)$ for $\lambda\in \R,$
we have 
$$\overline{ \mathcal F_{A,\varepsilon}(\check g)(\lambda)}=\int_{\R} \overline {g(x)} \Psi_\varepsilon({-\lambda},x) A(x)dx.$$
Applying   the identity \eqref{eqna} for $f$  implies 
\begin{equation}\label{plpar}
 \int_\R f(x)   \overline{g(x)}  A(x)dx
={1\over 4} \int_{\R^+} \Big\{ \mathcal F_{A,\varepsilon}( f)(\lambda) \overline{ \mathcal F_{A,\varepsilon}( g)(\lambda)} +\mathcal F_{A,\varepsilon}( \check f)(\lambda) \overline{ \mathcal F_{A,\varepsilon}( \check g)(\lambda)} \Big\}   \pi_\varepsilon(d\lambda).
\end{equation}
 On the other hand, from  \eqref{claim4}   it  follows that 
\begin{equation}\label{check}
\overline{\mathcal F_{A,\varepsilon}(\check f)(-\lambda)}= \Big( 1+{{\varepsilon \varrho}\over {i\lambda}}\Big) \overline{\mathcal F_{A,\varepsilon}(f)(\lambda)}-{{\varepsilon \varrho}\over {i\lambda}}\overline{ \mathcal F_{A,\varepsilon}(f)(-\lambda)}.
\end{equation}
Hence
\begin{multline}\label{eq1}
\Big( 1-{{\varepsilon \varrho}\over {i\lambda}}\Big)\mathcal F_{A,\varepsilon}( f)(\lambda) \overline{\mathcal F_{A,\varepsilon}(\check f)(-\lambda)}\\= \Big( 1+{{\varepsilon^2 \varrho^2}\over {\lambda^2}}\Big) \vert \mathcal F_{A,\varepsilon}(f)(\lambda)\vert^2
-{{\varepsilon \varrho}\over {i\lambda}}\Big( 1-{{\varepsilon \varrho}\over {i\lambda}}\Big) \mathcal F_{A,\varepsilon}( f)(\lambda) \overline{ \mathcal F_{A,\varepsilon}(f)(-\lambda)}.
\end{multline}
Now let us rewrite \eqref{check} as
$$\overline{\mathcal F_{A,\varepsilon}(f)(-\lambda)}={{i\lambda}\over{i\lambda-\varepsilon \varrho}}  \overline{\mathcal F_{A,\varepsilon}(\check f)(\lambda)}+{{\varepsilon \varrho}\over{-i\lambda+\varepsilon \varrho}}  
\overline{\mathcal F_{A,\varepsilon}( f)(\lambda)}.$$
Hence
$$\mathcal F_{A,\varepsilon}(f)(\lambda) \overline{\mathcal F_{A,\varepsilon}(f)(-\lambda)}={{i\lambda}\over{i\lambda-\varepsilon \varrho}} \mathcal F_{A,\varepsilon}(f)(\lambda)  \overline{\mathcal F_{A,\varepsilon}(\check f)(\lambda)}+{{\varepsilon \varrho}\over{-i\lambda+\varepsilon \varrho}}  
\vert{\mathcal F_{A,\varepsilon}( f)(\lambda)}\vert^2,$$
which implies 
\begin{multline}
\Big(-{{\varepsilon \varrho}\over {i\lambda}} \Big)\Big({{i\lambda-\varepsilon \varrho}\over{i\lambda}}\Big) \mathcal F_{A,\varepsilon}(f)(\lambda) \overline{\mathcal F_{A,\varepsilon}(f)(-\lambda)}\\
=-
\Big({{\varepsilon \varrho}\over {i\lambda}} \Big) \mathcal F_{A,\varepsilon}(f)(\lambda)  \overline{\mathcal F_{A,\varepsilon}(\check f)(\lambda)}-\Big({{\varepsilon^2 \varrho^2}\over {\lambda^2}} \Big) \vert{\mathcal F_{A,\varepsilon}( f)(\lambda)}\vert^2.
\end{multline}
Thus, \eqref{eq1} becomes 
 \begin{equation}\label{eq11}
\Big(1-{{ \varepsilon \varrho}\over{i\lambda}}\Big) \mathcal F_{A,\varepsilon}(f)(\lambda) \overline{\mathcal F_{A,\varepsilon}(\check f)(-\lambda)} 
= |\mathcal F_{A,\varepsilon}(f)(\lambda)|^2 -{{\varepsilon \varrho}\over{i\lambda}} \mathcal F_{A,\varepsilon}(f)(\lambda)  \overline{\mathcal F_{A,\varepsilon}(\check f)(\lambda)}.
\end{equation}   
This is the key identity towards  the Plancherel formula \eqref{pl2}.  One more thing, from \eqref{eq11} we  also   have
\begin{equation}\label{eq2}
\Big(1-{{ \varepsilon \varrho}\over{i\lambda}}\Big)\overline{ \mathcal F_{A,\varepsilon}(\check f)(-\lambda)} \mathcal F_{A,\varepsilon}( f)(\lambda) =  |\mathcal F_{A,\varepsilon}(\check f)(-\lambda)|^2 -{{\varepsilon \varrho}\over{i\lambda}} \overline{\mathcal F_{A,\varepsilon}(\check f)(-\lambda)}  {\mathcal F_{A,\varepsilon}(  f)(-\lambda)}.
\end{equation}
Indeed, we obtain \eqref{eq2} in three steps: 
\begin{itemize}
\item[{1.}]  replace  $f$ by $\check f $ in \eqref{eq11}.
\item[{2.}]  substitute $\lambda$ by $-\lambda$ in the resulting identity from step 1.
 \item[{3.}]   take the complex conjugates in the resulting identity from  step 2.
\end{itemize}
By putting  the pieces together we arrive at 
\begin{eqnarray*}
&&\int_{\R} \mathcal F_{A,\varepsilon}(f)(\lambda) \overline{\mathcal F_{A,\varepsilon}(\check f)(-\lambda)} \: \Big( 1-{{\varepsilon \varrho}\over {i\lambda}}\Big) \:\pi_\varepsilon(d\lambda)\\
&=&\int_{\R^+} |\mathcal F_{A,\varepsilon}(f)(\lambda)|^2   \pi_\varepsilon(d\lambda)-\int_{\R^+}  \mathcal F_{A,\varepsilon}(f)(\lambda) \overline{\mathcal F_{A,\varepsilon}(\check f)(\lambda)}  \:\Big( {{\varepsilon \varrho}\over {i\lambda}} \Big)\: \pi_\varepsilon(d\lambda)\\
&&+\int_{\R^-} |\mathcal F_{A,\varepsilon}(\check f)(-\lambda)|^2   \pi_\varepsilon(d\lambda)-\int_{\R^-} \mathcal F_{A,\varepsilon}(f)(-\lambda) \overline{\mathcal F_{A,\varepsilon}(\check f)(-\lambda)} \: \Big({{\varepsilon \varrho}\over {i\lambda}} \Big)\:  \pi_\varepsilon(d\lambda)\\
&=&\int_{\R^+} \Big\{ |\mathcal F_{A,\varepsilon}(f)(\lambda)|^2  + |\mathcal F_{A,\varepsilon}(\check f)(\lambda)|^2\Big\} \:  \pi_\varepsilon(d\lambda),
\end{eqnarray*}
which compares very well with $4 \Vert f\Vert_{L^2(\R, A(x)dx)}$ (see \eqref{plpar}).
\end{proof}
\begin{rems}
\begin{enumerate}[\upshape 1)]
\item For $\varepsilon =1,$ the   Plancherel formula \eqref{pl2}  corrects Theorem 5.13 in \cite{Fitou}
(stated without a proof).  
\item For  $\varepsilon=0$ we can prove the following stronger versions of  the inversion   and the Plancherel formulas: 
\begin{enumerate}[\upshape (i)]
\item  If $f\in L^1(\R, A(x)dx)$ and $\mathcal F_{A,0}(f)\in L^1(\R,  \pi_0(d\lambda))$ then
$$
f(x)={1\over 4} \int_{\R} \mathcal F_{A,0}(f)(\lambda)\, \Psi_{A,0}(\lambda, x) \, \pi_0(d\lambda)  \qquad
\text{\rm almost everywhere.}
$$
\item  If $f\in L^1\cap L^2(\R, A(x)dx),$ then $\mathcal F_{A,0} f \in L^2(\R, \pi_0(d\lambda))$ and $\Vert \mathcal F_{A,0} f \Vert_{L_\lambda^2}=2\Vert f\Vert_{L^2_x}.$ 
\item There exists a unique   isometry on $L^{2}(\R, A(x)dx)$ that coincides with $(1/2) \mathcal F_{A,0}$ on $L^1\cap L^2(\R, A(x)dx).$
\end{enumerate}
\end{enumerate}
\end{rems}

The following lemma is needed in the proof of the Paley-Wiener theorem below. 
\begin{lema}\label{support} For $R>0,$ denote by $\mathcal D_R(\R)$ the space of smooth functions with support inside  $[-R,R]$.
Then,  $f\in \mathcal D_R(\R)$    if and 
only if  $\,{}^{\rm t}V_{A, \varepsilon} f\in \mathcal D_R(\R).$
 \end{lema}
 \begin{proof} The direct statement follows from \eqref{Vt}. The converse direction is more involved. On the  one hand, one can prove that 
\begin{equation}\label{expVt}
{}^{\rm t}V_{A, \varepsilon} ^{-1} \, g(y)= \mathcal A_\varepsilon^{-1} g_e(y)+\Big( \varepsilon \varrho +{d\over{dy}}\Big) \mathcal A_\varepsilon^{-1} (J g_o)(y),
\end{equation}
where  $Jh(x):= \int_{-\infty}^x h(t)dt$ and $\mathcal A_\varepsilon^{-1}=\mathcal A^{-1}\circ {}^{\rm t}\mathcal E_\varepsilon$ (see \eqref{center2}). On the other hand, from \eqref{e-t11} and \cite[Lemma 4.10]{BX1} it follows that if $ g_e \in \mathcal D_R(\R)$ then $\mathcal A_\varepsilon^{-1} g_e  \in \mathcal D_R(\R).$  Further, on may check that $g_o\in \mathcal D_R(\R)$ if and only if $J g_o\in   \mathcal D_R(\R).$ Thus, as $J g_o$ is an even function, it follows from above that $\mathcal A_\varepsilon^{-1} (J g_o)  \in \mathcal D_R(\R).$
 \end{proof}
 
 Let $ { PW_R(\C)}$ be the space of entire functions $h$ on $\C$  which are of exponential type and rapidly decreasing, i.e. 
 \index{$PW(\C)$}
 \begin{equation}\label{R}
 \exists \:R>0, \,\forall t\in \N,\quad  \sup_{\lambda\in \C} (|\lambda |+1)^t e^{-R|\Im \lambda |}\; |h(\lambda)|<\infty.
 \end{equation}

 \begin{thm}\label{pw} Assume that $-1\leq \varepsilon\leq 1.$ The Fourier transform $\mathcal F_{A,\varepsilon}$ is a linear  isomorphism
  between   $\mathcal D_R(\R)$   and the space of all entire functions $h$ on $\C$ satisfying 
 \eqref{R}.
 \end{thm}
 \begin{proof}
 The proof is standard. We shall only  indicate how to proceed towards  the statement. On the one hand, the Fourier transform 
 $\mathcal F_{A,\varepsilon}$ can be written as $\mathcal F_{A,\varepsilon}(f)=\mathcal F_{\rm euc} \circ {}^{\rm t}V_{A, \varepsilon}(\check f),$ where $\mathcal F_{\rm euc}$
  is the Euclidean  Fourier transform and $ {}^{\rm t}V_{A, \varepsilon}$ is the intertwining operator \eqref{Vt}. Indeed, in view of Mehler's type Theorem \ref{Mh}, we have 
  \begin{eqnarray*}
 \mathcal F_{A,\varepsilon}(f)(\lambda)
 &=&\int_{\R} f(x) \Psi_{A, \varepsilon}(\lambda, -x)A(x)dx\\
& =&\int_\R \check f(x) V_{A, \varepsilon}(e^{i\lambda \,\cdot})(x) A(x)dx\\
&  =&\int_\R {}^{\rm t}V_{A, \varepsilon}\, \check f(x)\, e^{i\lambda \,x}  dx.
\end{eqnarray*}
  Now, in view of Lemma \ref{support}, appealing to the Paley-Wiener theorem for the Euclidean 
  Fourier transform $\mathcal F_{\rm euc}$ we get the desired statement. 
  \end{proof}

 For $-1\leq \varepsilon\leq 1$ and  $0<p\leq {2\over {1+  \sqrt{1-\varepsilon^2}}},$ set \index{$\vartheta_{p,\varepsilon} $} 
 $\vartheta_{p,\varepsilon}:= {2\over p} -1 -\sqrt{1-\varepsilon^2}.$ Observe that $1\leq {2\over {1+  \sqrt{1-\varepsilon^2}}} \leq 2.$  
 We introduce the   tube domain  
   $$\C_{p,\varepsilon}:=\{ \lambda\in \C\;|\; |\Im \lambda | \leq \varrho \;\vartheta_{p,\varepsilon} \}.$$ 
 For $\vartheta_{p,\varepsilon}=0$ or $\varrho=0,$ the domain 
 $\C_{p,\varepsilon}$ reduces to $\R.$ 
 
 \begin{prop}\label{bounded} For all  $\lambda \in \C_{1,\varepsilon},$ the function $\lambda\mapsto \Psi_{A,\varepsilon}(\lambda,x)$ is bounded for all $x\in \R.$ 
 \end{prop}
 \begin{proof} Let $R >0$ be arbitrary but fixed and let $\R_{1,\varepsilon}:=\{\nu\in \R\;|\; |\nu|\leq \varrho\,(1-\sqrt{1-\varepsilon^2})\}.$ Applying the maximum modulus principle together with the fact that $|\Psi_{A,\varepsilon}(\lambda, x)|\leq \Psi_{A,\varepsilon}(i \Im \lambda, x)$ in the domain 
$[-R,R]+i\R_{1, \varepsilon}$ implies that the maximum of $|\Psi_{A,\varepsilon}(\lambda, x)|$ is obtained when $\lambda$ belongs to the boundary of  $i\R_{1, \varepsilon}.$ That is 
$\lambda=i\eta$ with $|\eta|=\varrho\,(1-\sqrt{1-\varepsilon^2}). $
Now, recall that  $\Psi_{A,\varepsilon}(i\eta, x)+ \Psi_{A,\varepsilon}(i\eta, -x) =2\varphi_{\mu_\varepsilon}(x)$ when $\varepsilon\not = 0, \pm 1,$ and  $\Psi_{A,\varepsilon}(i\eta, x)+ \Psi_{A,\varepsilon}(i\eta, -x) =2$ when  $\varepsilon  = 0, \pm 1.$ 
The parameter   $\mu_\varepsilon$ satisfies  $\mu_\varepsilon^2=\lambda^2- (1-\varepsilon^2)\varrho^2=-\varrho^2\big( 1-2\sqrt{1-\varepsilon^2}\{1-\sqrt{1-\varepsilon^2}\}\big)\leq 0,$ and therefore $\mu_\varepsilon\in i\R$ with $|\mu_\varepsilon|\leq \varrho. $  
 Using the fact that $\Psi_{A,\varepsilon}(i\eta, x) >0$ for all $x\in \R,$ together with the fact that 
 $\varphi_{\mu_\varepsilon}(x) \leq 1$ for $\mu_\varepsilon$ is as above (see Lemma \ref{BX}.3),  it follows that  $\Psi_{A,\varepsilon}(i\eta, x) \leq 2$ for all $x\in \R$ and $-1\leq \varepsilon\leq 1.$
 \end{proof}
 \begin{coro}\label{LRL1} Let $f\in L^1(\R, A(x)dx).$ Then the following holds.
  \begin{enumerate}[\upshape 1)]
 \item The Fourier transform $\mathcal F_{A,\varepsilon} (f) (\lambda)$ is well defined for all $\lambda\in  {\C}_{1,\varepsilon}.$ Moreover,
$$| \mathcal F_{A,\varepsilon} (f) (\lambda)| \leq 2 \Vert f\Vert_1,\qquad \lambda\in   {\C}_{1,\varepsilon}.
$$ 
 \item The function $\mathcal F_{A,\varepsilon} (f)$ is holomorphic on  $\mathring{\C}_{1,\varepsilon},$ the interior of ${\C}_{1,\varepsilon}.$
  \item { \rm (Riemann-Lebesgue lemma)} 
\begin{equation}\label{limit1} 
\lim\limits_{   \lambda\in   {\C}_{1,\varepsilon} , |\lambda |\rightarrow \infty  } 
 | \mathcal F_{A,\varepsilon} (f) (\lambda)| =0.\end{equation}
\end{enumerate}
 \end{coro}
  \begin{proof}
 The first two statements are direct consequences of Proposition \ref{bounded}, the fact that 
$\Psi_{A, \varepsilon}(\lambda, \cdot) $ is holomorphic in $\lambda,$ and Morera's theorem. For the Riemann-Lebesgue lemma, a classical proof for the Euclidean Fourier transform carries over. More precisely, assume that $f\in \mathcal D(\R)$ (the space of smooth functions with compact support on $\R$).     Now using the Paley-Wiener Theorem \ref{pw} to conclude that the limit \eqref{limit1} holds for test functions; the general case then follows from the fact that $\mathcal D(\R)$ is dense in $L^1(\R, A(x)dx).$
 \end{proof}

 Next we discuss some properties of the Fourier transform $\mathcal F_{A,\varepsilon}$ on $L^p(\R, A(x)dx)$ with $p>1.$ 
\begin{lema}\label{LRL2}  Let $f\in L^p(\R, A(x)dx)$ with $1<p\leq {2\over {1+  \sqrt{1-\varepsilon^2}}}.$ Then the following holds.
 \begin{enumerate}[\upshape 1)]
 \item The Fourier transform $\mathcal F_{A,\varepsilon} (f) (\lambda)$ is well defined for all $\lambda$ in  
 $\mathring{\C}_{p,\varepsilon},$ the interior of  $\C_{p,\varepsilon}.$ Moreover,
$$| \mathcal F_{A,\varepsilon} (f) (\lambda)| \leq c \Vert f\Vert_p,\qquad \lambda\in  \mathring{\C}_{p,\varepsilon}.
$$ 
 \item The function $\mathcal F_{A,\varepsilon} (f)$ is holomorphic on  $\mathring{\C}_{p,\varepsilon}.$
  \item { \rm (Riemann-Lebesgue lemma)} 
\begin{equation}\label{limit} 
\lim\limits_{   \lambda\in  \mathring{\C}_{p,\varepsilon} , |\lambda |\rightarrow \infty  } 
 | \mathcal F_{A,\varepsilon} (f) (\lambda)| =0.\end{equation}
\end{enumerate}
 \end{lema}
 \begin{proof}
 The first two statements follow easily from the estimate 
$$ \Psi_{A, \varepsilon}(\lambda, x)\leq \begin{cases}
c (|x|+1)e^{|\Im \lambda | \,|x|} e^{-\varrho |x|(1-\sqrt{1-\varepsilon^2})}&\text{for}\,\varrho>0\\
c e^{|\Im \lambda | \,|x|},&\text{for}\,\varrho=0
\end{cases}
$$ the fact that $A(x)\leq c |x|^\beta \, e^{2\varrho |x|} $  (a consequence of the hypothesis   (H4)  on   Ch\'ebli's function $A$), the fact that 
$\Psi_{A, \varepsilon}(\lambda, \cdot) $ is holomorphic in $\lambda,$ and Morera's theorem.    The Riemann-Lebesgue lemma is established exactly as for \eqref{limit1} by approximating any function in $L^p(\R, A(x)dx)$ by  compactly supported smooth functions  for all $1\leq p<\infty.$
 \end{proof}

{\begin{thm}\label{injderniermin} The Fourier transform $\mathcal F_{A,\varepsilon}$ is injective on
   $L^p(\R, A(x)dx)$  for $1\leq p\leq  {2\over {1+  \sqrt{1-\varepsilon^2}}}.$ 
  \end{thm}
  \begin{proof} Take $q$ such that $p+q=pq.$  For $f\in L^p(\R, A(x)dx)$ et $g \in \mathcal D(\R)$ we have the inequalities  
$$
\big | (f, g)_A\big | :=\big | \int_\R f(x) g(-x) A(x) dx \big | \leq \Vert f\Vert_{L^p_x} \Vert g\Vert_{L^q_x}
$$
  and 
\begin{eqnarray} \label{RL}
\big | (\mathcal F_{A,\varepsilon} (f),\mathcal F_{A,\varepsilon} (g) )_{\pi_\varepsilon}\big | &:= &\big | \int_\R 
  \mathcal F_{A,\varepsilon} (f)(\lambda) \mathcal F_{A,\varepsilon} (g)(\lambda) 
   \Big( 1-{{\varepsilon \varrho} \over { i\lambda}} \Big) \pi_\varepsilon(d\lambda) \big | \nonumber\\
&   \leq& \Vert \mathcal F_{A,\varepsilon} (f)\Vert_{L^\infty_\lambda} \Vert \mathcal F_{A,\varepsilon}(g)\Vert_{L_\lambda^1} 
   \leq c \Vert f\Vert_{L^p_x}  \Vert \mathcal F_{A,\varepsilon}(g)\Vert_{L_\lambda^1} . 
   \end{eqnarray}
Above we have used Corollary \ref{LRL1} and Lemma  \ref{LRL2} to get  \eqref{RL}. 
   Therefore the mapping $f\mapsto (f,g)_A$ and $f\mapsto (\mathcal F_{A,\varepsilon} (f), \mathcal F_{A,\varepsilon} (g))_{\pi_\varepsilon}$ are continuous functionals on $L^p(\R, A(x)dx) .$  Now 
    $(f,g)_A= (\mathcal F_{A,\varepsilon} (f), \mathcal F_{A,\varepsilon} (g))_{\pi_\varepsilon}$  for all $f \in \mathcal D(\R)$ and by continuity for all $f\in L^p(\R, A(x)dx) .$   Assume that $f\in L^p(\R, A(x)dx) $ and that $\mathcal F_{A,\varepsilon} (f) =0,$ then for all $g\in \mathcal D(\R)$  we have $(f,g)_A =(\mathcal F_{A,\varepsilon} (f), \mathcal F_{A,\varepsilon} (g))_{\pi_\varepsilon} =0$ and therefore $f=0.$ 
  \end{proof}
  
  }

For $-1\leq \varepsilon\leq 1$ and $0<p\leq {2\over {1+  \sqrt{1-\varepsilon^2}}},$ let $\mathcal S_p(\R)$ be the  space consisting of all functions 
 $f\in C^\infty (\R)$ such that  \index{$\mathcal S_p(\R)$}
 \index{$\sigma_{s,k }^{(p)}(f)$}
\begin{equation}\label{sigma}
 \sup_{x\in \R} (|x|+1)^s \;  \varphi_0(x)^{-2/p}  \;| f^{(k)}(x)| <\infty,
 \end{equation} for any $s\in \N$ and any $k\in \N.$ The topology of $\mathcal S_p(\R)$ is defined by 
 the seminorms
 $$\sigma_{s,k}^{(p)}(f)= \sup_{x\in \R} (|x|+1)^s \;\varphi_0(x)^{-2/p}\; | f^{(k)}(x)| .$$
 We pin down that $\mathcal S_p(\R)$ is a dense subspace of $L^q(\R, A(x)dx)$
  for $p\leq q<\infty,$ while it is not contained in $L^q(\R, A(x)dx)$ for $0<q<p.$
  
  The following facts are standard; see for instance \cite[Appendix A]{DelT}.
  \begin{lema} 
  \begin{enumerate} [\upshape 1)]
  \item $\mathcal S_p(\R)$ is a Fr\'echet space with respect to the seminorms $\sigma_{s,k}^{(p)}.$
  \item $\mathcal D(\R)$ is a dense subspace of $\mathcal S_p(\R). $
  \end{enumerate}
  \end{lema}

Recall from above the tube domain 
 $$\C_{p,\varepsilon}:=\{ \lambda\in \C\;|\; |\Im \lambda| \leq\varrho \,\vartheta_{p,\varepsilon} \},$$ where  
 $\vartheta_{p,\varepsilon}= {2\over p} -1 -\sqrt{1-\varepsilon^2}.$
 
 The Schwartz space $ \mathcal S(\C_{p,\varepsilon})$ consists of  all complex valued
  functions $h$ that are analytic in the interior
 of $\C_{p,\varepsilon},$ and such that $h$ together with all  its derivatives extend 
continuously to $\C_{p,\varepsilon}$ and satisfy \index{$\mathcal S(\C_{p,\varepsilon})$}
\begin{equation}\label{var}
 \sup_{\lambda\in \C_{p,\varepsilon}} (|\lambda|+1)^t\; | h^{(\ell)} (\lambda)| <\infty,
 \end{equation} 
 for any $t\in \N$ and any $\ell\in \N.$
 The topology of $\mathcal S(\C_{p,\varepsilon})$ is defined by the seminorms \index{$\tau_{t,\ell }^{(\vartheta_{p,\varepsilon})}(f)$}
\begin{equation}\label{9.5}
\tau_{t,\ell}^{(\vartheta_{p,\varepsilon})} (h):=\sup_{\lambda\in \C_{p,\varepsilon}} (|\lambda|+1)^t \; | h^{(\ell)} (\lambda)| .
\end{equation}
  For $\vartheta_{p,\varepsilon} =0$ or $\varrho=0,$   $\mathcal S(\C_{p,\varepsilon})$ is the classical Schwartz space on $\R.$ 
 By \cite[Lemma 4.17]{BX1} the Paley-Wiener space $ PW(\C)$ is dense in the Schwartz space 
$ \mathcal S(\C_{p,\varepsilon}).$
 \begin{lema}\label{inje} The Fourier transform $\mathcal F_{A,\varepsilon} $ maps $\mathcal S_p(\R)$ continuously into  
 $ \mathcal S(\C_{p,\varepsilon})$ and is injective. 
 \end{lema}
 \begin{proof} Let $f\in \mathcal S_p(\R)$. For $\lambda\in \C_{p,\varepsilon}$ we have  
\begin{eqnarray*}
 \left\vert \mathcal F_{A,\varepsilon} (f)(\lambda)\right|
 &\leq& \int_\R |f(x)| \;|\Psi_{A, \varepsilon}(\lambda, -x) |\; A(x)dx \\
& \leq &\int_\R |f(x)| \;  \varphi_0(x)^{-2/p}  \;   \varphi_0(x)^{2/p}  \;\Psi_{A, \varepsilon}(0, -x)  e^{|\Im \lambda| |x|} A(x)dx \\
&  \leq &c_1 \int_\R |f(x)| \;  \varphi_0(x)^{-2/p}  \;     (|x|+1)^{2/p+1} e^{-2\varrho |x|  }   A(x)dx .
 \end{eqnarray*}
Under the hypothesis   (H4)  on   Ch\'ebli's function $A,$ there exists a $\beta>0$ such that 
 $$A(x)\leq c |x|^\beta e^{2\varrho |x|}. $$
 Hence,  
  $$\left\vert \mathcal F_{A,\varepsilon} (f)(\lambda)\right| \leq  c_2 \int_\R |f(x)| \;  \varphi_0(x)^{-2/p}  \;  (|x|+1)^{2/p+1}  |x|^\beta dx <\infty.$$
 This proves that $\mathcal F_{A,\varepsilon} (f)$ is well defined for all $f\in  \mathcal S_p(\R)$ when  $-1\leq \varepsilon\leq 1$ and $0<p\leq {2\over {1+  \sqrt{1-\varepsilon^2}}}.$ Moreover, since the map $\lambda\mapsto \Psi_{A, \varepsilon}(\lambda, x)$ is holomorphic on $\C,$ it follows that for all $f\in  \mathcal S_p(\R)$, the function  $\mathcal F_{A,\varepsilon}(f)$  is analytic in the interior of $\C_{p,\varepsilon},$ and continuous on $\C_{p,\varepsilon}.$ Furthermore, by Theorem \ref{estder}, we have 
 $$  \left\vert \mathcal F_{A,\varepsilon}(f)(\lambda)^{(k)}\right| \leq  c_3 \int_\R |f(x)|  \;  \varphi_0(x)^{-2/p}  \; (|x|+1)^{2/p+k+1}    |x|^\beta dx<\infty. $$
   Thus, all derivatives of $\mathcal F_{A,\varepsilon}(f)$ extend continuously to $\C_{p,\varepsilon}.$ Next, we will prove that given a continuous 
   seminorm $\tau$ on  $ \mathcal S(\C_{p,\varepsilon})$, there exists a continuous seminorm $\sigma$ on $\mathcal S_p(\R)$ such that 
 $$  \tau(\mathcal F_{A,\varepsilon}(f))\leq c_4 \sigma(f),\qquad \forall f\in  \mathcal S_p(\R).$$

 Note that the space $ \mathcal S(\C_{p,\varepsilon})$  and its topology are also determined by the seminorms
\begin{equation}\label{9.6}
h\mapsto  \tilde{\tau}_{t,\ell}^{(\vartheta_{p,\varepsilon})} (h):=\sup_{\lambda\in \C_{p,\varepsilon}}  \left| \Big\{ (\lambda+1)^t h(\lambda)\Big\} ^{(\ell)}\right|,
\end{equation}
  where $t$ and $\ell$ are two arbitrary positive integers.   By invoking Lemma \ref{lema1} we have for $r\in \N,$
  \begin{eqnarray*}
  (i\lambda)^r \mathcal F_{A,\varepsilon}(f)(\lambda)  &=&(i\lambda)^r \int_\R \check{f}(x) \Psi_{A, \varepsilon}(\lambda, x) A(x) dx\\
&   =&  \int_\R \check{f}(x)\; \Lambda_{A,\varepsilon}^r  \Psi_{A, \varepsilon}(\lambda, x) A(x) dx\\
&  =& (-1)^r \int_\R (\Lambda_{A,\varepsilon} +2\varepsilon \varrho S)^r\check{f}(x)  \Psi_{A, \varepsilon}(\lambda, x) A(x) dx\\
&=&\int_\R  \Lambda_{A,\varepsilon}^r f(-x) \Psi_{A, \varepsilon}(\lambda, x) A(x) dx\\
&=& \mathcal F_{A,\varepsilon}( \Lambda_{A,\varepsilon}^r\,f)(\lambda),
\end{eqnarray*}
  where $S$ denotes the symmetry $S\!f(x)=f(-x).$ Above we have used $(\Lambda_{A,\varepsilon} +2\varepsilon \varrho S)^r\circ S=(-1)^r \:S\circ \Lambda_{A, \varepsilon}^r.$ 
Thus 
$$
   \Big\{ (i\lambda)^r \mathcal F_{A,\varepsilon}(f)(\lambda)\Big\}  ^{(\ell)} 
   =  \int_\R    \Lambda_{A,\varepsilon}^r  f(x) \;\partial_\lambda^\ell \Psi_{A, \varepsilon}(\lambda, -x)A(x)dx. 
$$
On the one hand, using Theorem \ref{estder}.2 we obtain  
\begin{eqnarray*}
&&\left|    \Big\{ (i\lambda)^r \mathcal F_{A,\varepsilon}(f)(\lambda)\Big\}  ^{(\ell)} \right|\\
 &\leq &c_5
 \int_\R |\Lambda_{A,\varepsilon}^r  f(x)| \;  (|x|+1)^\ell \;  \varphi_{i\sqrt{1-\varepsilon^2}\varrho}(x)\; e^{|\Im \lambda | |x|} A(x)dx\\
&  =&c_5
 \int_{|x| \leq a} |\Lambda_{A,\varepsilon}^r f(x)| \; (|x|+1)^\ell  \;  \varphi_{i\sqrt{1-\varepsilon^2}\varrho}(x)\; e^{|\Im \lambda | |x|} A(x)dx\\
& &+c_5
 \int_{|x|>a} |\Lambda_{A,\varepsilon}^r f(x)|\;  (|x|+1)^\ell  \;  \varphi_{i\sqrt{1-\varepsilon^2}\varrho}(x)\; e^{|\Im \lambda | |x|} A(x)dx\\
& \leq& c_6
 \int_{|x| \leq a} |\Lambda_{A,\varepsilon}^r f(x)| \;  \varphi_0(x)^{-2/p}  \;  (|x|+1)^{2/p+\ell+1}   
\; e^{-2\varrho |x|  }   A(x)dx \\
&&+c_6
 \int_{|x|>a} |\Lambda_{A,\varepsilon}^r f(x)| \;  \varphi_0(x)^{-2/p}  \; 
(|x|+1)^{2/p+\ell+1}   \;  e^{-2\varrho |x|  }   A(x)dx .
\end{eqnarray*}
On the other hand, by mimicking  the proof of \cite[Lemma 4.18]{BX1} we have:
\begin{itemize}
\item[{(i)}] For $|x| \leq a, $
$$|\Lambda_{A,\varepsilon}^r f(x)| \leq c_7   \Big( \sum_{i=0}^r |f^{(i)}(x)| +\sum_{i=0}^{r-1}  |f^{(i)}(-x)| +\sum_{i=0}^r \sum_{m=0} ^{n_r} | f^{(i)}(\xi_m)| \Big),$$
where $\xi_m=\xi_m(x,r)\in ]-|x|, |x|[.$ 
\item[{(ii)}] For $|x|>a,$
$$|\Lambda_{A,\varepsilon}^r f(x)| \leq c'_7  \Big( \sum_{i=0}^r |f^{(i)}(x)| +\sum_{i=0}^{r-1}  |f^{(i)}(-x)|   \Big).$$
 \end{itemize}
The estimate 
$$
\tau (\mathcal F_{A,\varepsilon}(f))\leq c_8 \sum_{\text {finite}} \sigma (f) , \qquad \forall f\in \mathcal S_p(\R)
$$
is now a matter of putting the pieces together. 

{The injectivity of the transform $\mathcal F_{A,\varepsilon}$ on $\mathcal S_p(\R)$ is clear, by the fact that $\mathcal F_{A,\varepsilon}$ is injective on $L^q(\R, A(x)dx)$ for $1\leq q\leq   {2\over {1+  \sqrt{1-\varepsilon^2}}}$ (see  Theorem \ref{injderniermin})  and the fact that  $\mathcal S_p(\R)$ is a dense subspace of $L^q(\R, A(x)dx)$  for all $q<\infty$ so that $p\leq q $.}

This concludes the proof of Lemma \ref{inje}.
 \end{proof}

 \begin{lema}\label{surj} Let  $-1\leq \varepsilon\leq 1$ and $0<p\leq {2\over {1+  \sqrt{1-\varepsilon^2}}}.$ The inverse Fourier transform 
 $\mathcal F_{A,\varepsilon}^{-1}: PW(\C) \longrightarrow \mathcal D(\R)$ given by 
  $$\mathcal F_{A,\varepsilon}^{-1}h(x)={1\over  4} \int_\R h(\lambda) \Psi_{A,\varepsilon}(\lambda, x)\, \Big(1-{{\varepsilon \varrho}\over{i\lambda}}\Big)\;  \pi_\varepsilon(d\lambda)$$
 is continuous for the topologies induced by $ \mathcal S(\C_{p, \varepsilon}) $  and $\mathcal S_p(\R).$   
 \end{lema}
 \begin{proof} Let $f\in \mathcal D(\R)$ and let 
 $h\in PW(\C)$ so that $f=\mathcal F_{A,\varepsilon}^{-1}(h).$ Given a seminorm 
 $\sigma$ on  $\mathcal S_p(\R)$ we should find a continuous seminorm $\tau$ on  $\mathcal S(\C_{p,\varepsilon})$ 
 such that $\sigma(f)\leq c\, \tau(h).$   
 
 Denote by   $g$  the image of $h$ by the inverse Euclidean Fourier 
 transform  $\mathcal F_{\rm euc}^{-1}.$ Making use of the Paley-Wiener Theorem \ref{pw} for $\mathcal F_{A,\varepsilon}$ and 
 the classical Paley-Wiener theorem   for $\mathcal F_{\rm euc},$ we have the following support 
 conservation property: $\supp(f)\subset I_R:=[-R, R] \Leftrightarrow \supp(g)\subset I_R.$ 
 
For  $j\in \N_{\geq 1}  ,$ let $\omega_j \in C^\infty(\R)$ with $\omega_j=0$ on $I_{j-1}$ and $\omega_j=1$
  outside of $I_j.$ Assume that  $\omega_j$ and all its derivatives are bounded, uniformly in $j.$
   We will write $g_j=\omega_j g$, and define $h_j:=\mathcal F_{\rm euc}(g_j)$ and $f_j:={\mathcal F}_{A,\varepsilon}^{-1}(h_j).$ 
   Note that $g_j=g$ outside $I_j.$ Hence, by the above support property, 
    $f_j=f$ outside $I_j.$ We shall estimate the function 
 $$x\mapsto (|x|+1)^s\;  \varphi_0(x)^{-2/p}  \;  | f_j^{(k)}(x)| $$ on $I_{j+1}\setminus I_j$ with $j\in \N_{\geq 1} .$
  Recall that $f_j=f$ on $I_{j+1}\setminus I_j.$  In view of Theorem \ref{estder}  we have 
  \begin{eqnarray*}
 | f_j^{(k)}(x)| &\leq& \int_\R |h_j(\lambda)| \; | \partial_x^k \Psi_{A, \varepsilon}(\lambda, x)| \; \left|1-{{\varepsilon \varrho}\over{i\lambda}}\right|\;  \pi_\varepsilon(d\lambda)\\
& \leq&  \varphi_{i\sqrt{1-\varepsilon^2}\varrho}(x)  \int_\R |h_j(\lambda)|\; ( |\lambda|+1)^k  \; \left|1-{{\varepsilon \varrho}\over{i\lambda}}\right|\;    \pi_\varepsilon(d\lambda),
\end{eqnarray*}
where 
$$\left| 1-{{\varepsilon \varrho}\over{i\lambda}}\right| \,  \pi_\varepsilon(d\lambda) =
{{\sqrt{\lambda^2+\varepsilon^2\varrho^2}\over{\sqrt{\lambda^2-(1-\varepsilon^2)\varrho^2}}}} {1\over {\big |c\big(\sqrt{\lambda^2-(1-\varepsilon^2 )\varrho^2}\big)\big |^2}}
{\bf 1}_{\mathbb{R}\,\backslash\, \left]-\sqrt {1-\varepsilon^2}\varrho,\sqrt {1-\varepsilon^2}\varrho\right[}(\lambda)
d\lambda.$$
By knowing about the asymptotic behavior of the $c$-function (see Section 2),  one comes to 
$$ | f_j^{(k)}(x)|
  \leq  c_1  \varphi_{i\sqrt{1-\varepsilon^2}\varrho}(x) \;  \tau_{t_1, 0}^{(0)}(h_j),$$
 for some integer $t_1>0.$  It follows that 
$$\sup_{x\in  I_{j+1}\setminus I_j}  (|x|+1)^s \;  \varphi_0(x)^{-2/p}  \;   | f_j^{(k)}(x)|
\leq c_2 \; j^{s+1} \; e^{\varrho j  ({2\over p}-1 +\sqrt{1-\varepsilon^2})} \;  \tau_{t_1, 0}^{(0)}(h_j) .$$

 Recall that the two seminorms ${\tau}_{t,\ell}^{(\vartheta_{p,\varepsilon})}$ (see \eqref{9.5}) and $ \tilde{\tau}_{t,\ell}^{(\vartheta_{p,\varepsilon})} $ (see \eqref{9.6})
 are equivalent on $\mathcal S(\C_{p,\varepsilon}).$ Since $h_j=\mathcal F_{\rm euc}(g_j),$ it follows that 
 $$(1+\lambda)^{t_1} h_j(\lambda) 
 =\sum_{\ell=0}^{t_1} 
 \begin{pmatrix} t_1\\ \ell\end{pmatrix} \lambda^\ell \mathcal F_{\rm euc}(g_j)(\lambda).$$
Thus
\begin{eqnarray*}
 \tilde{\tau}_{t_1,0}^{(0)} (h_j)&\leq& 
 \sum_{\ell=0}^{t_1} 
 \begin{pmatrix} t_1\\ \ell\end{pmatrix} \int_\R\;  | g_j^{(\ell)}(y)| \; dy \\
& \leq &
c_3 \sum_{\ell=0}^{t_1}  \sup_{y\in \R} (|y|+1)^2 \; | g_j^{(\ell)}(y)|  \\
&=&
c_3 \sum_{\ell=0}^{t_1} \sup_{w\in \{\pm 1\}} \sup_{y\in \R_+} (y+1)^2\;  | g_j^{(\ell)}(wy)|  .
\end{eqnarray*}
Now one uses the Leibniz rule to compute the derivatives of $g_j=\omega_j g.$ Since 
$\omega_j =0$ on $I_{j-1} $ and is bounded, together with all its derivatives uniformly in $j,$
 then we  have 
  $$\tilde{\tau}_{t_1,0}^{(0)} (h_j)\leq c_4 \sum_{\ell =0}^{t_1}  \sup_{w\in \{\pm 1\}} 
  \sup_{y\in \R_+\setminus I_{j-1}}  (y+1)^2\;  | g^{(\ell)} (wy)| .$$
Hence 
$$ j^{s+1} \; e^{\varrho j  ({2\over p}-1+\sqrt{1-\varepsilon^2} )} \;   \tilde \tau_{t_1, 0}^{(0)}(h_j) 
 \leq c_5 \sum_{\ell =0}^{t_1}  \sup_{w\in \{\pm 1\}} \sup_{y\in \R_+\setminus I_{j-1}}  (y+1)^{s+3} \; 
 e^{\varrho y ({2\over p}-1+\sqrt{1-\varepsilon^2} )} \;  | g^{(\ell)}(wy)| .$$

Recall that $g(x)=\mathcal F_{\rm euc}^{-1} (h)(x),$ where $\mathcal F_{\rm euc}$ is the Euclidean Fourier 
transform and  $h\in PW(\C) .$  By Cauchy's integral theorem, it is known that 
$$ p(u)\;e^{\alpha u} \;g^{(\ell)}(u)=\text{cst}\int_\R p(i\partial_\lambda) \Big\{ (i\lambda-\alpha)^\ell 
h(\lambda+i\alpha)\Big\} e^{i\lambda u} d\lambda, $$ for any polynomial $p\in \R[u].$ Hence,   
\begin{eqnarray*}
 &&\sum_{\ell=0}^{t_1}  \sup_{w\in \{\pm 1\}} \sup_{y\in \R_+\setminus I_{j-1}}  (y+1)^{s +3} \;
 e^{\varrho y ({2\over p}-1+\sqrt{1-\varepsilon^2}  )}  \; | g^{(\ell)}(wy)| \\
&\leq& c_6 \sum_{r=0}^{s+3} \; \sup_{|\Im \lambda| \leq \varrho\, \vartheta_{p,\varepsilon} } 
(|   \lambda|+1) ^{t_2}  \;
|h^{(r)}(\lambda)|   \\
&=& c_6 \sum_{r=0}^{s+3}    {\tau}_{t_2,r}^{(\vartheta_{p,\varepsilon})} (h),
\end{eqnarray*}
for some integer $t_2>0.$

It remains for us to estimate the function 
$$x\mapsto  (|x|+1)^s\; \varphi_0(x)^{-2/p} \; | f^{(k)}(x)| $$ on $I_{1}=[-1, 1].$ First, it is not hard to prove that for $|x|\leq 1,$  there is a  positive constant $c$ and an integer $m_k\geq 1$ such that 
\begin{equation} 
\Big| {{\partial ^k}\over{\partial x^k}} \Psi_{A,\varepsilon}(\lambda, x)\Big| \leq c 
 {{(|\lambda|+1)^{m_k} }\over {|i\lambda-\varepsilon \varrho|} } \varphi_0(x) 
 \end{equation} 
 for $\lambda\in \R$ such that $|\lambda|\geq \sqrt{1-\varepsilon^2}\varrho.$  Now, arguing as above,  we have
 $$ | f^{(k)}(x)| \leq  c_1  \varphi_{0}(x)  \int_\R |h(\lambda)|\;  {{(|\lambda|+1)^{m_k} } \over {|i\lambda-\varepsilon \varrho|}}\;\left| 1-{{\varepsilon \varrho}\over {i\lambda}}\right|  \;  \pi_\varepsilon(d\lambda).$$
Since $I_1$ is compact, it follows that 
\begin{eqnarray*}
 \sup_{x\in I_1} (|x|+1)^s\; \varphi_0(x)^{-2/p} \;  | f^{(k)}(x)| &\leq& c_2  \int_\R |h(\lambda)|\;  {{(|\lambda|+1)^{m_k} } \over {|i\lambda-\varepsilon \varrho|}} \;  \left| 1-{{\varepsilon \varrho}\over {i\lambda}}\right|  \;  \pi_\varepsilon(d\lambda)\\
& \leq& c_3 \tau_{t, 0}^{(0)} (h),
\end{eqnarray*}
 for some integer $t>0.$ 
 
 This finishes the proof of Lemma \ref{surj}.
 \end{proof}
In summary, we have proved:
 \begin{thm}\label{+-1} Let $-1\leq \varepsilon\leq 1$ and  $0<p\leq {2\over {1+  \sqrt{1-\varepsilon^2}}}  .$
  Then the Fourier transform $\mathcal F_{A,\varepsilon}$ is a topological isomorphism between $\mathcal S_p(\R)$ and
 $ \mathcal S(\C_{p,\varepsilon})$.
 \end{thm}

\section{Pointwise multipliers}\label{sec9}

For  $-1\leq \varepsilon\leq 1$ and   $0<p\leq {2\over {1+  \sqrt{1-\varepsilon^2}}},$  denote by $\mathcal S'_p(\R)$ \index{$\mathcal S_p'(\R)$} and by  $\mathcal S'(\C_{p,\varepsilon})$ \index{$\mathcal S'(\C_{p,\varepsilon})$} the topological dual spaces of 
$\mathcal S_p(\R)$ and $\mathcal S(\C_{p,\varepsilon})$, respectively. 
 
Let $f$ be a Lebesgue measurable function on $\R$ such that 
$$ \int_\R |f(x) |  \varphi_0(x)^{2/p} (|x|+1)^{-\ell} A(x)dx<\infty$$ for some $\ell \in \N.$ Then the functional 
$T_f$ defined on  $\mathcal S_p(\R)$ by 
 $$\langle T_f,\phi \rangle =\int_\R f(x)  {\phi(-x)} A(x)dx,\qquad \phi \in \mathcal S_p(\R)$$
  is in $\mathcal S'_p(\R).$ Indeed, 
$$ | \langle T_f, \phi\rangle| \leq \sigma_{\ell,0}^{(p)} (\phi)\int_\R  |f(x)| \varphi_0(x)^{2/p} 
 (|x|+1)^{-\ell} A(x) dx<\infty.$$
Further, since $p \leq {2\over {1+  \sqrt{1-\varepsilon^2}}}\leq 2,$ the Schwartz space $\mathcal S_p(\R)$ can be seen as a subspace of 
 $ \mathcal S'_p(\R)$   by 
 identifying $f\in \mathcal S_p(\R)$ with $T_f\in  \mathcal S_p'(\R).$

Now let $h$ be a measurable function on $\R$ such that 
$$\int_{\R} |h(\lambda)| (|\lambda|+1)^{-\ell} \left| 1-{{\varepsilon \varrho}\over {i\lambda}}\right|  \pi_\varepsilon(d\lambda)   <\infty$$
for some $\ell \in \N.$ Here $ \pi_\varepsilon(d\lambda)$  denotes   the Plancherel measure  \eqref{plancherel},
$$
 \pi_\varepsilon(d\lambda)= 
\frac{|\lambda|}{\sqrt{\lambda^2-(1-\varepsilon^2 )\varrho^2}\;
\big |c\big(\sqrt{\lambda^2-(1-\varepsilon^2 )\varrho^2}\big)\big |^2}
{\bf 1}_{\mathbb{R}\,\backslash\, \left]-\sqrt {1-\varepsilon^2}\varrho,\sqrt {1-\varepsilon^2}\varrho\right[}(\lambda)
d\lambda,
$$
where $c$ is the Harish-Chandra's function associated with the  operator $\Delta $ (see Section \ref{sec2}).  Then the functional $\mathcal T_h$ defined on $\mathcal S(\C_{p,\varepsilon}) $ by 
$$\langle \mathcal T_h, \psi\rangle =\int_{\R }  h(\lambda)  \psi(\lambda)  \Big(1-{{ \varepsilon \varrho  }\over{i\lambda  }}\Big) \pi_\varepsilon(d\lambda),
\qquad \psi\in \mathcal S(\C_{p,\varepsilon})$$ is in the dual space $\mathcal S'(\C_{p,\varepsilon}).$
In fact, 
 $$|\langle \mathcal T_h, \psi\rangle |\leq c\, \tau_{0,\ell}^{(0)}(\psi) \int_{\R} |h(\lambda)| 
 (|\lambda|+1)^{-\ell} \left|1-{{\varepsilon \varrho }\over{i\lambda}}\right|  \pi_\varepsilon(d\lambda)<\infty.$$
Moreover, since $|c(\mu)|^{-2} \sim |\mu|^{2\alpha+1}$  for $|\mu|$ large (with $\alpha>-1/2$) and
   $$|c(\mu)|^{-2}\sim
   \begin{cases}
       |\mu |^2 & \text{for}  \:|\mu| \lleq 1 \: \text{and} \:\varrho>0,\\
    |\mu|^{2\alpha+1} & \text{for}\:  |\mu| \lleq 1  \:\text{and}\: \varrho=0,
   \end{cases}
$$ 
it follows that the Schwartz space $\mathcal S(\C_{p,\varepsilon})$ can be identified with a subspace of $\mathcal S'(\C_{p,\varepsilon}).$

For $T$ in $\mathcal S'_p(\R),$ we define the distributional Fourier 
transform $\mathcal F_{A,\varepsilon}(T)$ of $T$ on  $\mathcal S (\C_{p,\varepsilon})=\mathcal F_{A,\varepsilon}(\mathcal S_p(\R))$  by
\index{$\mathcal F_{A,\varepsilon}(T)$}
\begin{equation}\label{dis}
 \langle \mathcal F_{A,\varepsilon}(T),\mathcal F_{A,\varepsilon}( \phi)\rangle = \langle  T,\phi \rangle ,\qquad \phi\in \mathcal S_p(\R).
\end{equation}
That is,
\begin{equation*} 
 \langle \mathcal F_{A,\varepsilon}(T), \psi  \rangle = \langle  T,\mathcal F^{-1}_{A,\varepsilon} (\psi) \rangle ,\qquad \psi \in 
 \mathcal S (\C_{p,\varepsilon}).
\end{equation*}
This definition is an extension of the Fourier transform on $\mathcal S_p(\R).$ Indeed, 
let $f \in \mathcal S_p(\R)$ with $0<p \leq {2\over {1+  \sqrt{1-\varepsilon^2}}}\leq 2.$ Applying Fubini's theorem, then, for  every $\phi\in \mathcal S_p(\R),$  we have 
\begin{eqnarray*}
\langle \mathcal T_{\mathcal F_{A,\varepsilon}(f)}, \mathcal F_{A,\varepsilon}(\phi)\rangle &=&
\int_{\R }  \mathcal F_{A,\varepsilon}(f)(\lambda)  \mathcal F_{A,\varepsilon}(  \phi)( \lambda) \Big(1-{{ \varepsilon \varrho }\over{i\lambda  }}\Big) \pi_\varepsilon(d\lambda)\\
& =&\int_{\R } f(x) \Big\{ \int_{\R } \mathcal F_{A,\varepsilon}(  \phi)(\lambda)  \Psi_{A,\varepsilon}(\lambda, -x) \Big(1-{{ \varepsilon \varrho }\over{i\lambda  }}\Big) \pi_\varepsilon(d\lambda) \Big\} A(x)dx\\
&=&\int_\R f(x) {\phi(-x)} A(x) dx \\
&=&\langle T_f, \phi\rangle. 
\end{eqnarray*} 
Hence $\mathcal F_{A,\varepsilon}(T_f)=\mathcal T_{\mathcal F_{A,\varepsilon}(f)}.$  

A function $\psi$ defined on $\C_{p,\varepsilon}$  is called a pointwise multiplier of $\mathcal S(\C_{p,\varepsilon})$ 
 if the mapping $\phi\mapsto  \psi\phi$ is continuous from  $\mathcal S(\C_{p,\varepsilon})$ into itself.
 The following statement comes from \cite[Proposition 3.2]{BBM}, with changes appropriate to our setting.
 \begin{lema}\label{bbm} 
Let $\psi$ be a function defined on $\C_{p,\varepsilon}.$ Then,    $\psi$    
is a pointwise multiplier of $\mathcal S(\C_{p,\varepsilon})$ if and only if $\psi$ satisfies the following three conditions:
 \begin{itemize}
 \item[{(i)}] $\psi$ is holomorphic in the interior of $\C_{p,\varepsilon}.$
 \item[{(ii)}] For every $t\in \N,$ the derivatives $\psi^{(t)}$ extend continuously to $\C_{p,\varepsilon}.$
 \item[{(iii)}] For every $t\in \N,$ there exists $n_t\in \N,$ such that \index{$n_t$}
\begin{equation}\label{modi2}
\sup_{\lambda\in \C_{p,\varepsilon}} (|\lambda|+1)^{-n_t} | \psi^{(t)} (\lambda)| <\infty.
\end{equation}
 \end{itemize}
 \end{lema}

 \begin{thm}\label{pwm} Suppose that  $0<p \leq {2\over {1+  \sqrt{1-\varepsilon^2}}}$  whenever $\varrho=0,$ and $ {2\over {2+  \sqrt{1-\varepsilon^2}}}\leq p \leq {2\over {1+  \sqrt{1-\varepsilon^2}}}$  whenever $\varrho>0.$
 If $T\in \mathcal S_p'(\R)$ such that $\psi:=\mathcal F_{A,\varepsilon}(T)$  is 
a pointwise multiplier of  $\mathcal S(\C_{p,\varepsilon}),$ then for any $s\in \N$  
there exist $\ell\in \N$ and continuous functions $f_m$ defined on $\R,$ $m=0,1,\ldots  ,\ell ,$  such that 
$$T=\sum_{m=0}^\ell \Lambda_{A,\varepsilon}^m \,f_m  $$ and, for every such $m,$ 
\begin{equation}\label{modi}
\sup_{x\in \R} \,(|x|+1)^s \,\varphi_0(x)^{-{2\over p}+\sqrt{1-\varepsilon^2}} \,| f_m (x)|<\infty.
\end{equation} Here 
$\Lambda_{A,\varepsilon}$ is the differential-reflection operator \eqref{Op}.
 \end{thm}
 \begin{proof} It is assumed that  $\psi= \mathcal F_{A,\varepsilon}(T)$  is   a pointwise multiplier of $\mathcal S(\C_{p,\varepsilon}).$  Then  by Lemma \ref{bbm}, for all $t\in \N$ there is an integer $n_t \in \N$  such that 
\begin{equation}\label{F}
\sup_{\lambda\in \C_{p,\varepsilon}} (|\lambda|+1)^{-n_t} | \psi^{(t)} (\lambda)| <\infty.
\end{equation}
 Fix $s\in \N$ and consider an integer $\ell  $ that will be later specified. Define the function 
 $\kappaup$ on $\C_{p,\varepsilon}$ by 
$$  \kappaup(\lambda )=(i\lambda + \varrho+1)^{-\ell} \psi(\lambda). $$ 
In view of our assumption on $p,$ the function  $\kappaup$ satisfies the first and the second conditions in the definition of the space $\mathcal S(\C_{p,\varepsilon}).$
Further,  since $|\Psi_{A,\varepsilon}(\lambda, x)|\leq \sqrt 2$ for all $\lambda\in \R,$ we have 
$$ |\mathcal F_{A,\varepsilon}^{-1}(\kappaup)(x)|:=  \Big|c\int_\R \kappaup(\lambda) \Psi_{A,\varepsilon}(\lambda, x)\, \Big(1-{{\varepsilon \varrho}\over{i\lambda}}\Big)\;  \pi_\varepsilon(d\lambda)\Big|
\leq  c_1  \int_{\R } |\kappaup(\lambda)|\,\left| 1-{{\varepsilon \varrho}\over{i\lambda}}\right| 
\,  \pi_\varepsilon(d\lambda),$$
where 
$$\left| 1-{{\varepsilon \varrho}\over{i\lambda}}\right| \,  \pi_\varepsilon(d\lambda) =
{{\sqrt{\lambda^2+\varepsilon^2\varrho^2}\over{\sqrt{\lambda^2-(1-\varepsilon^2)\varrho^2}}}} {1\over {\big |c\big(\sqrt{\lambda^2-(1-\varepsilon^2 )\varrho^2}\big)\big |^2}}
{\bf 1}_{\mathbb{R}\,\backslash\,   ]-\sqrt {1-\varepsilon^2}\varrho,\sqrt {1-\varepsilon^2}\varrho  [}(\lambda)
d\lambda.$$
 Thus, in view of the estimate \eqref{F} and the behavior of $|c(\mu)|^{-2}$ for small and large $|\mu|,$ it follows that  $\mathcal F^{-1}_{A,\varepsilon}(\kappaup)(x)$ exists for all $x\in \R$ provided that   $\ell >n_0+2\alpha+2.$ Moreover, for all $\phi\in \mathcal S_p(\R),$ Fubini's theorem leads to 
 \begin{eqnarray*}
&& \int_\R \phi(-x) \mathcal F_{A,\varepsilon}^{-1}(\kappaup)(x) A(x) dx\\
 &=&c_1 \int_\R \phi(-x)\, \Big( \int_\R \kappaup (\lambda) \Psi_{A,\varepsilon}(\lambda, x)\, \Big(1-{{\varepsilon \varrho}\over{i\lambda}}\Big)\;  \pi_\varepsilon(d\lambda) \Big) \,A(x) dx\\
 &=&c_1 \int_\R  \kappaup (\lambda) \, \Big(   \int_\R   \phi(-x) \Psi_{A,\varepsilon}(\lambda, x)\,  A(x) dx \, \Big) \Big(1-{{\varepsilon \varrho}\over{i\lambda}}\Big)\;  \pi_\varepsilon(d\lambda)\\
 &=& c_1 \int_\R  \kappaup (\lambda)  \mathcal F_{A, \varepsilon} (\phi) (\lambda) \Big(1-{{\varepsilon \varrho}\over{i\lambda}}\Big)\;  \pi_\varepsilon(d\lambda). 
  \end{eqnarray*}
 It follows that the inverse Fourier transform $\mathcal F_{A,\varepsilon}^{-1}(\kappaup)$ of $\kappaup$ as an element of 
 $\mathcal S'(\C_{p,\varepsilon})$
concurs  with the classical Fourier transform of $\kappaup.$
Further 
$$
T= \mathcal F_{A,\varepsilon}^{-1}( (i\lambda+\varrho+1)^{\ell} \kappaup)
= \sum_{m=0}^\ell  \begin{pmatrix}
 \ell \cr m \end{pmatrix}
(\varrho+1)^{\ell-m}    \Lambda_{A,\varepsilon }^m  \mathcal F^{-1}_{A,\varepsilon}(\kappaup)
  :=  \sum_{m=0}^\ell   \Lambda_{A,\varepsilon }^m   {   f_m}.
$$
  It remains for us to show that, given $s\in \N,$ the functions $f_m$ satisfy  \eqref{modi},  provided that $\ell$ is   large enough. To do so, we will use a similar approach to that  in  the proof of Lemma \ref{surj}.
 
Denote by  $\xiup:=\mathcal F^{-1}_{A,\varepsilon}(\kappaup)$ and by  $g:=\mathcal F_{\rm euc}^{-1}(\kappaup),$
   where $\mathcal F_{\rm euc}$ denotes the 
Euclidean Fourier transform. Observe that if $\ell$ is large enough, then $g$ is well defined. 
For $j\in \N_{\geq 1},$ let $\omega_j \in C^\infty(\R)$ such that   $\omega_j=0$
 on $I_{j-1}:=[-(j-1), j-1]$  and $ \omega_j=1$ outside of $I_j.$ We shall assume that $\omega_j$
together with all its derivatives are bounded, uniformly in $j.$ 
  
We set $g_j :=\omega_j g$, and define 
 $\kappaup_j:=\mathcal F_{\rm euc}(g_j)$ and $\xiup_j=\mathcal F^{-1}_{A,\varepsilon}(\kappaup_j).$   Since $\omega_j=1$ outside of $I_j,$
  it follows that $g_j-g=0$ outside of $I_j.$ That is  $\supp(g_j-g)\subset I_j.$
 Using the support conservation property from the proof of Lemma \ref{surj}, we deduce that 
 $\xiup$ may differ from $\xiup_j$ only inside  $I_j.$ Now, we will estimate the function 
\begin{equation}\label{ref}
x\mapsto \,(|x|+1)^s \,\varphi_0(x)^{-{2\over p}+\sqrt{1-\varepsilon^2}} \, \xiup(x),
\end{equation} first on $I_1$ 
 and next on $I_{j+1}\setminus I_j$ for $j\in \N_{\geq 1}.$
 
We claim   that $|\Psi_{A,\varepsilon}(\lambda, x)|\leq c_2(|\lambda|+1)\varphi_0(x)$ for  $\lambda\in \R$ such that $|\lambda|\geq \sqrt{1-\varepsilon^2}\varrho.$ Indeed, as $\lambda\in \R$ such that $|\lambda|\geq \sqrt{1-\varepsilon^2}\varrho,$ it follows that $\mu_{\varepsilon}\in \R.$  Thus, the claim follows from the superposition \eqref{2.2} of $\Psi_{A,\varepsilon}(\lambda, x)$ and the facts that $|\varphi_{\mu_\varepsilon}(x)| \leq \varphi_0(x)$ and $|\varphi_{\mu_\varepsilon}'(x)| \leq c\, (\mu_\varepsilon^2+\varrho^2)\varphi_0(x)$ (see Lemma \ref{BX}.4 and  \ref{BX}.5).

From the claim above we have  
\begin{eqnarray*}
|\xiup  (x)|& \leq& c_3 \int_{\R } |\kappaup(\lambda)|\, |   \Psi_{A,\varepsilon}(\lambda,x)|\,\left| 1-{{\varepsilon \varrho}\over {i\lambda}}\right| \,   \pi_\varepsilon(d\lambda) \\
&\leq & c_4 \,\varphi_0(x)  \int_{\R } |\kappaup(\lambda)| \,(| \lambda|+1)  \, \left| 1-{{\varepsilon \varrho}\over {i\lambda}}\right| \,  \pi_\varepsilon (d\lambda). 
\end{eqnarray*}
Since $I_1$ is compact, we deduce that for every $s\in \N$   
 $$\sup_{x\in I_1}\,(|x|+1)^s \,\varphi_0(x)^{-{2\over p}+\sqrt{1-\varepsilon^2}} \,  |\xiup (x)|<\infty$$ 
 whenever $\ell>n_0 +2\alpha+3.$ Here the parameter $n_0$ comes from \eqref{F}. 
Now we consider  the  estimate of the function \eqref{ref} on  $I_{j+1}\setminus I_j$ for $j\in \N_{\geq 1}.$
 Recall that $\xiup=\xiup_j$ outside of $I_j.$ 

Arguing as above, we obtain
$$|\xiup_j(x)| \leq c_5\, \varphi_0(x) \,\sup_{\lambda\in \R\setminus  ]-\sqrt{1-\varepsilon^2}\varrho,\sqrt{1-\varepsilon^2}\varrho [}|(\lambda+1)^{t_1} \kappaup_j(\lambda)|$$
 for some integer $t_1>2\alpha+3.$ It follows that 
$$\sup_{x\in I_{j+1}\setminus I_j} \,(|x|+1)^s \,\varphi_0(x)^{-{2\over p}+\sqrt{1-\varepsilon^2}} \, |\xiup_j(x)| \leq c_6
j^{s }  e^{\big({2\over p}-{1} -\sqrt{1-\varepsilon^2}\big) \varrho j} \sup_{\lambda\in \R\setminus  ]-\sqrt{1-\varepsilon^2}\varrho,\sqrt{1-\varepsilon^2}\varrho [}|(\lambda+1)^{t_1} \kappaup_j(\lambda)|.$$
Since $\kappaup_j=\mathcal F_{\rm euc}(g_j)$ 
with $g_j=\omega_j g,$ we claim  that 
\begin{equation}\label{claimc7}
|(\lambda+1)^{t_1} \kappaup_j(\lambda)| \leq c_7 
 \sum_{q=0}^{t_1}  \sup_{w\in\{\pm 1\}}  \sup_{x\in \R^+\setminus I_{j-1}}(x+1)^2 |g^{(q)}(wx)| .
 \end{equation}
Indeed, on the one hand 
\begin{eqnarray}\label{limit-1}
(\lambda+1)^{t_1} \kappaup_j(\lambda) &=&\sum_{r=0}^{t_1} c_r\int_\R g_j(x)\, \partial_x^r e^{i\lambda x} \,dx \nonumber \\
&=&\sum_{r=0}^{t_1} c_r\int_\R \omega_j(x) g(x)\, \partial_x^r e^{i\lambda x} \,dx.
\end{eqnarray}
On the other hand, we have 
\begin{equation}\label{limit0}
(\omega_j g)^{(r)}(x)= \sum_{q=0}^r c_q g^{(q)}(x) \omega_j^{(r-q)}(x)\rightarrow 0 \quad\text{as}\quad  |x|\rightarrow +\infty.
\end{equation}
In fact, starting from $g=\mathcal F_{\rm euc}^{-1}(\kappaup),$ we obtain 
\begin{equation}\label{gq}
g^{(q)}(x)=c \int_\R \kappaup(\lambda) (i\lambda)^q e^{i\lambda x} d\lambda.
\end{equation}
 Thus, if $\ell>n_0 +t_1+1$  then by Riemann-Lebesgue lemma for the Euclidean Fourier transform, 
 $g^{(q)}(x)\rightarrow 0$
  as $|x|\rightarrow \infty.$ Thus \eqref{limit0} holds true. 
Now, in view of \eqref{limit0} we may rewrite \eqref{limit-1} as 
$$(\lambda+1)^{t_1} \kappaup_j(\lambda) = \sum_{r=0}^{t_1}  \sum_{q=0}^r c_{q,r} \int_\R g^{(q)}(x) \omega_j^{(r-q)}(x) e^{i\lambda x} dx.$$
Recall that the function
 $ \omega_j$ vanishes on $I_{j-1}$ and is bounded, together with all its derivatives, uniformly in $j.$
Therefore, 
\begin{eqnarray*}
|(\lambda+1)^{t_1} \kappaup_j(\lambda) | &\leq& c \sum_{q=0}^{t_1} \int_{\R\setminus I_{j-1}} |g^{(q)}(x)| dx\\
&\leq &c \sum_{q=0}^{t_1}   \sup_{x\in \R\setminus I_{j-1}}(|x|+1)^2 |g^{(q)}(x)|.
\end{eqnarray*}
This finishes the proof of our claim \eqref{claimc7}.

It follows that 
\begin{align*}
& j^{s}  e^{\big({2\over p}-{1} -\sqrt{1-\varepsilon^2}\big) \varrho j}  \sup_{\lambda\in \R\setminus  ]-\sqrt{1-\varepsilon^2}\varrho,\sqrt{1-\varepsilon^2}\varrho [}
|(\lambda+1)^{t_1} \kappaup_j(\lambda)|\\
&\qquad \quad  \leq c_{7} \sum_{q=0}^{t_1}  \sup_{w\in\{\pm 1\}}  
\sup_{x\in \R^+\setminus I_{j-1}}(x+1)^{s+2} e^{\big({2\over p}-{1} -\sqrt{1-\varepsilon^2}\big) \varrho x}    |g^{(q)}(wx)| .
\end{align*}
Next we shall prove that the right hand is finite. Assume first  that $\varrho=0.$ By \eqref{gq} we have 
\begin{equation}\label{cab}(x+1)^{s+2} g^{(q)}(wx)=\sum_{r=0}^{s+2} c_{q,r} \int_\R \kappaup(\lambda)\lambda^q \, \partial_\lambda^r e^{i\lambda w x}\, d\lambda.
\end{equation} We claim that 
\begin{equation}\label{claimc8}
(\kappaup(\lambda) \lambda^q)^{(r)}\rightarrow 0 \qquad \text{as}\qquad  |\lambda|\rightarrow +\infty
\end{equation}
provided that $\ell$ is large enough. Indeed, this claim  follows immediately from the fact that 
\begin{eqnarray}\label{1012}
(\kappaup(\lambda) \lambda^q)^{(r)} &=&\sum_{a=0}^r c_a  \lambda^{q-r+a}\kappaup ^{(a)}(\lambda)\qquad\qquad   {(\text{with }   r-a \leq q)}\nonumber\\
&=&\sum_{a=0}^r \sum_{b=0}^a c_{a,b} \lambda^{q-r+a} (i\lambda+1)^{-\ell-a+b} \psi^{(b)}(\lambda),
\end{eqnarray}
together with the fact that $\psi $ satisfies \eqref{F}. Thus, by \eqref{claimc8} we may rewrite \eqref{cab} as 
\begin{equation}\label{1011}
(x+1)^{s+2} g^{(q)}(wx) =\sum_{r=0}^{s+2} c_{q,r}' \int_\R (\kappaup(\lambda) \lambda^q)^{(r)}   \,e^{i\lambda w x} d\lambda.
\end{equation}
Using again the fact that  $\psi$ satisfies \eqref{F} together with the double sum \eqref{1012},  it follows  from \eqref{1011} that for $\varrho=0$
$$\sup_{w\in\{\pm 1\}}  
\sup_{x\in \R^+\setminus I_{j-1}}(x+1)^{s+2}       |g^{(q)}(wx)| <\infty$$ provided that $\ell$ is large enough. 

Now assume that $\varrho>0.$ 
Since $g=\mathcal F_{\rm euc}^{-1}(\kappaup)$ and $\kappaup$ is holomorphic in the interior of $\C_{p,\varepsilon},$ Cauchy's integral theorem gives 
$$ p(u)\;e^{\alpha u} \;g^{(q)}(u)=\text{cst}\int_\R p(i\partial_\lambda) \Big\{ (i\lambda-\alpha)^q 
\kappaup(\lambda+i\alpha)\Big\} e^{i\lambda u} d\lambda, $$ with $p(x)=(x+1)^{s+2}$ and  
$\alpha=\big({2\over p}-{1} -\sqrt{1-\varepsilon^2}\big)\varrho. $ The same argument as above implies that 
$$ \sup_{w\in\{\pm 1\}}  \sup_{x\in \R^+\setminus I_{j-1}}(x+1)^{s+2} e^{\big({2\over p}-{1} -\sqrt{1-\varepsilon^2}\big) \varrho x} |g^{(q)}(wx)|<\infty$$
provided that $\ell$ is large enough.

Putting the pieces together we conclude that 
$$\sup_{x\in I_{j+1}\setminus I_j} \,(|x|+1)^s \,\varphi_0(x)^{-{2\over p}+\sqrt{1-\varepsilon^2}} \,   |\xiup_j(x)| <\infty$$
for $\ell$ large enough. 
\end{proof}

\end{document}